\documentclass{article}
\usepackage[T1]{fontenc}
\usepackage{amsthm, stmaryrd, fullpage}
\usepackage{amsmath}
\usepackage{amssymb}
\usepackage{amsfonts}
\usepackage{eucal}
\usepackage[all]{xy}
\usepackage{pslatex}
\usepackage{hyperref}

\newtheorem{prop}{Proposition}[subsection]
\newtheorem{theo}[prop]{Theorem}
\newtheorem{coro}[prop]{Corollary}
\newtheorem{lemm}[prop]{Lemma}
\newtheorem{lemm2}{Lemma}[prop]

\theoremstyle{definition}

\newtheorem{conj*}{Conjecture}
\newtheorem{empt}[prop]{}
\newtheorem{defi}[prop]{Definition}
\newtheorem*{nota*}{Notation}

\theoremstyle{remark}
\newtheorem{rema}[prop]{Remarks}
\newtheorem{exam}[prop]{Examples}
\newtheorem{nota}[prop]{Notation}

\numberwithin{equation}{prop}

\newcommand{\riso}{ \overset{\sim}{\longrightarrow}\, }
\newcommand{\liso}{ \overset{\sim}{\longleftarrow}\, }

\newcommand{\Spec}{\mathrm{Spec}\,}
\newcommand{\Spf}{\mathrm{Spf}\,}

\renewcommand{\sp}{\mathrm{sp}}

\newcommand{\gr}{\mathrm{gr}}

\renewcommand{\AA}{{\mathcal{A}}}

\newcommand{\RR}{{\mathcal{R}}}
\newcommand{\FF}{{\mathcal{F}}}
\newcommand{\B}{{\mathcal{B}}}

\newcommand{\E}{{\mathcal{E}}}
\newcommand{\G}{{\mathcal{G}}}
\renewcommand{\H}{{\mathcal{H}}}
\newcommand{\M}{{\mathcal{M}}}

\newcommand{\D}{{\mathcal{D}}}
\newcommand{\I}{{\mathcal{I}}}
\newcommand{\PP}{{\mathcal{P}}}
\newcommand{\QQ}{{\mathcal{Q}}}
\renewcommand{\O}{{\mathcal{O}}}

\newcommand{\V}{\mathcal{V}}
\renewcommand{\S}{\mathcal{S}}
\newcommand{\T}{{\mathcal{T}}}
\newcommand{\Y}{\mathcal{Y}}
\newcommand{\ZZ}{\mathcal{Z}}
\newcommand{\X}{\mathfrak{X}}

\newcommand{\HH}{\mathfrak{H}}

\newcommand{\U}{\mathfrak{U}}

\newcommand{\Xan}{\X_K}
\newcommand{\Xlogan}{\X_K^{\#}}

\newcommand{\A}{\mathbb{A}}

\newcommand{\DD}{\mathbb{D}}
\renewcommand{\L}{\mathbb{L}}
\newcommand{\R}{\mathbb{R}}
\newcommand{\Q}{\mathbb{Q}}
\newcommand{\Z}{\mathbb{Z}}
\newcommand{\N}{\mathbb{N}}

\newcommand{\hdag}{  \phantom{}{^{\dag} }    }

\begin{document}
\title{Overholonomicity of overconvergent $F$-isocrystals over smooth varieties}
\author{Daniel Caro\footnote{The first author was supported by the european network TMR
\textit{Arithmetic Algebraic Geometry}},
Nobuo Tsuzuki\footnote{The second author was supported by Japan Society for the Promotion
of Science and Inamori Foundation}}

\date{March 14, 2008}

\maketitle

\begin{abstract}
  We prove the overholonomicity of overconvergent $F$-isocrystals over smooth varieties.
  This implies that the notions of overholonomicity and devissability in overconvergent $F$-isocrystals
  are equivalent. Then the overholonomicity is stable under tensor products. So, the overholonomicity
  gives a $p$-adic cohomology stable under Grothendieck's cohomological operations.
\end{abstract}

\tableofcontents

\section*{Introduction}


Let $\V$ be a complete discrete valuation ring of characteristic $0$, with
perfect residue field $k$ of characteristic $p>0$ and fractions field $K$.
In order to define a good category of $p$-adic coefficients over
$k$-varieties (i.e., separated schemes of finite type over $\Spec k$)
stable under cohomological operations, Berthelot introduced the
notion of arithmetic $\D$-modules and their cohomological operations
(see \cite{Be0}, \cite{Beintro2}, \cite{Be1}, \cite{Be2}). These
arithmetic $\D$-modules over $k$-varieties correspond to an
arithmetic analogue of the classical theory of $\D$-modules over
complex varieties. Also, he defined holonomic $F$-complexes of
arithmetic $\D$-modules and conjectured its stability under the following
Grothendieck's five operations: direct images (to be precise,
morphisms should be proper at the level of formal $\V$-schemes),
extraordinary direct images, inverse images, extraordinary inverse images,
tensor products (see \cite[5.3.6]{Beintro2}).
We checked that the conjecture on the stability of holonomicity under inverse images
implies the others ones (see \cite{caro_surholonome}).

In order to avoid these conjectures and to get a category of
$F$-complexes of arithmetic $\D$-modules which satisfies these
stability conditions, the first step was to introduced the notion of
overcoherence as follows: a coherent $F$-complex of arithmetic
$\D$-modules is overcoherent (in fact, the `$F$', i.e. the Frobenius
structure, is not necessary) if its coherence is stable under
extraordinary inverse image (see \cite{caro_surcoherent} for the definition
and \cite{caro_caract-surcoh} for this characterization). We checked that this notion of
overcoherence is stable under extraordinary inverse image, direct
image (by a proper morphism at the level of formal $\V$-schemes) and
local cohomological functors. This stability allows for instance to
define canonically overcoherent arithmetic $\D$-modules over
$k$-varieties (otherwise, we work on formal $\V$-schemes). To
improve the stability properties, we defined the category of
overholonomic $F$-complexes over $k$-varieties which is, roughly
speaking, the smallest subcategory of overcoherent $F$-complexes
such that it is moreover stable by dual functors (more precisely,
see the definition \cite[3.1]{caro_surholonome}).
We got the
stability of overholonomicity by direct images, extraordinary direct images,
extraordinary inverse images and inverse images.
Moreover, it is already known that this category of $p$-adic coefficients is not zero since
it contains unit-root overconvergent $F$-isocrystals (see
\cite{caro_surholonome}) and in particular the constant coefficient
associated to a $k$-variety (i.e., which gives for example the
corresponding Weil's zeta functions). Because an overholonomic
arithmetic $F$-$\D$-module is holonomic (which is not obvious),
these gave new examples of holonomicity. This was checked by descent
of the overholonomicity property (this descent is technically
possible thanks to its stability) using de Jong's desingularization
theorem. Now, it remains to check the stability of overholonomicity by (internal or
external) tensor products.

The second step was to construct an equivalence between the category
of overconvergent $F$-isocrystal over a smooth $k$-variety $Y$
(which is the category of $p$-adic coefficients associated to Berthelot's rigid
cohomology: see \cite{LeStum-book-isoc}) and the category of
overcoherent $F$-isocrystals on $Y$, where this last one is a
subcategory of arithmetic $F$-$\D$-modules over $Y$ (see
\cite{caro_devissge_surcoh} and \cite{caro-2006-surcoh-surcv} for
the general case). Next, we got from this equivalence the notion of
$F$-complexes of arithmetic $\D$-modules devissable in
overconvergent $F$-isocrystals. We proved first that overholonomic
(see \cite{caro_devissge_surcoh}) and next overcoherent (see
\cite{caro-2006-surcoh-surcv}) $F$-complexes of arithmetic
$\D$-modules are devissable in overcoherent $F$-isocrystals. Since
overcoherent $F$-isocrystals are stable under tensor products, we
established that $F$-complexes devissable in overcoherent
$F$-isocrystals are also stable under tensor products (see
\cite{caro-stab-prod-tens}).

The third step is to prove that the notions of overcoherence,
overholonomicity and devissability in overconvergent $F$-isocrystals
are identical. With what we have proved in the first and second steps,
the equality between the overholonomicity
and the devissability in overconvergent $F$-isocrystals
implies that the overholonomicity is stable under
Grothendieck's aforesaid five cohomological operations and is wide enough since it contains
overconvergent $F$-isocrystals on smooth $k$-varieties. Also, for this
purpose, it is enough to prove the overholonomicity of
overconvergent $F$-isocrystals on smooth $k$-varieties. Fortunately,
Kedlaya has just checked that Shiho's semistable reduction
conjecture is exact, i.e., that given an overconvergent
$F$-isocrystal on a smooth $k$-variety, one can pull back along a
suitable generically finite cover to obtain an isocrystal which
extends, with logarithmic singularities and nilpotent residues, to
some complete variety (see \cite{kedlaya-semistableI},
\cite{kedlaya-semistableII}, \cite{kedlaya-semistableIII} and at last
\cite{kedlaya-semistableIV}). Kedlaya's semistable reduction
theorem gives us a very important tool since we come down by
descent (indeed overholonomicity behaves well by proper
generically \'etale descent thanks to its stability by extraordinary
inverse images and direct images) to study the case of the overconvergent
$F$-isocrystals which extend with logarithmic singularities and
nilpotent residues to some complete variety. We began this study in
\cite{caro_log-iso-hol}. We proceed in this article and check the
overholonomicity of these log-extendable overconvergent
$F$-isocrystals, which finish the check of our third step. The
technical key point of this overholonomicity is a comparison theorem
between relative logarithmic rigid cohomology and rigid cohomology and above all, in a more
general essential context, the fact that both cohomologies are not so different. This
fundamental key point was checked by the second author and the fact
that this implies the overholonomicity of log-extendable
overconvergent $F$-isocrystals was checked by the first one.

Now, let us describe the contents.
Let $g : \X \rightarrow \T$ be a smooth morphism of smooth formal $\V$-schemes,
relative dimension pure of $d$,
let $\ZZ$ be a relatively strict normal crossing divisor of $\X$ over $\T$,
let $\Y$ be a complement of $\ZZ$ in $\X$,
let $D$ be a closed subscheme of $X$ and $U$ the complement of $D$ in $X$.
Let $\X^\# = (\X,\ZZ)$ be the logarithmic formal $\mathcal V$-scheme
with the logarithmic structure associated to $\ZZ$ and $u\,:\, \X ^\# \rightarrow \X$ be
the canonical morphism.

In the first chapter, we compare
logarithmic rigid cohomology and rigid cohomology with overconvergent coefficients in the relative situations.
Let $E$ be a log-isocrystal on $U^\#/\T_K$ overconvergent along $D$ (see the definition in \ref{logovcon}).
Suppose that, along each irreducible component of $Z$ which is not included in $D$,
(a) none of differences of exponents is a $p$-adic Liouville number and
(b') any exponent is neither a $p$-adic Liouville number nor a positive integer.
Then the natural comparison map
$\R g_{K*}(j_U^\dag\Omega_{\Xlogan/\T_K}^\bullet
   \otimes_{j_U^\dag\O_{]X[_\X}} E)
       \riso  \R g_{K*}(j_{Y \cap U}^\dag\Omega_{\Xan/\T_K}^\bullet
       \otimes_{j_{Y \cap U}^\dag\O_{]X[_\X}}
       j_{Y \cap U}^\dag E)$ is an isomorphism (see \ref{Nobuorigid}).
Let us consider the case where $g$ has a section which is identified with $\ZZ$
such that $Z \not\subset D$.
If one assumes (a) above and (b) none of exponents is a $p$-adic Liouville number,
then the difference is given by the complex which consists of overconvergent log-isocrystals
on the divisor (see \ref{section}).
In the second section we develop a notion of quasi-coherence in formal log-schemes,
which was studied by Berthelot in the case of formal schemes (see \cite{Beintro2}), and
cohomological operators such as direct images and
extraordinary inverse images by morphisms of smooth formal $\V$-log-schemes.
Furthermore, we translate this comparison in the language of arithmetic $\D$-modules
in the third section.

In the second chapter, we recall in the first section Kedlaya's semistable reduction theorem.
Let $\E$ be a coherent
$\D ^\dag _{\X^\#,\Q}$-module which is a locally projective
$\O_{\X,\Q}$-module of finite type which satisfies the conditions (a) and (b') above.
Then, using the comparison theorem of the first chapter, we check that the canonical morphism
$u _{+} (\E) \rightarrow \E (\hdag Z)$ is an isomorphism (see \ref{theorhoiso}).
This implies that the canonical morphism
$\Omega _{\X ^\#/\T,\Q} ^\bullet \otimes _{\O _{\X,\Q}}  \E
\rightarrow
\Omega _{\X /\T,\Q} ^\bullet \otimes _{\O _{\X,\Q}}
  \E (\hdag Z)$
is a quasi-isomorphism (see \ref{coro-theorhoiso}).
In the third section, we prove that if
(c) none of elements of $\mathrm{Exp}(\E) ^\gr $ (the group generated by all exponents of $\E$)
is a $p$-adic Liouville number,
then $u _{+} (\E)$ is overholonomic, which implies that $\E (\hdag Z)$
(the isocrystal on $Y$ overconvergent along $Z$ associated to $\E$) is overholonomic.
The principal reason why we need to replace the conditions (a) and (b') by the condition (c) is
because we need here something stable under duality and because the log-relative duality isomorphism is of the form
(see \cite[5.25.2]{caro_log-iso-hol} and \cite[5.22]{caro_log-iso-hol}):
$\DD _{\X } \circ u_+ (\E) \riso u_+ (\E ^\vee (-\ZZ))$, where ``$\DD _{\X }$''
means the dual as $\D ^\dag _{\X,\Q}$-module
and ``$\vee$'' is the dual as a convergent log-isocrystal
(e.g., even if $\E$ is a convergent log-$F$-isocrystal,
then unfortunately $\E^\vee(-\ZZ)$ have positive exponents).
Hence, using Kedlaya's semistable reduction theorem, we obtain by descent the overholonomicity of
overconvergent $F$-isocrystals on smooth $k$-varieties.
Thus, the notion of overholonomicity, overcoherence and devissability in overconvergent $F$-isocrystals
are the same. Also, the overholonomicity behaves as good as the holonomicity in the classical theory.
Finally, we extend some results of \cite{caro_courbe-nouveau}.
More precisely, let $\X$ be a smooth separated formal $\V$-scheme of dimension $1$,
$Z$ a divisor of $X$, $\Y:= \X \setminus T$ and $\E $ a complex of
$ F\text{-}D ^\mathrm{b} _\mathrm{coh} (\D ^\dag _{\X } (\hdag Z) _{\Q})$.
Then, firstly $\E$ is holonomic if and only if $\E$ is overholonomic. Secondly,
 if the restriction of $\E$ on $\Y$ is a holonomic $F$-$\D ^\dag _{\Y ,\Q}$-module,
 then $\E$ is a holonomic $F$-$\D ^\dag _{\X ,\Q}$-module.
Both results should be true in higher dimensions but are still conjectures.
Besides, this second conjecture implies the first one
and is the strongest Berthelot's conjecture on the stability of holonomicity (see \cite[5.3.6.D]{Beintro2}).

\begin{nota*}
\label{nota}
Let $\V$ be a complete valuation ring of characteristic $0$,
$k$ its residue field of characteristic $p>0$,
$K$ its fractions field with a multiplicative valuation $|\mbox{-}|$,
$\S := \Spf \V$.
From the section \ref{nota-1.2} we assume furthermore that $K$ is discrete,
$\pi$ is a uniformizer and the residue field $k$ is perfect.
We also fix $\sigma$ : $\V \rightarrow \V$ a lifting of the $a$th power Frobenius.

If $\X \rightarrow \T$ is a morphism of smooth formal schemes over $\S$ and
if $\ZZ$ is a relatively strict normal crossing divisor of $\X$ over $\T$,
we denote by $\X^\# = (\X,\ZZ)$ the smooth log-formal
$\V$-scheme whose underlying smooth formal $\V$-scheme is $\X$
and whose logarithmic structure is the canonical one induced by $\ZZ$.
To indicate the corresponding special fibers, we use roman letters, e.g.,
$X$, $Z$ and $T$ are the special fibers of $\X$, $\ZZ$ and $\T$.
Similarly, $X^\# = (X,Z)$ means the canonical log-scheme induced by any smooth scheme $X$
and any strict normal crossing divisor $Z$ of $X$.
We denote by $d_X$ or simply $d$ the dimension of $X$.
The subscript $\Q$ means that we have applied the functor $- \otimes _\Z \Q$.
Modules over a noncommutative ring are left modules, unless otherwise indicated.
\end{nota*}
\vspace{3mm}

\noindent
\textbf{Acknowledgment.}
Both authors heartily thank Mr. Horiba who supported the conference ``$p$-adic aspects in Arithmetic Geometry, Tambara"(June, 2007),
where we started this project.
The first author thanks the University of Paris-Sud for his excellent working conditions and Hiroshima University
for his nice hospitality in June, 2007.
The second author expresses his appreciation for the hospitality
of Department of Mathematics, Hiroshima University, where he has done this work.

\section{A comparison theorem between relative log-rigid cohomology and
relative rigid cohomology}

\subsection{Proof of the comparison theorem}

In this section we only suppose that $K$ is a complete field of characteristic $0$ under the valuation $| \mbox{-} |$
and the residue field $k$ of the integer ring $\V$ is of characteristic $p > 0$.
Let us fix several notation in rigid cohomology.
For a formal $\V$-scheme $\PP$ of finite type,
let $\PP_K$ be the Raynaud generic fiber of $\PP$ which is a
quasi-compact and quasi-separated rigid analytic $K$-space,
$\sp : \PP_K \rightarrow \PP$ the specialization map,
and $]T[_\PP = \sp^{-1}(T)$ the tube of a locally closed
subscheme $T$ in $P = \PP \times_{\Spf \V}\, \Spec\, k$.
For a morphism $u : \PP \rightarrow \QQ$, we denote by $u_K : \PP_K \rightarrow \QQ_K$
the morphism of rigid analytic spaces associated to $u$.
Let $X$ be a closed subscheme of $P$, $Z$ a closed subscheme
of $X$, and $Y$ the complement of $Z$ in $X$.
For any admissible open subset $V \subset  ]X[_\PP$,
we denote by
$\alpha_V : V \rightarrow ]X[_\PP$
the canonical inclusion.
Let $\AA$ be a sheaf of rings on $]X[_\PP$.
For an $\AA$-module $\H$, let
$j_Y^\dag\H = \underset{\underset{V}{\longrightarrow}}{\mathrm{lim}}\,
\alpha_{V*}(\H|_V)$ denote the sheaf of sections of $\H$ overconvergent along $Z$,
where $V$ runs over all strict neighborhoods of $]Y[_\PP$ in $]X[_\PP$.
The functor $j_Y^\dag$ is exact and the natural morphism
$\H \rightarrow j_Y^\dag\H$ is an epimorphism \cite[2.1.3]{Be}.
The sheaf $\underline{\Gamma}_{]Z[_\PP}^\dag(\H)$
of sections of $\H$ whose supports are included in $]Z[_\PP$ is defined by
the exact sequence
\begin{equation}
  \label{exact-trngl-rig}
       0\, \longrightarrow\, \underline{\Gamma}_{]Z[_\PP}^\dag(\H)\,
     \longrightarrow\, \H\, \longrightarrow\, j_Y^\dag\H\, \longrightarrow\, 0.
\end{equation}
Then $\underline{\Gamma}_{]Z[_\PP}^\dag$ is an exact functor by the snake lemma
\cite[2.1.6]{Be}.

We will fix some notation:
let $g : \X \rightarrow \T$ be a smooth morphism of smooth formal schemes over $\S$,
relative dimension pure of $d$,
let $\ZZ$ be a relatively strict normal crossing divisor of $\X$ over $\T$,
let $\Y$ be a complement of $\ZZ$ in $\X$,
let $D$ be a closed subscheme of $X$ and $\U$ the complement of $D$ in $\X$.
Let $\X^\# = (\X,\ZZ)$ be the logarithmic formal $\mathcal V$-scheme
with the logarithmic structure associated to $\ZZ$, and $\U^\#$
the restriction of $\X^\#$ on $\U$.
Let $\Xlogan = (\X_K,\ZZ _K)$ be the rigid analytic space endowed
with the logarithmic structure associated to $\ZZ _K$
and $\Omega_{\Xlogan/\T_K}^\bullet$ the de Rham complex of logarithmic
K\"ahler differential forms on $\Xlogan$.
Then the underlying analytic space of $\Xlogan$ is $]X[_\X = \Xan$ and
$\Omega_{\Xlogan/\T_K}^\bullet\, \cong\, \sp^*\Omega_{\X^\# /\T, \Q}^\bullet$.

We recall the definition of logarithmic connection
with the overconvergent condition
(\cite[4.2]{caro_log-iso-hol} and \cite[6.5.4]{kedlaya-semistableI}).
Since the condition is local,
we may suppose that $\X$ and $\T$ are affine and $D$ is defined by $f = 0$
in $X$ for $f \in \Gamma(\X, \O_\X)$.
Let $z_1, z_2, \cdots, z_d$ are relatively local coordinates of $\X$ over $\T$
such that the irreducible component $\ZZ_i$ of the relatively strict normal crossing
divisor $\ZZ = \cup_{i=1}^s\, \ZZ_i$ is defined by $z_i = 0$.
An integrable logarithmic connection
$\nabla : E \rightarrow
j_U^\dag\Omega_{\Xlogan/\T_K}^1 \otimes_{j_U^\dag\O_{]X[_\X}} E$
is overconvergent if there exist
a strict neighborhood $V$ of $]U[_\X$ in $]X[_\X$
and a locally free $\O_V$-module $\E$
furnished with an integrable logarithmic connection
$\nabla : \E \rightarrow (\Omega_{\Xlogan/\T_K}^1|_V)
\otimes_{\O_V} \E$ such that $j_U^\dag(\E, \nabla) = (E, \nabla)$, which
satisfies the following overconvergent condition :
for any $\xi \in |K^\times|_\Q \cap\, ]0, 1[$,
there exists an affinoid strict neighborhood $W \subset V$ of $]U[_\X$ in $]X[_\X$
such that
\begin{equation}
\label{logovcon}
       ||\underline{\partial}_\#^{[\underline{n}]}(e)||\xi^{|\underline{n}|} \rightarrow 0\, \, \,
       (\mathrm{as}\, \, |\underline{n}| \rightarrow \infty)
\end{equation}
for any section $e \in \Gamma(W, \E)$.
Here $||\, \mbox{-}\, ||$ is a Banach $\Gamma(W, \O_{]X[_\X})$-norm on
$\Gamma(W, \E)$,
$\partial_{\#i} = \nabla(z_i\frac{\partial}{\partial z_i})$ for $1 \leq i \leq s$,
$\partial_i = \nabla(\frac{\partial}{\partial z_i})$ for $s + 1 \leq i \leq d$,
and, $|\underline{n}| = n_1 + \cdots + n_d$,
$\underline{n}! = n_1! \cdots n_d!$ and $\underline{\partial}_{\#}^{[\underline{n}]}
= \frac{1}{\underline{n}!}
\left(\prod_{i = 1}^s\prod_{j= 0}^{n_i - 1}(\partial_{\#i} - j)\right)
\partial_{s+1}^{n_{s+1}} \cdots \partial_d^{n_d}$
for a multi-index $\underline{n} = (n_1, \cdots, n_d)$. $(E, \nabla)$
is called a log-isocrystal on $U^{\#}/\T_K$ overconvergent along $D$
(simply denote by $E$ and called an overconvergent log-isocrystal).

Let $(E, \nabla)$ be a log-isocrystals on $U^\#/\T_K$ overconvergent along $D$
and let $Z_i$ be an irreducible component $Z_i$ of $Z$
 which is not included in $D$.
The eigenvalues of the {\it residue} of $\nabla$ along $\ZZ_{iK}$
at the generic point of $\ZZ_{iK}$ is called
 ``exponent'' of  $E$ along $Z_i$
(for a definition of the residue, see for example \cite[2.3.9]{kedlaya-semistableI}).
This is related with the definition in \cite[1, sect. 6]{AB1}.
Any exponent is contained in $\Z_p$ by \ref{logovcon}.

Let $\I_\ZZ$ be a sheaf of ideals of $\ZZ$ in $\X$. Since $\I_\ZZ$ is invertible,
$\I_{\ZZ, \Q}$ is a coherent
$\D_{\X^\#, \Q}^\dag$-module which is an invertible $\O_{\X, \Q}$-module.
Hence, $I_{\ZZ, \Q} = \sp^*\I_{\ZZ, \Q}$ is a convergent isocrystal on $X/K$
with logarithmic poles along $Z$. Let $E$ be a log-isocrystal
on $U^{\#}/\T_K$ overconvergent along $D$.
For an integer $m$, we put
$$
        E(m\ZZ) = E \otimes_{j_U^\dag\O_{]X[_\X}}j_U^\dag I_{\ZZ, \Q}^{\otimes -m}.
$$
$E(m\ZZ)$ is an overconvergent log-isocrystal
and the exponents of $E(m\ZZ)$ is the exponents of $E$ minus $m$.
Then there is a natural commutative diagram
\begin{equation}
\label{finitepoles}
      \begin{array}{ccc}
          E &\displaystyle{\mathop{\longrightarrow}^\subset} &E(m\ZZ) \\
          = \downarrow \hspace*{3mm} & &\downarrow \\
          E &\longrightarrow &j_{Y \cap U}^\dag E
       \end{array}
\end{equation}
for any nonnegative integer $m$.

A $p$-adic integer $\alpha$ is a ``$p$-adic Liouville number'' if
the radius of convergence of formal power series,
either $\sum_{n \in \Z_{\geq 0}, n \ne \alpha}\, x^n/(n - \alpha)$
or $\sum_{n \in \Z_{\geq 0}, n \ne -\alpha}\, x^n/(n + \alpha)$, is less than $1$.
Note that (1) a $p$-adic integer which is an algebraic number
is not a $p$-adic Liouville number
and (2) a $p$-adic integer $\alpha$ is a $p$-adic Liouville number if and only if
so is $-\alpha$ (resp. $\alpha + m$ for any integer $m$).
For $p$-adic Liouville numbers, we refer to \cite[VI, 1]{dgs} and \cite[1.2]{BC1}.

\begin{theo}\label{Nobuorigid}
With the above notation, let $E$ be a log-isocrystal on $U^\#/\T_K$ overconvergent
along $D$. Suppose that
\begin{list}{}{}
\item[\mbox{\rm (a)}] none of differences of exponents
of $E$ is a $p$-adic Liouville number, and
\item[\mbox{\rm (b)}] none of exponents of $E$ is a $p$-adic Liouville number
\end{list}
\noindent
along each irreducible component $Z_i$ of $Z$ such that $Z_i \not\subset D$. Let $c$ be a nonnegative integer defined by
$$
      c = \max\{ e\, |\, e\, \mbox{is a positive integral exponent
      of $E$ along some irreducible
      component $Z_i$ of $Z$ such that $Z_i \not\subset D$} \}\, \cup\, \{ 0 \}
$$
Then the diagram \ref{finitepoles} induces an isomorphism
\begin{equation}
  \label{Nobuorigid-isocone}
\R g_{K*}\underline{\Gamma}^\dag _{]Z[  _\X}
(j_U^\dag\Omega_{\Xlogan/\T_K}^\bullet \otimes_{j_U^\dag\O_{]X[_\X}} E)\,
\cong\, \R g_{K*}\mathrm{Cone}\left(
j_U^\dag\Omega_{\Xlogan/\T_K}^\bullet \otimes_{j_U^\dag\O_{]X[_\X}} E
\rightarrow j_U^\dag\Omega_{\Xlogan/\T_K}^\bullet
   \otimes_{j_U^\dag\O_{]X[_\X}} E(m\mathcal Z)\right)[-1]
\end{equation}
for any $m \geq c$.
In particular, if none of exponents along each irreducible component $Z_i$
of $Z$ such that $Z_i \not\subset D$
is a positive integer, then the restriction induces an isomorphism
\begin{equation}
\label{Nobuorigid-iso}
\R g_{K*}(j_U^\dag\Omega_{\Xlogan/\T_K}^\bullet
   \otimes_{j_U^\dag\O_{]X[_\X}} E)
       \riso  \R g_{K*}(j_{Y \cap U}^\dag\Omega_{\Xan/\T_K}^\bullet
       \otimes_{j_{Y \cap U}^\dag\O_{]X[_\X}}
       j_{Y \cap U}^\dag E).
\end{equation}
\end{theo}

\begin{rema}
\begin{enumerate}
\item In fact, we will see in \ref{coro-theorhoiso} that the isomorphism \ref{Nobuorigid-iso} remains true without the functor $\R g_{K*}$.
But the first step towards this result is to establish \ref{Nobuorigid}.

\item Note that $ j_{Y \cap U}^\dag E$ is an isocrystal on $Y \cap U/\T_K$
overconvergent along $Z \cup D$
and the right handside of the isomorphism in the theorem above is a relative rigid cohomology
with respect to
the closed immersion $T \rightarrow \T$. It is independent of the choice
of $\X$ which is smooth over $\T$ around $U$ \cite[sect. 10]{descent}.
The left handside of \ref{Nobuorigid-isocone} in the theorem above
is regarded as a relative logarithmic rigid cohomology.

\item This type of comparison theorem between
$p$-adic cohomology with logarithmic poles and rigid cohomology was
studied in \cite[3.1]{BC2}, \cite[3.5.1]{Tsu1}, \cite[2.2.4 and 2.2.13]{Shi1} (see also the definition \cite[2.1.5]{Shi1})
and \cite[A.1]{BB1}. They suppose that $\FF$ is locally free
on the formal side or for \cite[2.2.4 and 2.2.13]{Shi1} it concerns the absolute case.
In the theorem above we relax this assumption
and suppose that $\FF$ is locally free only on the analytic side.

\item One can also prove the comparison theorem in the case $g$ is smooth around $U$
replacing \ref{wsfib} and \ref{fib} (the weak fibration theorem) by the strong
forms (the strong fibration theorem) with modifications.
\end{enumerate}
\end{rema}

\begin{rema}\label{remexp}
For a log-isocrystal $E$ on $U^{\#}/\T_K$ overconvergent along $D$, we put
a monoid $\mathrm{Exp}(E)$ (resp. an abelian group $\mathrm{Exp}(E)^{\mathrm{gr}}$)
which is generated by all exponents along irreducible components $Z_i$ of $Z$
such that $Z_i \not\subset D$.
$\mathrm{Exp}(E)$ (resp. $\mathrm{Exp}(E)^{\mathrm{gr}}$)
is included in $\mathbb Z_p$ and
does not depend on the choice of local coordinates.

\begin{enumerate}
\item\label{remexp-5}
Let $\X^\# = (\X, \ZZ)$ and ${\X'}^\# = (\X', \ZZ')$ be
smooth formal $\V$-schemes with relatively strict normal crossing divisors over $\T$,
let $\U, D, \U^\#, \U', D', \U'{^\#}$ as above,
and let $h : \X' \rightarrow \X$ be a morphism over $\T$ such that
$h^{-1}(D \cup Z) \subset D' \cup Z'$.
Suppose that $h$ induces a log-morphism $(h|_{\U'})^\# : {\U'}^\# \rightarrow \U^\#$.
Then the inverse image $h_K^{\#*}E$ is a log-isocrystal
on ${U'}^{\#}/\T_K$ overconvergent along $D'$ because
$h_K$ induces a log-morphism of rigid analytic spaces
between suitable strict neighborhoods by our assumption.
Suppose furthermore that
none of elements in $\mathrm{Exp}(E)$ (resp. $\mathrm{Exp}(E)^{\mathrm{gr}}$)
is a $p$-adic Liouville number.
Then the same holds for the inverse image $h_K^{\#*}E$.
Indeed, for a suitable choice of local coordiantes $z_i\, (1 \leq i \leq s)$
and $z_j'\, (1 \leq j \leq s')$ along normal crossing divisors
$\ZZ$ and $\ZZ'$ of $\X$ and $\X'$ respectively, we have
$z_i = u_i{z_1'}^{m_{i1}}\cdots{z_{s'}'}^{m_{is'}}$ locally at a generic point of $\ZZ'$.
Here $u_i$ is a unit of $\O_{\U'}$ and $m_{ij}$ is a nonnegative integer.
Since the residues of $E$ with respect to $Z_{i_1}$ and $Z_{i_2}$ commute
with each other by the integrability of the log-conenction
and $dz_i/z_i\, \equiv\, \sum_j\, m_{ij}dz_j'/z_j'\, (\mathrm{mod}\, \Omega_{\U'/\T}^1)$,
$\mathrm{Exp}(h_K^{\#*}E)$ is a submonoid of $\mathrm{Exp}(E)$.
(See \cite[6.2.5]{AB1}.)

Even if $\mathrm{Exp}(E)$ does not contained any positive integers,
it might happen that some exponent of inverse image $h_K^{\#*}E$
is a positive integer. If we denote by $\Q_{\geq 0}$
the monoid consisting of nonnegative rational numbers,
then $\mathrm{Exp}(E) \cap \Q_{\geq 0}$ is finitely generated as a monoid.
Hence, if one takes a sufficiently large integer $m$, then
$\mathrm{Exp}(E(m\ZZ))$ does not contained any positive rational numbers
and the same holds for any inverse image
$h_K^{\#*}E(m\ZZ)$ as above.

\item
\label{remexp-6}
Let $h^{\#} : {\X'}^\#\rightarrow \X^\#$ be a log-morphism
such that $h^{-1}(D) = D'$ and $h^{-1}(\ZZ) = \ZZ'$.
Suppose that the underlying morphism $h$ is finite \'etale.
Note that local parameters of $\X^\#$ becomes local parameters of ${\X'}^\#$.
Then,
for a log-isocrystal $E'$ on ${U'}^{\#}/\T_K$ overconvergent along $D'$,
$h_{K*}^{\#}E'$ is a log-isocrystal on $U^{\#}/\T_K$ overconvergent along $D$.
Moreover, for an irreducible component $Z_i$ of $Z$ such that $Z_i \not\subset D$,
the exponents of $h_{K*}^{\#}E'$
$Z_i$ coincide with the exponents of $E'$
along $h^{-1}(Z)$ (including multiplicities).
In particular, $\mathrm{Exp}(h_{K*}^{\#}E') = \mathrm{Exp}(E')$.
(See \cite[6.5.4]{AB1}.)
The first part easily follows from our geometrical situation and
we have $\mathrm{rank}_{j^\dag\O_{]\overline{X}[_\X}} h_{K*}^{\#}E'
= \mathrm{deg}(h)\mathrm{rank}_{j^\dag\O_{]\overline{X}'[_{\X'}}} E'$, where
$\mathrm{deg}(h)$ is the degree of the underlying morphism of $h$.
The second part is a problem only along the generic point
of $\ZZ_i$. We may assume that $Z$ is irreducible and does not included in $D$.
Let $(j^\dag \O_{]\overline{X}[_\X})_\ZZ^{\widehat{\, \, }}$ be a completion of
localization of $j^\dag \O_{]\overline{X}[_\X}$ along $\ZZ_K$.
Then the ring of global sections of
$(j^\dag \O_{]\overline{X}[_\X})_\ZZ^{\widehat{\, \, }}$ is isomorphic
to $K(\ZZ)[[z]]$,
where $z$ is a local coordinate of $\ZZ$ and $K(\ZZ)$ is the function field of $\ZZ$,
and the ring of global sections of
$(j^\dag \O_{]\overline{X}'[_{\X'}})_\ZZ^{\widehat{\, \, }}$ is isomorphic to a
direct sum of finite unramified extensions of $K(\ZZ)[[z]]$.
We may replace the residue field $K(\ZZ)$ of
$K(\ZZ)[[z]]$ by its algebraic closure $\overline{K(\ZZ)}$
since all exponents are contained in $\Z_p$
and invariant under any automorphism of $\overline{K(\ZZ)}$.
Hence, the corresponding extension to
$(j^\dag \O_{]\overline{X}'[_{\X'}})_\ZZ^{\widehat{\, \, }}$ is
a direct sum of $\mathrm{deg}(h)$ copies of $\overline{K(\ZZ)}[[z]]$.
Now our second assertion is clear.
\end{enumerate}
\end{rema}

\vspace{3mm}

First we prove a special case.

\begin{prop}\label{section} Under the hypothesis in \ref{Nobuorigid}, suppose
that $\ZZ$ is irreducible such that $Z \not\subset D$,
and that the composition $g \circ i : \ZZ \rightarrow \T$
of the closed immersion $i : \ZZ \rightarrow \X$ and
$g : \X \rightarrow \T$ is an isomorphism.
If we define $T \cap U = Z \cap U$ through the isomorphism $g \circ i : Z \rightarrow T$,
then $g_{K*}\nabla : g_{K*}(E(m\ZZ)/E)
\rightarrow g_{K*}(j_U^\dag\Omega_{\Xlogan/\T_K}^1
   \otimes_{j_U^\dag\O_{]X[_\X}} E(m\ZZ)/E)$
is a $j_{T \cap U}^\dag\O_{]T[_\T}$-homomorphism of locally free
$j_{T \cap U}^\dag\O_{]T[_\T}$-modules of finite type
and the natural morphism \ref{Nobuorigid-iso} induces an isomorphism
\begin{equation}
   \R g_{K*}\underline{\Gamma}^\dag_{]Z[_\X}(j_U^\dag\Omega_{\Xlogan/\T_K}^\bullet
   \otimes_{j_U^\dag\O_{]X[_\X}} E)\,
   \cong\, \left[g_{K*}(E(m\ZZ)/E)\,
   \displaystyle{\mathop{\longrightarrow}^{g_{K*}\nabla}}\,
   g_{K*}(j_U^\dag\Omega_{\Xlogan/\T_K}^1
   \otimes_{j_U^\dag\O_{]X[_\X}} E(m\ZZ)/E)\right][-1]
\end{equation}
for any $m \geq c$ in the derived category of complexes of
$j_{T \cap U}^\dag\mathcal O_{]T[_\T}$-modules.
Here $[A \rightarrow B]$ means a complex consisting of the terms of degree $0$ and
degree $1$.
\end{prop}

\vspace{3mm}

We will see, in \ref{overconv}, the overconvergence of the induced
Gauss-Manin connection
on $g_{K \ast}(E(m\ZZ)/E)$ in the relative case.
An example such that the cokernel of
$g_{K*}\nabla : g_{K*}(E(m\ZZ)/E)
\rightarrow g_{K*}(j_U^\dag\Omega_{\Xlogan/\T_K}^1
   \otimes_{j_U^\dag\O_{]X[_\X}} E(m\ZZ)/E)$ is not locally free
is also given in \ref{example}

\vspace{3mm}

\begin{proof} We divide the proof of \ref{section} into 7 parts.

\vspace{3mm}

\noindent
$0^\circ$ {\it Reduce to the case where none of exponents of $E$
along $\ZZ$ is a positive integer, thai is, $c = 0$.}

We shall prove that $\R^q g_{K *}(E(\ZZ)/E) =  0$ for $q \ne 0$
and the locally freeness of $g_{K*}(E(\ZZ)/E)$.
Since $i^{-1}(X \setminus U) = Z \setminus U$ as underlying topological spaces,
$i_K^*E(\ZZ) = j_{Z \cap U}^\dag\O_{]Z[_\ZZ}
\otimes_{i_K^{-1}j_U^\dag\O_{]X[_\X}} i_K^{-1}E(\ZZ)$
is a locally free $j_{Z \cap U}^\dag\O_{]Z[_\ZZ}$-module and
the adjoint gives an isomorphism $i_{K *}i_K^*E(\ZZ)\, \cong\, E(\ZZ)/E$.
Because $i$ is a closed immersion, $i_K : ]Z[_\ZZ \rightarrow ]X[_\X$ is an affinoid morphism.
Hence $\R i_{K *}\M = i_{K *}\M$ for any coherent
$j_{Z \cap U}^\dag\O_{]Z[_\ZZ}$-module $\M$ by $i^{-1}(X \setminus U) = Z \setminus U$
\cite[5.2.2]{descent}. Since $g\circ i$ is an isomorphism, we have
$$
    \R g_{K*}(E(\ZZ)/E) = \R g_{K*}(i_{K*}i_K^*E(\ZZ))
    = \R g_{K*}\R i_{K*}i_K^*E(\ZZ) = \R (g\circ i)_{K *}i_K^*E(\ZZ)
    = (g\circ i)_{K *}i_K^*E(\ZZ).
$$
and the two assertions above.
Therefore, we show, for $m \geq 0$, $\R^q g_{K *}(E(m\ZZ)/E) =  0$ for $q \ne 0$
and $g_{K*}(E(m\ZZ)/E)$ is a locally free $j_{T \cap U}^\dag\O_{]T[_\T}$-module
of finite type by induction on $m$.

The commutative diagram \ref{finitepoles} induces a triangle
$$
    \begin{array}{c}
       \R g_{K*}\mathrm{Cone}\left(j_U^\dag\Omega_{\Xlogan/\T_K}^\bullet
   \otimes_{j_U^\dag\O_{]X[_\X}} E \rightarrow
       j_U^\dag\Omega_{\Xlogan/\T_K}^\bullet
   \otimes_{j_U^\dag\O_{]X[_\X}} E(m\ZZ)\right)[-1] \\
   +1 \swarrow \hspace*{30mm} \nwarrow \\
       \R g_{K*}\underline{\Gamma}^\dag_{]Z[_\X}
       (j_U^\dag\Omega_{\Xlogan/\T_K}^\bullet \otimes_{j_U^\dag\O_{]X[_\X}} E)\,
   \rightarrow  \R g_{K*}\underline{\Gamma}^\dag_{]Z[_\X}
   (j_U^\dag\Omega_{\Xlogan/\T_K}^\bullet
   \otimes_{j_U^\dag\O_{]X[_\X}} E(m\ZZ))
   \end{array}
$$
for any $m \geq 0$.
If we prove the vanishing
$\R g_{K*}\underline{\Gamma}^\dag_{]Z[_\X}(j_U^\dag\Omega_{\Xlogan/\T_K}^\bullet
       \otimes_{j_U^\dag\O_{]X[_\X}} E) = 0$ for $c = 0$,
then the triangle above induces the desired isomorphism.
Hence, we may assume $m = c = 0$ and we shall prove the vanishing.

\vspace{3mm}

\noindent
$1^\circ$ {\it Local problem on $X$ and $U$.}

By the \v{C}ech spectral sequences associated to a finite open covering $\{ \X_i \}$ of $\X$
(resp. a finite open covering $\{ \U_{ij} \}$ of each $\X_i \cap \U$) \cite[4.1.3]{Be0}
\cite[8.3.3]{descent}, the vanishing is local on $X$ and $U$. Since the vanishing
of $\R g_{K*}\underline{\Gamma}^\dag_{]Z[_\X}(j_U^\dag\Omega_{\Xlogan/\T_K}^\bullet
   \otimes_{j_U^\dag\O_{]X[_\X}} E)$ is trivial
in the case where $Z = \emptyset$, we may assume that $\X$ is affine,
$D$ is defined by a single equation $f = 0$ in $X$ for some $f \in \Gamma(\X, \O_\X)$,
and there are coordinates $z$ of $\X$ over $\T$
such that $\ZZ$ is defined by $z = 0$ in $\X$. Indeed, it is enough to take a certain covering
consisting of $\X \setminus \ZZ$ and a covering $\ZZ$.

\vspace*{3mm}

\noindent
$2^\circ$ {\it Reduction to the local case by rigid analytic geometry.}

Let us add some notation. Let us put
$]U[_{\X, \lambda} = \{ x \in ]X[_{\X}\, |\, |f(x)| \geq \lambda \}$
(resp. $]Y[_{\X, \lambda} = \{ x \in ]X[_{\X}\, |\, |z(x)| \geq \lambda \}$,
resp. $]Z \cap U[_{\ZZ, \lambda} = \{ x \in ]Z[_{\ZZ}\, |\, |\overline{f}(x)| \geq \lambda \}$,
resp. $[Z]_{\X, \lambda} = \{ x \in ]Z[_{\X}\, |\, |z(x)| \leq \lambda \}$)
for $\lambda \in |K^\times|_\Q \cap ]0, 1[$,
where $\overline{f}$ is the reduction of $f$ in $\Gamma(\ZZ, \O_{\ZZ})$.
We define $]T \cap U[_{\T, \lambda} = ]Z \cap U[_{\ZZ, \lambda}$
by the idetification through $g \circ i$.
Note that the set
$\{ ]U[_{\X, \lambda} \}_{\lambda \in \mathbb |K^\times|_\mathbb Q \cap ]0, 1[}$
forms a fundamental system of strict neighborhoods of $]U[_\X$ in $]X[_\X$.
Let $\alpha_V : V \rightarrow ]X[_\X$ denote the canonical morphism
for admissible open sets $V$ in $]X[_\X$.

Take $\nu \in |K^\times|_\Q \cap\, ]0, 1[$ such that
there is a locally free $\O_{]U[_{\X, \nu}}$-module $\E$
endowed with a logarithmic connection
$\nabla : \E \rightarrow (\Omega_{\Xlogan/\T_K}^1|_{]U[_{\X, \nu}})
\otimes_{\O_{]U[_{\X, \nu}}} \E$ which satisfies the overconvergent
condition \ref{logovcon}. Hence, there exist
a strictly increasing sequence $\underline{\xi} = (\xi_l)$
in $|K^\times|_\Q \cap\, ]0, 1[$
with $\xi_l \rightarrow 1^-$ as $l \rightarrow \infty$
and an increasing sequence $\underline{\lambda} = (\lambda_l)$
in $|K^\times|_\Q \cap\, [\nu, 1[$ such that, for any $l$,
\begin{equation}
\label{logoct}
       ||\partial_\#^{[n]}(e)||\xi_l^n \rightarrow 0\, \, \,
       (\mathrm{as}\, \, n \rightarrow \infty)
\end{equation}
for any section $e \in \Gamma(]U[_{\X, \lambda_l}, \mathcal E)$.
Here $\partial_{\#} = \nabla(z\frac{d}{dz})$ and
$\partial_{\#}^{[l]} = \frac{1}{l!}\partial_{\#}^l$.

Let $\AA$ be a sheaf of rings on $]X[_\X$.
Let $\eta \in |K^\times|_\Q \cap\, ]0, 1[$. We define a functor
$\underline{\Gamma}_{]Z[_\X, \eta}^\dag$
between the category of $\AA$-modules by the exact sequence
\begin{equation}
\label{suppZ}
      0\, \longrightarrow\, \underline{\Gamma}_{]Z[_\X, \eta}^\dag(\H)\, \longrightarrow\,
      \H\, \longrightarrow \underset{\mu \rightarrow \eta^-}{\mathrm{lim}}\,
    \alpha_{]Y[_{\X, \mu}*}(\H|_{]Y[_{\X, \mu}})\,
    \longrightarrow\, 0
\end{equation}
for any $\AA$-module $\H$. Here the morphism $\H\,
    \rightarrow\, \underset{\mu \rightarrow \eta^-}{\mathrm{lim}}\,
    \alpha_{]Y[_{\X, \mu}*}(\H|_{]Y[_{\X, \mu}})$
    is an epimorphism by the same reason
of the epimorphism $\H \rightarrow j_Y^\dag\H$.
One can easily see that $\underline{\Gamma}_{]Z[_\X, \eta}^\dag(\H)|_{]Y[_{\X, \eta}} = 0$
and $\underline{\Gamma}_{]Z[_\X, \eta}^\dag$ is an exact functor by the snake lemma.
For $\xi \in |K^\times|_\Q \cap [\eta, 1[$,
the restriction induces a morphism
$$
       \underline{\Gamma}_{]Z[_\X, \eta}^\dag(\H) \rightarrow
       \underline{\Gamma}_{]Z[_\X, \xi}^\dag(\H)
$$
of $\AA$-modules. By definition we have

\begin{prop}\label{supp} With the notation as above, the inductive system
induces an isomorphism
$$
\underset{\eta \rightarrow 1^-}{\mathrm{lim}}\,
\underline{\Gamma}_{]Z[_\X, \eta}^\dag(\H)\,
\cong\, \underline{\Gamma}_{]Z[_\X}^\dag(\H).
$$
\end{prop}

\begin{prop}\label{suppdag} Let $\lambda \in |K^\times|_\Q \cap ]0, 1[$.
\begin{enumerate}
\item The functor $\underline{\Gamma}_{]Z[_\X, \eta}^\dag$ commutes with filtered inductive limites. Also, for any $\AA$-module $\H$,
the natural morphism
$$
      \alpha_{]U[_{\X, \lambda}*}\left(
      \underline{\Gamma}_{]Z[_\X, \eta}^\dag(\H)|_{]U[_{\X, \lambda}}\right) \rightarrow
      \underline{\Gamma}_{]Z[_\X, \eta}^\dag
      \left(\alpha_{]U[_{\X, \lambda}*}(\H|_{]U[_{\X, \lambda}})\right)
$$
is an isomorphism. Moreover,
$j_U^\dag\underline{\Gamma}_{]Z[_\X, \eta}^\dag =
\underline{\Gamma}_{]Z[_\X, \eta}^\dag j_U^\dag$.
\item For any coherent $\O_{]U[_{\X, \lambda}}$-module $\H _\lambda$
and any $q \geq 1$ we have
$\R^q\alpha_{]U[_{\X, \lambda}*}\left(
      \underline{\Gamma}_{]Z[_\X, \eta}^\dag
      (\alpha_{]U[_{\X, \lambda}*}\H _\lambda)|_{]U[_{\X, \lambda}}\right) = 0$.
\end{enumerate}
\end{prop}

\begin{proof} (1) Since the morphism $\alpha_{]Y[_{\X, \mu}}$ is quasi-compact and quasi-separated, we obtain from \ref{suppZ} the first assertion.
By applying the functor
$\alpha_{]U[_{\X, \lambda}*}\alpha_{]U[_{\X, \lambda}}^{-1}$
to the exact sequence \ref{suppZ}, we get the sequence
$$
      0\, \longrightarrow\, \alpha_{]U[_{\X, \lambda}*}\left(
      \underline{\Gamma}_{]Z[_\X, \eta}^\dag(\H)|_{]U[_{\X, \lambda}}\right)\,
      \longrightarrow\, \alpha_{]U[_{\X, \lambda}*}(\H|_{]U[_{\X, \lambda}})\,
      \longrightarrow\, \alpha_{]U[_{\X, \lambda}*}
    \left(\left(\underset{\mu \rightarrow \eta^-}{\mathrm{lim}}\,
    \alpha_{]Y[_{\X, \mu}*}(\H|_{]Y[_{\X, \mu}})\right)|_{]U[_{\X, \lambda}}\right)
    \longrightarrow\, 0,
$$
which is exact by the similar proof of \cite[2.1.3.(i)]{Be}.
The quasi-compactness and quasi-separateness of $\alpha_{]U[_{\X, \lambda}}$
implies the assertions.

(2) Because $\H _\lambda$ is a coherent $\O_{]U[_{\X, \lambda}}$-module
and both $]U[_{\X, \lambda}$ and $]Y[_{\X, \mu}$
are affinoid subdomains of the affinoid $]X[_\X$,
$\R^q\alpha_{]U[_{\X, \lambda}*}(\H _\lambda) = 0$
and $\R^q\alpha_{]U[_{\X, \lambda}*}
\left(\left(\underset{\mu \rightarrow \eta^-}{\mathrm{lim}}\,
    \alpha_{]Y[_{\X, \mu} *}(\H _\lambda|_{]Y[_{\X, \mu}})\right)|_{]U[_{\X, \lambda}}\right) = 0$
for $q \geq 1$ by Kiehl's Theorem B \cite[2.4]{quasistein}.
These facts and the exactness of the sequence in the proof of (1) imply the vanishing of higher direct images.
\end{proof}

Since $g_K$ is an affinoid morphism, it is quasi-compact and $\R g_{K*}$ commutes with
filtered inductive limits \cite[0.1.8]{Be}.
Hence we have
$$
   \begin{array}{lll}
      \R^q g_{K*}\underline{\Gamma}_{]Z[_\X}^\dag
   (j_U^\dag\Omega_{\Xlogan/\T_K}^\bullet \otimes_{j_U^\dag\O_{]X[_\X}} E) \\
   \hspace*{30mm} \cong\,
   \R^q g_{K*}\left(\underset{\eta \rightarrow 1^-}{\mathrm{lim}}\,
   \underline{\Gamma}_{]Z[_\X, \eta}^\dag
   \left(j_U^\dag(\Omega_{\Xlogan/\T_K}^\bullet
   \otimes_{\O_{]X[_\X}} \alpha_{]U[_{\X, \nu}*}\E)\right)\right) \\
      \hspace*{30mm} \cong\, \underset{\eta \rightarrow 1^-}{\mathrm{lim}}\,
      \R^q g_{K*}\underline{\Gamma}_{]Z[_\X, \eta}^\dag
   \left(\underset{\lambda \rightarrow 1^-}{\mathrm{lim}}\,\alpha_{]U[_{\X, \lambda}*}
   \left((\Omega_{\Xlogan/\T_K}^\bullet|_{]U[_{\X, \lambda}})
    \otimes_{\O_{]U[_{\X, \lambda}}} \E |_{]U[_{\X, \lambda}} \right)\right) \\
   \hspace*{30mm} \cong\, \underset{\eta \rightarrow 1^-}{\mathrm{lim}}\,
   \underset{\lambda \rightarrow 1^-}{\mathrm{lim}}\,
   \R^q g_{K*}\underline{\Gamma}_{]Z[_\X, \eta}^\dag
   \left(\alpha_{]U[_{\X, \lambda}*}
   \left((\Omega_{\Xlogan/\T_K}^\bullet|_{]U[_{\X, \lambda}})
    \otimes_{\O_{]U[_{\X, \lambda}}} \E |_{]U[_{\X, \lambda}} \right)\right) \\
      \hspace*{30mm} \cong\, \underset{\eta, \lambda \rightarrow 1^-}{\mathrm{lim}}\,
   \R^q g_{K*}\underline{\Gamma}_{]Z[_\X, \eta}^\dag
   \left(\alpha_{]U[_{\X, \lambda}*}
   \left((\Omega_{\Xlogan/\T_K}^\bullet|_{]U[_{\X, \lambda}})
    \otimes_{\O_{]U[_{\X, \lambda}}} \E |_{]U[_{\X, \lambda}} \right)\right)
   \end{array}
$$
for any $q$. Indeed, the first isomorphism follows from \ref{supp} and the other ones from
the commutation of the functors $\R g_{K*}$ and $\underline{\Gamma}_{]Z[_\X, \eta}^\dag$ (by \ref{suppdag})
with filtered inductive limits.
We will consider a filtered category indexed by
\begin{equation}
\label{index}
    \Lambda_{\underline{\xi}, \underline{\lambda}}
    = \left\{ (\lambda, \eta)\, \in \left(|K^\times|_\Q \cap ]0, 1[\right)^2\,\left|\,
    \begin{array}{l}
    \lambda > \eta,
     \lambda \geq \max\{\lambda_l, \nu\}, \\
     \eta < \xi_l\,\, \mbox{for some}\,\, l
     \end{array} \right.\right\}.
\end{equation}
Here the condition $\lambda > \eta$ comes from \ref{wsfib} (2).
This filtered category becomes a fundamental system for $\eta, \lambda \rightarrow 1^-$,
so that the limit with respect to $\Lambda_{\underline{\xi}, \underline{\lambda}}$
is same to the original one.

Let $g_\lambda :\, ]U[_{\X, \lambda} \rightarrow ]T[_\T$
and $g_{\lambda, \eta} :\, ]U[_{\X, \lambda} \cap [Z]_{\X, \eta} \rightarrow ]T[_\T$
denote restrictions of $g$ for
$(\lambda, \eta) \in \Lambda_{\underline{\xi}, \underline{\lambda}}$.
Then
$$
\begin{array}{l}
     \R g_{K*}\underline{\Gamma}_{]Z[_\X, \eta}^\dag
   \left(\alpha_{]U[_{\X, \lambda}*}
   \left(
   (\Omega_{\Xlogan/\T_K}^\bullet|_{]U[_{\X, \lambda}})
    \otimes_{\O_{]U[_{\X, \lambda}}} \E |_{]U[_{\X, \lambda}} \right)\right)
    \\
    \hspace*{45mm}
    \cong\,
    \R g_{\lambda*}\left(\underline{\Gamma}_{]Z[_\X, \eta}^\dag
   \left(\alpha_{]U[_{\X, \lambda}*}
   \left((\Omega_{\Xlogan/\T_K}^\bullet|_{]U[_{\X, \lambda}})
    \otimes_{\O_{]U[_{\X, \lambda}}} \E |_{]U[_{\X, \lambda}}
    \right)\right)|_{]U[_{\X, \lambda}}\right)
\end{array}
$$
by \ref{suppdag}.
Since $\underline{\Gamma}_{]Z[_\X, \eta}^\dag
\left((\Omega_{\Xlogan/\T_K}^\bullet|_{]U[_{\X, \lambda}})
    \otimes_{\O_{]U[_{\X, \lambda}}} \E\right)|_{]Y[_{\X, \eta}} = 0$
    and $\{ ]U[_{\X, \lambda} \cap\, ]Y[_{\X, \eta},
]U[_{\X, \lambda} \cap [Z]_{\X, \eta}\}$ is an admissible covering of $]U[_{\X, \lambda}$,
we have
$$
\begin{array}{l}
\R g_{\lambda*}\left(\underline{\Gamma}_{]Z[_\X, \eta}^\dag
   \left(\alpha_{]U[_{\X, \lambda}*}
   \left((\Omega_{\Xlogan/\T_K}^\bullet|_{]U[_{\X, \lambda}})
    \otimes_{\O_{]U[_{\X, \lambda}}} \E |_{]U[_{\X, \lambda}}
    \right)\right)|_{]U[_{\X, \lambda}}\right)
    \\
    \hspace*{45mm} \cong\,
    \R g_{\lambda, \eta*}
    \left(
    \underline{\Gamma}_{]Z[_\X, \eta}^\dag
   \left( \alpha_{]U[_{\X, \lambda}*}
   \left((\Omega_{\Xlogan/\T_K}^\bullet|_{]U[_{\X, \lambda}})
    \otimes_{\O_{]U[_{\X, \lambda}}} \E |_{]U[_{\X, \lambda}}
    \right)\right) |_{ ]U[_{\X, \lambda} \cap [Z]_{\X, \eta}} \right).
    \end{array}
$$
Hence, in order to prove the vanishing
$\R g_{K*}\underline{\Gamma}^\dag_{]Z[_\X}
      (j_U^\dag\Omega_{\Xlogan/\T_K}^\bullet \otimes_{j_U^\dag\O_{]X[_\X}} E)
       = 0$, we have only to prove the vanishing
\begin{equation}
\label{coneZ}
\R g_{\lambda, \eta*}
    \left(
    \underline{\Gamma}_{]Z[_\X, \eta}^\dag
   \left(\alpha_{]U[_{\X, \lambda}*}
   \left((\Omega_{\Xlogan/\T_K}^\bullet|_{]U[_{\X, \lambda}})
    \otimes_{\O_{]U[_{\X, \lambda}}} \E |_{]U[_{\X, \lambda}}
    \right)\right)|_{ ]U[_{\X, \lambda} \cap [Z]_{\X, \eta}} \right)
   = 0
\end{equation}
for any $(\lambda, \eta) \in \Lambda_{\underline{\xi}, \underline{\lambda}}$.

\vspace{3mm}

\noindent
$3^\circ$ {\it Reduce to the local computations.}

Let us denote the $1$-dimensional open (resp. closed) unit disk
over $\mathrm{Spm}\, K$ of radius $\eta \in |K^\times|_\Q$ by $D(0, \eta^-)$
(resp. $D(0, \eta^+)$). Since $Z \not\subset D$, we have the lemma below
by the weak fibration theorem \cite[1.3.1, 1.3.2]{Be} (see also \cite[4.3]{BC2}.).

\begin{lemm}\label{wsfib} With the notation as above, we have
\begin{enumerate}
\item There is an admissible covering $\{V_\beta \}_\beta$ of $]T[_{\T}$
such that
$$
       g_K^{-1}(V_\beta) \cap ]Z[_\X\, \cong\, V_\beta \times_{\mathrm{Spm}\, K} D(0, 1^-)
$$
of rigid analytic $K$-spaces, where the coordinate of $D(0, 1^-)$ is $z$
as above under this isomorphism.
\item Under the isomorphism in (1),
$$
      g_{\lambda, \eta}^{-1}(V_\beta)\, \cong\,
      \left(V_\beta \cap ]T \cap U[_{\T, \lambda}\right)\times_{\mathrm{Spm}\, K} D(0, \eta^+)
$$
for any $\lambda, \eta \in |K^\times|_\Q \cap ]0, 1[$ with $\lambda > \eta$.
\end{enumerate}
\end{lemm}

In order to prove \ref{wsfib} (2), the condition $\lambda > \eta$ is needed
because of using $\overline{f}$
for the definition of $]T \cap U[_{\T, \lambda}$.

Let $S = \mathrm{Spm}\, R$ be an integral smooth $K$-affinoid subdomain of
$V_\beta \cap\, ]T \cap U[_{\T, \lambda}$
with a complete $K$-algebra norm $|\, \mbox{-}\, |_R$ on $R$.
Since $R$ is an integral $K$-Banach algebra,
all complete $K$-algebra norms are equivalent \cite[3.8.2, Cor. 4]{BGR}.
In order to prove the vanishing \ref{coneZ}, it is sufficient to prove
the vanishing
$$
\R\Gamma\left(g_{\lambda, \eta}^{-1}(S), \,
\underline{\Gamma}_{]Z[_\X, \eta}^\dag\left(
  \left[\E \overset{\nabla}{\longrightarrow}
  (\Omega_{\Xlogan/\T_K}^1|_{]U[_{\X, \nu}})
  \otimes_{\O_{]U[_{\X, \nu}}} \E\right]\right)\right) =
\R\Gamma\left(g_{\lambda, \eta}^{-1}(S), \,
\underline{\Gamma}_{]Z[_\X, \eta}^\dag\left(
  \left[\E \overset{\partial_{\#}}{\longrightarrow} \E\right]\right)\right) = 0
$$
of hypercohomology for any such $S$ by \ref{wsfib} (2) since
$]T[_\T = ]Z[_\ZZ$ is integral and smooth and $\Omega_{\Xlogan/\T_K}^1$
is a free $\O_{]X[_\X}$-module of rank $1$ generated by $\frac{dz}{z}$.
The hypercohomology above can be calculated by
$$
   \R^q\Gamma\left(g_{\lambda, \eta}^{-1}(S),
\underline{\Gamma}_{]Z[_\X, \eta}^\dag\left(
  \left[\E \overset{\partial_{\#}}{\longrightarrow} \E\right]\right)\right)\,
  \cong\,  H^q\left( \mathrm{Tot}\left[\begin{array}{ccc}
       \Gamma(g_{\lambda, \eta}^{-1}(S), \E) &\rightarrow
       &\underset{\mu \rightarrow \eta^-}{\mathrm{lim}}\,
       \Gamma(g_{\lambda, \eta}^{-1}(S) \cap ]Y[_{\X, \mu}, \E) \\
       \partial_{\#} \downarrow \hspace*{4mm} & &\hspace*{4mm} \downarrow \partial_{\#} \\
              \Gamma(g_{\lambda, \eta}^{-1}(S), \E)
              &\rightarrow
       &\underset{\mu \rightarrow \eta^-}{\mathrm{lim}}\,
       \Gamma(g_{\lambda, \eta}^{-1}(S) \cap ]Y[_{\X, \mu}, *\E)
       \end{array}\right]\right).
$$
Here $\mathrm{Tot}$ means the total complex induced by the commutative bicomplex,
the left top item in the bicomplex is located at degree $(0, 0)$
and the horizontal arrows in the bicomplex are the natural injections.
Indeed, the cohomological functor commutes with filtered direct limits since $g_{\lambda, \eta}$
is an affinoid morphism, and the vanishings
$H^q(g_{\lambda, \eta}^{-1}(S), \E) = 0$ and
$H^q(g_{\lambda, \eta}^{-1}(S) \cap\, ]Y[_{\X, \mu}, \E) = 0$
for $q \geq 1$ hold by Kiehl's Theorem B \cite[2.4]{quasistein}
since $g_{\lambda, \eta}^{-1}(S)$ and
$g_{\lambda, \eta}^{-1}(S) \cap\, ]Y[_{\X, \mu}$ are  affinoid.

More explicitly, the following formula \ref{coneG} holds
when $\E|_{g_{\lambda, \eta}^{-1}(S)}$ is a free
$\O_{g_{\lambda, \eta}^{-1}(S)}$-module of rank $r$.
We will prove the freeness in the next step $4^\circ$. Put $R$-algebras
$$
     \begin{array}{lll}
        \AA_R(\eta) &= &\Gamma(g_{\lambda, \eta}^{-1}(S), \O_{]X[_\X})
        = \left\{ \left. \displaystyle{\sum_{n = 0}^\infty}\, a_nz^n\, \right|\, a_n \in R,\, \,
        |a_n|_R\eta^n \rightarrow 0\, \, \mathrm{as}\,\, n \rightarrow \infty \right\} \\
                \AA_R(\eta^-) &= &\Gamma\left(\underset{\mu < \eta}{\cup}\,
                g_{\lambda, \mu}^{-1}(S), \O_{]X[_\X}\right)
        = \left\{ \left. \displaystyle{\sum_{n = 0}^\infty}\, a_nz^n\, \right|\, a_n \in R,\, \,
        |a_n|_R\mu^n \rightarrow 0\, \, \mathrm{as}\, \, n \rightarrow \infty\, \,
        \mbox{for any}\, \, \mu < \eta\right\} \\
                \RR_R(\eta) &= &\underset{\mu \rightarrow \eta^-}{\mathrm{lim}}\,
                \Gamma(g_{\lambda, \eta}^{-1}(S),
\alpha_{]Y[_{\X, \mu}*}\O_{]Y[_{\X. \mu}}) \\
        &= &\left\{\left. \displaystyle{\sum_{n = -\infty}^\infty}\, a_nz^n\, \right|\,
        a_n \in R,\, \, \begin{array}{l}
               |a_n|_R\eta^n \rightarrow 0\, \, \mathrm{as}\, \, n \rightarrow \infty \\
               |a_n|_R\mu^n \rightarrow 0\, \, \mathrm{as}\, \, n \rightarrow -\infty
              \, \, \mathrm{for\,\, some}\, \mu < \eta
               \end{array} \right\},
        \end{array}
$$
and define a norm on $\mathcal A_R(\eta)$ by $|\sum_n a_nz^n|_{\mathcal A_R(\eta)}
= \mathrm{sup}_n |a_n|\eta^n$.
$\mathcal A_R(\eta), \mathcal A_R(\eta^-)$ and $\mathcal R_R(\eta)$
are independent of the choice of complete
$K$-algebra norms on $R$ since there exist positive real number $\rho_1$ and $\rho_2$
such that $\rho_1|\mbox{-}| \leq |\mbox{-}|' \leq \rho_2|\mbox{-}|$
for equivalent norms $|\mbox{-}|$ and $|\mbox{-}|'$ by \cite[2.1.8, Cor. 4]{BGR}.
Let $\underline{v}$ be a vector of basis of $\Gamma(g_{\lambda, \eta}^{-1}(S), \E)$
over $\AA_R(\eta)$ such that the derivation along $z$
is given by $\partial_{\#}(\underline{v}) = \underline{v}G$
for a matrix $G$ with entries in $\AA_R(\eta)$. Then we have
\begin{equation}
   \begin{array}{lll}
   \label{coneG}
    \R^q\Gamma\left(g_{\lambda, \eta}^{-1}(S),
\underline{\Gamma}_{]Z[_\X, \eta}^\dag\left(
  \left[\E \overset{\partial_{\#}}{\longrightarrow} \E\right]\right)\right)
  &\cong &H^q\left(\mathrm{Tot}\left[
  \begin{array}{ccc}
       \AA_R(\eta)^r &\rightarrow
       &\RR_R(\eta)^r \\
     \partial_{\#} + G \downarrow \hspace*{11mm} &
       &\hspace*{10mm} \downarrow \partial_{\#} + G \\
               \AA_R(\eta)^r &\rightarrow &\RR_R(\eta)^r
       \end{array}
       \right]\right) \\
       &\cong
       &H^q\left(\left[\left(\RR_R(\eta)/\AA_R(\eta)\right)^r\,
       \overset{\partial_{\#} + G}{\longrightarrow} \left(\RR_R(\eta)/\AA_R(\eta)\right)^r
       \right][-1]\right).
    \end{array}
\end{equation}

\vspace{3mm}

\noindent
$4^\circ$ {\it Local classification of logarithmic connections along a smooth divisor.}

\begin{prop} \label{clog} Let $S = \mathrm{Spm}\, R$ be
a smooth integral $K$-affinoid variety,
and let $W = S \times_{\mathrm{Spm}\,K} D(0, \xi^-)$ be a quasi-Stein space over $S$
for some $\xi \in |K^\times|_\Q \cap ]0, 1]$.
Let $\M$ be a locally free $\O_W$-module furnished
with an $R$-derivation $\partial_\# = z\frac{d}{dz} : M \rightarrow M$,
where $M = \Gamma(W, \M)$, such that
\begin{list}{}{}
\item[\mbox{\rm (i)}] for any $\eta \in |K^\times|_\Q \cap\, ]0, \xi[$,
if $W_\eta = S \times_{\mathrm{Spm}\,K} D(0, \eta^+)$ is an affinoid subdomain
of $W$ and if $||\, \mbox{-}\, ||$
is a Banach $\mathcal A_R(\eta)$-norm on $M_\eta = \Gamma(W_\eta, \M)$, then
$||\frac{1}{n!}\prod_{j=0}^{n-1}\, (\partial_\# - j)(e)||\mu^n \rightarrow 0\,
(n \rightarrow \infty)$ for any $e \in M_\eta$ and $0 < \mu < 1$,
and
\item[\mbox{\rm (ii)}] any difference of exponents of $(\M, \partial_\#)$
along $z = 0$ is neither a $p$-adic Liouville number nor a non-zero integer.
\end{list}

\noindent
Then there are a projective $R$-module $L$ of finite type
furnished with a linear $R$-operator $N : L \rightarrow L$
such that
$||\frac{1}{n!}\prod_{j=0}^{n-1}\, (N - j)(e)||\mu^n \rightarrow 0\,
(n \rightarrow \infty)$ for any $e \in L$ and $0 < \mu < 1$,
where $||\, \mbox{-}\, ||$ is a Banach $R$-norm on $L$,
and an isomorphism $(\M, \partial_\#) \cong (\O_W \otimes_R L, \partial_{\#N})$
in which the $R$-derivation $\partial_{\# N}$ on $\O_W \otimes_R L$ is defined by
$\partial_{\#N}(a \otimes e) = \partial_\#(a) \otimes e + a \otimes N(e)$.
\end{prop}

If $\M$ is a free $\O_W$-module in the proposition above, then
the assertion is a part of the Christol's transfer theorem \cite[Thm. 2]{transfer}
and its generalization in \cite{BC1}. The Christol's transfer theorem is in the case
where $R$ is a field $K$. By the argument in \cite[4.1]{BC1}, the transfer
theorem also works on an integral $K$-affinoid algebra $R$.
A part means that we consider solutions not in meromorphic functions
but only in holomorphic functions. When $M$ is free, one has a formal
matrix solution by the hypothesis that any difference of exponents
is not an integer except $0$, and then all entries are contained in
$\AA_R(\xi^-)$ because of the conditions (i) and (ii).

\begin{lemm}\label{free} Let $R$ be an integral $K$-affinoid algebra.
\begin{enumerate}
\item There exists a finite injective morphism $T_l \rightarrow R$ of $K$-affinoid
algebras from a free Tate $K$-algebra $T_l$ of some dimension $l$.
\item Suppose furthermore that $R$ is Cohen-Macaulay.
Then, for any finite injective morphism $T_l \rightarrow R$ of $K$-affinoid
algebras, $R$ is projective of finite type
over $T_l$. Moreover, if $M$ is a projective $R$-module of finite type,
then $M$ is free over $T_l$.
\end{enumerate}
\end{lemm}

\begin{proof} (1) The assertion is the Noether normalization theorem
\cite[6.1.2 Cor. 2]{BGR}.

(2) Since $T_l$ is regular and $R$ is Cohen-Macaulay,
$R$ is projective over $T_l$ by \cite[25.16]{nagata_local}.
If $M$ is a projective $R$-module of finite type,
then $M$ is also projective of finite type over $T_l$, hence
$M$ is free over $T_l$ by \cite[6.5]{kedlaya_full-faithful}.
\end{proof}

With the notation as in \ref{clog},
let us fix a finite injective morphism $T_l \rightarrow R$ of $K$-affinoid
algebras \ref{free} (1). Considering the norm on $R$ which is defined by the maximum
of norms of tuples under an identity $R \cong T_l^m$ by \ref{free} (2),
we regard $M_\eta$ as an $\AA_{T_l}(\eta)[\partial_{\#}]$-module
by the natural finite injective morphism $\AA_{T_l}(\eta) \rightarrow \AA_R(\eta)$
of $K$-affinoid algebras for $\eta \in |K^\times|_\Q \cap\, ]0, \xi[$.
Moreover, $\AA_{T_l}(\eta)[\partial_{\#}]$-module $M_\eta$
satisfies the hypothesis  in \ref{clog} (see \ref{remexp-6})
and $M_{\eta}$ is a free $\AA_{T_l}(\eta)$-module \ref{free} (2).
Fix a basis $\underline{v}$ of $M_\eta$ over $\AA_{T_l}(\eta)$ and
let $G_\eta$ be a matrix with entries in $\AA_{T_l}(\eta)$
such that $\partial_\#(\underline{v}) = \underline{v}G_\eta$.
By applying a generalization of Christol's transfer theorem (as we explain after \ref{clog}),
there is an invertible matrix $Y$ with entries in
$\AA_{T_l}(\eta{}^-)$ such that
\begin{equation}\label{eq}
    \partial_\#\,Y + G_\eta Y = YG_\eta(0),
\end{equation}
where $G_\eta(0) = G_\eta\, (\mathrm{mod}\, z\AA_{T_l}(\eta))$ is
a matrix with entries in $T_l$. Then there is
a free $T_l$-module $L_\eta$
with a $T_l$-linear homomorphism $N_\eta$ defined by the matrix $G_\eta(0)$
such that $(M_\eta \otimes_{\AA_{T_l}(\eta)} \AA_{T_l}(\eta^-), \partial_{\#}) \cong
(\AA_{T_l}(\eta^-) \otimes_{T_l} L_\eta, \partial_{\#N_\eta})$.
If we put $H^0(M_\eta) =
\mathrm{ker}(\partial_\# : M_\eta \rightarrow M_\eta)$,
then $H^0(M_\eta) \cong \mathrm{ker}(N_\eta : L_\eta \rightarrow L_\eta)$.

\begin{lemm}\label{freetate} With the notation as above, the followings hold.
\begin{enumerate}
\item The pair $(L_\eta, N_\eta)$ is independent of the choices of
$\eta  \in |K^\times|_\Q \cap\, ]0, \xi[$ up to canonical isomorphisms.
Moreover, $(M, \partial_{\#}) \cong
(\AA_{T_l}(\xi^-) \otimes_{T_l} L_\eta, \partial_{\#N_\eta})$
for any $\eta$.
\item If we put $H^0(M)
= \mathrm{ker}(\partial_\# : M \rightarrow M)$,
then the natural $R$-homomorphism $H^0(M) \rightarrow H^0(M_\eta)$
(not only the $T_l$ structure) induced by
the restriction is an isomorphism.
\end{enumerate}
\end{lemm}

\begin{proof} (1) For $\eta' \leq \eta$, there is an invertible matrix $Q$ with entries in
$\AA_{T_l}(\eta')$ such that $\partial_{\#}Q + G_{\eta'}(0)Q = QG_\eta(0)$
by the restriction. Since none of the differences of exponents is an integer except $0$,
$Q$ is an invertible matrix with entries in $T_l$.
Hence the pair is independent of the choices of $\eta$.
Note that
$\{ W_\eta \}_{\eta \in |K^\times|_\Q \cap ]0, \xi[}$ is an affinoid covering
of the quasi-Stein space $W$ and
$M$ is the projective limit of $M_\eta\, (\eta \in |K^\times|_\Q \cap ]0, \xi[)$.
Therefore, the assertion holds.

(2) follows from (1).
\end{proof}

\begin{lemm}\label{splittingR} Let $R$ be an integral domain over
a field $\mathbb Q_p$ with the field $F$ of fractions, and let
$(L, N)$ be a pair such that $L$ is a free $R$-module of rank $r$ and $N : L \rightarrow L$
is an $R$-linear endomorphism. Suppose that $e_1, \cdots, e_s$ are distinct eigenvalues of
$N \otimes F$ with multiplicities $m_1, \cdots, m_s$, respectively,
such that $e_1, \cdots, e_s$ are contained in $\mathbb Z_p$
and let $\varphi_N(x) = (x - e_1)^{m_1}\cdots(x-e_s)^{m_s} \in \mathbb Z_p[x]$
the characteristic polynomial of $N$.
If we put $L(e_i) = \varphi_i(N)L$ where $\varphi_i(x) = \varphi_N(x)/(x - e_i)^{m_i}$,
then $L$ is a direct sum of $R$-submodules $L(e_1), \cdots, L(e_s)$ of $L$
such that all eigenvalues of $N|_{L(e_i)} \otimes F$ are $e_i$ for any $i$.
Such a decomposition is unique.
\end{lemm}

\begin{lemm}\label{splitting} With the notation in \ref{clog},
let $e_1, \cdots, e_s$ be distinct exponents of $(M, \partial_{\#})$ along $z = 0$. Then
$M$ is a direct sum of $\AA_R(\xi^-)[\partial_{\#}]$-submodules
$M(e_1), \cdots, M(e_s)$ of $M$ such that all exponents of $(M(e_i), \partial_{\#})$
are $e_i$ for any $i$.
\end{lemm}

\begin{proof} With the notation in \ref{freetate} and \ref{splittingR},
take a free $T_l$-module $L$
of finite type
furnished with an $T_l$-linear homomorphism $N$ such that $(M, \partial_{\#}) \cong
(\AA_{T_l}(\xi^-) \otimes_{T_l} L, \partial_{\#N})$.
Since $L(e_i)$ is a direct summand of the free $T_l$-module $L$, $L(e_i)$ is free.
Put $M(e_i) = (\AA_{T_l}(\xi^-) \otimes_{T_l} L(e_i), \partial_{\#N|_{L(e_i)}})$.
Then $M$ is a direct sum of $M(e_1), \cdots, M(e_s)$
as $\AA_{T_l}(\xi^-)[\partial_{\#}]$-modules.
Since any $\AA_{T_l}(\xi^-)[\partial_{\#}]$-homomorphism between
$M(e_i)$ and $M(e_j)$ for $i \ne j$ is a zero map, $M(e_i)$
is an $\mathcal A_R(\xi^-)[\partial_{\#}]$-module for all $i$.
Hence, the decomposition is the desired one.
\end{proof}

\begin{lemm}\label{covering} Let $S = \mathrm{Spm}\, R$
be a $K$-affinoid variety, $W = S \times_{\mathrm{Spm}\,K} D(0, \xi^+)$
for some $\xi \in |K^\times|_\Q$, and let $\M$ be a locally free $\O_W$-module.
Then there exist a finite affinoid
covering $\{ S_i \}$ of $S$ and a real number $\xi' \in |K^\times|_\Q \cap\, ]0, \xi]$
such that, if $W_{S_i, \xi'}$ denotes an affinoid subdomain $S_i \times D(0, \xi'{}^+)$
of $W$, then $\M|_{W_{S_i, \xi'}}$ is a free $\O_{W_{S_i, \xi'}}$-module for all $i$.
\end{lemm}

\begin{proof} Since $\M/z\M$ is regarded as a locally free $\O_S$-module,
there is a finite affinoid covering $\{ S_i \}$ of $S$ such that
$(\M/z\M)|_{S_i}$ is a free $\O_{S_i}$-module for all $i$.
Since $W_i = S_i \times_{\mathrm{Spm}\,K} D(0, \xi^+)$ is an affinoid,
$\M/z\M$ is generated by
$\Gamma(W_i, \M)$ by Kiehl's Theorems A and B \cite[2.4]{quasistein}.
Let $v_1, \cdots, v_r \in \Gamma(W_i, \M)$
be elements whose reductions form a basis of $(\M/z\M)|_{S_i}$ over $\O_{S_i}$.
The support of $\M|_{W_i}/(v_1, \cdots, v_r)$ is an analytic closed subset
of $W_i$ which does not intersect
with the closed subspace defined by
$z = 0$. Since $\M$ is locally free, there is a real number $\xi'_i \in |K^\times|_\Q \cap ]0, \xi[$ such that
$\M|_{S_i \times_{\mathrm{Spm}\,K} D(0, \xi'_i{}^+)}$ is free and is generated by
$v_1, \cdots, v_r$ because of the maximum modulus principle
\cite[6.2.1, Prop.4]{BGR}. Then it is enough to take
$\xi' = \mathrm{min}_i\, \xi'_i$.
\end{proof}

\noindent
{\it Proof of \mbox{\rm \ref{clog}}.} We may assume that any exponent
of $\M$ along $z = 0$ is $0$
by \ref{splitting} and by twisting an object of rank $1$ with a suitable exponent. We may
also assume that $\M|_{W_\xi'}$ is a free $\O_{W_{\xi'}}$-module for
some $\xi' \in |K^\times|_\Q \cap\, ]0, \xi[$ by \ref{covering}.
By applying the transfer theorem \ref{clog} for the free cases
with the conditions (i) and (ii) ,
if one takes an $\eta \in |K^\times|_\Q \cap ]0, \xi'[$,
then there is a free $R$-module $L$
furnished with an $R$-linear operator $N : L \rightarrow L$
such that $\beta_\eta :
(\M, \partial_\#)|_{W_\eta} \riso (\O_{W_\eta} \otimes_R L, \partial_{\# N})$.
Denote the dual of $\M$ by $(\M^\vee, -\partial_\#)$. Then we have a
natural commutative diagram
$$
     \begin{array}{ccc}
    \mbox{Hom}_{\O_W[\partial_\#]}(\M, \O_W \otimes_R L)
    &\longrightarrow
    &\mbox{Hom}_{\O_W[\partial_\#]}(\M|_{W_\eta}, \O_{W_\eta} \otimes_R L) \\
    \cong \downarrow \hspace*{4mm} & &\hspace*{4mm} \downarrow \cong \\
      H^0(M^\vee \otimes_R L) &\riso &H^0(M_\eta^\vee \otimes_R L),
      \end{array}
$$
where the vertical arrows are isomorphisms since $\M$ is locally free and
the bottom horizontal arrow is an isomorphism by \ref{freetate} (2) since
all difference of exponents of
$(\M^\vee \otimes_R L, -\partial_\# \otimes 1 + 1 \otimes \partial_N)$
along $z = 0$ are $0$.

Let $\beta : (\M, \partial_\#) \rightarrow (\O_W \otimes_R L, \partial_{\#N})$ be an
$\O_W[\partial_\#]$-homomorphism corresponding to $\beta_\eta$
via the isomorphisms above. We will prove that $\beta$ is an isomorphism.
In the case where $R$ is a field, i.e., $d = 1$,
$\beta$ is an isomorphism since the support of an $\AA_R(\xi^-)[\partial_\#]$-module, which
is finitely generated over $\AA_R(\xi^-)$, is either $W$ or one point $z = 0$
by B\'ezout property of $\AA_R(\xi^-)$ \cite[4.6]{crew-finite}.
Let us return to the case of general $R$. For a maximal ideal $x$ of $R$,
the induced homomorphism $\beta\, (\mathrm{mod}\, x)$ is an isomorphism
by the case where $R$ is a field.
Hence, $\beta$ is an isomorphism around $x \times_{\mathrm{Spm}\,K} D(0, \xi^-)$
by Nakayama's lemma.
Since both sides of $\beta$ is coherent, $\beta$ is an isomorphism
\cite[9.4.2, Corollary 7]{BGR}.
\hspace*{\fill} $\Box$\

\vspace{3mm}

\noindent
$5^\circ$ {\it The vanishing \ref{coneZ} in special cases : any difference
of exponents is neither a $p$-adic Liouville number nor an integer except $0$.}

Let us first suppose that (ii) in \ref{clog} and $c = 0$ for the exponents
along $z = 0$ by $0^\circ$.

\begin{lemm}\label{estimate} With the notation in \ref{splittingR}, the followings hold.
\begin{enumerate}
\item Let $j$ be an integer. Then there is a monic polynomial $g_j(x) \in \Z_p[x]$
of degree $r-1$ such that $(N - j)g_j(N) + \varphi_N(j)I_L = 0$.
Here $I_L$ is the identity of $L$.
\item If all of $e_1, \cdots, e_s$ are neither $p$-adic Liouville numbers
nor positive integers, then $(N - j)$ is invertible and, for any $0 < \eta < 1$,
$|\varphi_N(j)^{-1}|\eta^j \rightarrow 0$ as $j \rightarrow \infty$
\end{enumerate}
\end{lemm}

Take $(\lambda, \eta) \in \Lambda_{\underline{\xi}, \underline{\lambda}}$
such that $\lambda \geq \lambda_m$ and $\eta < \xi_m$ for some $m$.
Then the restriction $(\E, \partial_{\#})$ on $S \times_{\mathrm{Spm} K} D(0, \xi_m^-)$
for an integral smooth $K$-affinoid $S = \mathrm{Spm}\, R$
in $V_\beta \cap\, ]Z \cap U[_{\ZZ, \lambda_m}$
satisfies the assumption of \ref{clog}  by the overconvergent condition in $2^\circ$.
Considering an admissible affinoid covering of $S$,
we may assume that there is a basis of $\Gamma(g_{\lambda, \eta}^{-1}(S), \E)$
over $\AA_R(\eta)$
such that $G$ is a matrix with entries in $R$.

Since any proper values of $G$ is not a positive integer,
$\partial_{\#} + G$ is injective on $\left(\RR_R(\eta)/\AA_R(\eta)\right)^r$.
Since any proper values of $G$ is neither a $p$-adic Liouville nor a positive integer,
$\partial_{\#} + G$ is surjective on $\left(\RR_R(\eta)/\AA_R(\eta)\right)^r$.
Indeed, with the notation in \ref{estimate} (1), $\partial_{\#} + G$ maps
$-\sum_{j=1}^\infty\, \varphi_G(j)^{-1}g_j(G)\underline{a}_jz^{-j}$
to  $\sum_{j=1}^\infty\, \underline{a}_jz^{-j}$ and
$\sum_{j=1}^\infty\, \varphi_G(j)^{-1}g_j(G)\underline{a}_jz^{-j}$ is contained
in $\left(\RR_R(\eta)/\AA_R(\eta)\right)^r$ by \ref{estimate} (2).
Hence, the cohomology groups in \ref{coneG} vanish for any $q$ and
it implies the vanishing \ref{coneZ}.

\vspace{3mm}

\noindent
$6^\circ$ {\it The vanishing \ref{coneZ} in general cases : any difference
of exponents is not a $p$-adic Liouville number.}

Let us suppose the conditions (a) in \ref{Nobuorigid} and $c = 0$ for the exponents
along $z = 0$ by $0^\circ$.

\begin{prop}\label{intexp} With the notation as in \ref{clog},
we assume the conditions (i) in \ref{clog},
(a) in \ref{Nobuorigid} and $c = 0$ for exponents of $(\M, \partial_\#)$
along $z = 0$.
Suppose that $\M|_{W_\eta}$ is locally free
for some $\eta \in |K^\times|_\Q \cap ]0, \xi[$.
Then there is a locally free $\O_W$-submodule $\M'$ of $\M$ which is stable under
$\partial_\#$ such that (1) $(\M', \partial_\#)$ satisfies the conditions (i) and (ii) in \ref{clog}
such that none of exponents along $z = 0$ is a positive integer,
(2) the support of $\M/\M'$ is included in the closed subset defined by $z = 0$
and it is a free $O_S$-module of finite rank,
and (3) the induced homomorphism
$\overline{\partial}_\# : \M/\M' \rightarrow \M/\M'$ is an isomorphism.
\end{prop}

\begin{lemm}\label{expch} Let $R$ be an integral $K$-affinoid
and let $\eta \in |K^\times|_\Q$.
Suppose that $M$ is an $\AA_R(\eta)$-module of rank $r$ furnished
with an $R$-derivation $\partial_\# = z\frac{d}{dz} : M \rightarrow M$
such that $e_1, \cdots, e_s$ are distinct exponents of $(M, \partial_{\#})$ along $z = 0$
with multiplicities $m_1, \cdots, m_s$, respectively.
\begin{enumerate}
\item
There exists a basis $\underline{v}$ of $M$ such that,
if $G$ is a square matrix with entries in $\AA_R(\xi)$ defined by
$\partial_\#(\underline{v}) = \underline{v}G$, then
$G(0) = \left(\begin{array}{ccc} G_1(0) & &0 \\ &\ddots & \\ 0& &G_s(0) \end{array}\right)$
by square matrices $G_1(0), \cdots, G_s(0)$ of degree $m_1, \cdots, m_s$,
respectively, with entries in $R$ such that
all eigenvalues of $G_i(0)$ are $e_i$ for any $i$.
\item Let $\underline{v}_i$ be a part of the basis as in (1) such that
it corresponds to the $i$-th direct summand modulo $z$, that is,
$\partial_{\#}(\underline{v}_i)\, \equiv\, \underline{v}_iG_i(0)\,
(\mathrm{mod}\, z\AA_R(\eta))$. Let $M'$ be an $\AA_R(\eta)$-submodule of $M$
generated by $z\underline{v}_1, \underline{v}_2, \cdots, \underline{v}_s$.
Then $M'$ is stable under $\partial_{\#}$ whose exponents are
$e_1 + 1, \cdots, e_s$ with multiplicities $m_1, \cdots, m_s$, respectively.
Moreover, $M/M'$ is
a free $R$-module of rank $m_1$, and,
if $e_1 \ne 0$, then the induced $R$-homomorphism
$\overline{\partial}_{\#} : M/M' \rightarrow M/M'$ is an isomorphism.
\end{enumerate}
\end{lemm}

\begin{proof} (1) follows from \ref{splittingR}.

(2) The stability follows from (1).
If we denote the matrix which represents the derivation of $M'$ by $G'$, then
$$
G' = P^{-1}z\frac{d}{dz}P + P^{-1}GP\, \equiv\,
\left(\begin{array}{cccc} G_1(0) + I_{m_1}
& & &\ast \\ &G_2(0) & & \\ & &\ddots & \\ 0& & &G_s(0) \end{array}\right)\,
(\mathrm{mod}\, z\AA_R(\eta))
$$
for $P = \left(\begin{array}{cc} zI_{m_1} &0 \\ 0&I_{r-m_1} \end{array}\right)$.
Here $I_t$ is the identity matrix of degree $t$.
The induced $R$-homomorphism
$\overline{\partial}_{\#} : M/M' \rightarrow M/M'$ is given by the matrix $G_1(0)$.
\end{proof}

\vspace{3mm}

\noindent
{\it Proof of \mbox{\rm \ref{intexp}}.} We use the induction on
the largest integral difference of exponents and its multiplicity.
By \ref{free} we may assume that $\mathcal M|_{W_\eta}$ is free
for some $\eta \in |K^\times|_\Q \cap ]0, \xi[$.
We have an $\O_{W_\eta}$-submodule $\M_\eta'$ of $\M|_{W_\eta}$
such that exponents are improved by \ref{expch}. Indeed, we apply
\ref{expch} to an exponents which is neither a positive integer nor $0$
because of the condition $c = 0$.
Since the support of $\M|_{W_\eta}/\M_\eta'$
is included in $z = 0$, one can glue $\M_\eta'$ and $\M|_{W \setminus \{ z = 0 \}}$.
Hence, the induction works.
\hspace*{\fill} $\Box$

\vspace{3mm}

We use the same notation in $5^\circ$. Considering an admissible affinoid covering of $S$,
we may assume that $\E|_{g_{\lambda, \mu}^{-1}(S)}$ is free
for some $\mu \in |K^\times|_\Q \cap\, ]0, \xi_m]$ by \ref{covering} and, then, we
can apply \ref{intexp}. Let $\E'$ be a locally free
$\O_{g_{\lambda, \xi_m}^{-1}(S)}$-submodule of $\E|_{g_{\lambda, \xi_m}^{-1}(S)}$
which is stable under $\partial_{\#}$
such that it satisfies the condition (1), (2)
and (3) in \ref{intexp}. Now we calculate the difference of
the local computation of cohomology
between $\E$ and $\E'$ by the module version of the second form of \ref{coneZ}.
If $E_\eta = \Gamma(g_{\lambda, \eta}^{-1}(S), \E)$
and $E' _\eta = \Gamma(g_{\lambda, \eta}^{-1}(S), \E)$,
then $E' \otimes \RR_R(\eta) = E \otimes \RR_R(\eta)$
by the condition (2) of the support of $\E/\E'$.
The difference is calculated by the complex
$$
      \mathrm{Tot}\left[\begin{array}{ccc}
             E_\eta' &\rightarrow &E_\eta \\
       \partial_{\#} \downarrow \hspace*{4mm} & &\hspace*{4mm} \downarrow \partial_{\#} \\
             E_\eta' &\rightarrow &E_\eta
       \end{array}\right]\, \cong\,
       \left[E_\eta/E_\eta'\, \overset{\partial_{\#}}{\longrightarrow}\, E_\eta/E_\eta'\right],
$$
and it is $0$ by (3). Hence, the vanishing \ref{coneZ} for $\E$
follows from the vanishing for $\E'$ by $5^\circ$.

\vspace{3mm}

This completes the proof of Proposition \ref{section}.
\end{proof}

\vspace{3mm}

\noindent
{\it Proof of Theorem \ref{Nobuorigid}.} By the same reason
of $0^\circ$ in the proof of \ref{section}, we may assume $c = 0$
and have only to prove the vanishing
$\R g_{K *}\underline{\Gamma}_{]Z[_\X}^\dag
        (j_U^\dag\Omega_{\Xlogan/\T_K}^\bullet \otimes_{j_U^\dag\O_{]X[_\X}} E)
         = 0$.

By the \v{C}ech spectral sequence
the problem of
the vanishing is local on $X$ and $U$ as in $1^\circ$ in the proof of \ref{section}.
We may assume that $\X$ is affine,
$D$ is defined by a single equation $f = 0$ in $X$ for some $f \in \Gamma(\X, \O_\X)$,
and there is a relatively local coordinate $z_1, z_2, \cdots, z_d \in \Gamma(\X, \O_\X)$
of $\X$ over $\T$
such that each irreducible component $\ZZ_i$ of the relatively strict normal crossing divisor
$\ZZ = \cup_{i = 1}^s\, \ZZ_i$ is defined
by $z_i = 0$. Let us denote by $Z_i$ (resp. $Y_i$)
the closed subscheme of $X$ defined by $z_i = 0$ (resp. the complement of $Z_i$ in $X$).

Let us define $]U[_{\X, \lambda}$ (resp. $]Y[_{\X, \lambda}$, resp.
$[Z_i]_{\X, \lambda}$) by the same manners
as in $2^\circ$ of the proof of \ref{section} (resp. replacing $\ZZ$, $Z$
by $\ZZ_i$, $Z_i$).

By the hypothesis of $(E, \nabla)$ there exist
a strict neighborhood $]U[_{\X, \nu}$ of $]U[_\X$ in $]X[_\X$
for some $\nu \in |K^\times|_\Q \cap\, ]0, 1[$
and a locally free $\O_{]U[_{\X, \nu}}$-module $\E$
furnished with a logarithmic connection
$\nabla : \E \rightarrow (\Omega_{\Xlogan/\T_K}^1|_{]U[_{\X, \nu}})
\otimes_{\O_{]U[_{\X, \nu}}} \E$
such that $j_U^\dag(\E, \nabla) = (E, \nabla)$, which
satisfies the overconvergent condition \ref{logovcon}.

\vspace*{3mm}

\noindent
$7^\circ$ {\it Induction on the number $s$
of irreducible components of the strict normal crossing divisor $Z$.}

If $s = 0$, then the assertion is trivial.
Put $Z' = \cup_{i = 2}^s\, Z_i$. Applying the natural exact sequence
$$
      0 \longrightarrow \underline{\Gamma}_{]Z_1[_\X}^\dag(\H)
      \longrightarrow \underline{\Gamma}_{]Z[_\X}^\dag(\H) \longrightarrow
      \underline{\Gamma}_{]Z'[_\X}^\dag(j_{Y_1}^\dag\H) \longrightarrow 0
$$
for a sheaf $\H$ of abelian groups on $]X[_\X$ (see the proof of \cite[2.1.7]{Be}),
we have a triangle
$$
     \begin{array}{c}
         \R g_{K *}\underline{\Gamma}_{]Z'[_\X}^\dag
         (j_{Y_1 \cap U}^\dag\Omega_{\Xlogan/\T_K}^\bullet
         \otimes_{j_{Y_1 \cap U}^\dag\O_{]X[_\X}} j_{Y_1 \cap U}^\dag E) \\
         +1 \swarrow \hspace*{20mm} \nwarrow \\
         \R g_{K *}\underline{\Gamma}_{]Z_1[_\X}^\dag
         (j_U^\dag\Omega_{\Xlogan/\T_K}^\bullet
         \otimes_{j_U^\dag\O_{]X[_\X}} E)\, \,
         \longrightarrow\, \,
           \R g_{K *}\underline{\Gamma}_{]Z[_\X}^\dag
         (j_U^\dag\Omega_{\Xlogan/\T_K}^\bullet \otimes_{j_U^\dag\O_{]X[_\X}} E).
      \end{array}
$$
Hence we have only to prove the vanishing
$$
      \R g_{K *}\underline{\Gamma} ^\dag _{]Z_1[_\X}
         (j_U^\dag\Omega_{\Xlogan/\T_K}^\bullet \otimes_{j_U^\dag\O_{]X[_\X}} E)
         = 0
$$
by the induction on $s$. If $Z_1 \subset D$, the vanishing is trivial.
Hence, we may assume that $Z_1$ is not included in $D$.

\vspace*{3mm}

\noindent
$8^\circ$ {\it Reduction to the case of sections.}

Let us denote the formal affine space of relative dimension $r$
over $\T$ by $\widehat{\mathbb A}_\T^r$.
By our hypothesis there is a commutative diagram
\begin{equation}\label{fib1}
     \begin{array}{ccccc}
     \ZZ_1 &\longrightarrow &\X \\
     \downarrow & &\downarrow \\
     \widehat{\A}_{\T}^{d-1} &\longrightarrow &\widehat{\A}_{\T}^d &\longrightarrow
     &\widehat{\A}_{\T}^{d-1}
     \end{array}
\end{equation}
of formal $\V$-schemes such that the vertical arrow $\X \rightarrow \widehat{\A}_{\T}^d$,
which is \'etale, (resp. $\ZZ_1 \rightarrow \widehat{\A}_{\T}^{d-1}$)
is induced by $z_1, \cdots, z_d$ (resp. $z_2, \cdots, z_d$)
and the composite of bottom arrows is the identity. Since the diagonal morphism
$\Delta : \ZZ_1 \rightarrow \ZZ_1 \times_{\widehat{\A}_{\T}^{d-1}} \ZZ_1$
is \'etale and  a closed immersion,
$\widetilde{\X} = \ZZ_1 \times_{\widehat{\A}_{\T}^{d-1}} \X \setminus
(\ZZ_1 \times_{\widehat{\A}_{\T}^{d-1}} \ZZ_1 \setminus \Delta(\ZZ_1))$
is an open formal subscheme of $\ZZ_1 \times_{\widehat{\A}_{\T}^{d-1}} \X$.
Let us now consider a commutative diagram
\begin{equation}\label{fib2}
     \begin{array}{ccccc}
           & &\widetilde{\X} & & \\
     &\overset{\Delta}{\nearrow} \hspace*{4mm}&\downarrow &\overset{h}{\searrow} & \\
     \ZZ_1 &\overset{\Delta}{\longrightarrow} &\ZZ_1 \times_{\widehat{\A}_\V^{d-1}} \X &
     \underset{\mathrm{pr}_2}{\longrightarrow} &\X \\
     = \downarrow \hspace*{3mm} & &\downarrow \\
     \ZZ_1 &\underset{\mathrm{pr}_1}{\longleftarrow}
     &\ZZ_1 \times_{\widehat{\A}_\V^{d-1}} \widehat{\A}_\V^d
     \end{array}
\end{equation}
of formal $\T$-schemes, and define $h : \widetilde{\X} \rightarrow \X$
(resp. $\widetilde{g}_1 : \widetilde{\X} \rightarrow \ZZ_1$, resp.
$\widetilde{g}' : \ZZ_1 \rightarrow \T$,
resp. $\widetilde{g} = \widetilde{g}' \circ \widetilde{g}_1$)
as in the diagram
(resp. by the composition, resp. the canonical morphism).
We identify $\Delta(\ZZ_1)$ (resp. $\Delta(Z_1)$) with $\ZZ_1$ (resp. $Z_1$),
and denote the special fiber of $\widetilde{\X}$
(resp. the complement of $Z_1$, resp. the inverse image of $U$ by $h$)
by $\widetilde{X}$ (resp. $\widetilde{Y}_1$, resp. $\widetilde{U}$).
$\ZZ_1$ is a smooth divisor over $\T$ and
note that, \'etale locally, $h^{-1}(\ZZ)$ is a relatively normal crossing divisor.
$\widetilde{\X}_K^{\#}$
denotes the formal $\V$-scheme with a logarithmic structure over $\T_K$
which is induced by the logarithmic structure of $\Xlogan$,
and $\Omega_{\widetilde{\X}_K^{\#}/\T_K}^1$
denotes the sheaf of logarithmic K\"ahler differentials on $\widetilde{\X}_K^{\#}$
over $\T_K$.
Then $h_K^*\Omega_{\Xlogan/\T_K}^\bullet\, \cong\,
\Omega_{\widetilde{\X}_K^{\#}/\T_K}^\bullet$.

Let us define $]\widetilde{U}[_{\widetilde{\X}, \lambda}$
(resp. $]\widetilde{Y} _1[_{\widetilde{\X}, \lambda}$, resp.
$[Z_1]_{\widetilde{\X}, \lambda}$) by the same manners
as in $2^\circ$ of the proof of \ref{section}.

\begin{lemm}\label{fib} With the notation as above, we have
\begin{enumerate}
\item $h_K^{-1}(]Z_1[_\X) = ]Z_1[_{\widetilde{\X}}$.
\item The restriction of $h_K$ gives an isomorphism $]Z_1[_{\widetilde{\X}} \riso ]Z_1[_\X$.
\item Under the isomorphism in (2),
$$
      ]\widetilde{U}[_{\widetilde{\X}, \lambda} \cap [Z_1]_{\widetilde{\X}, \eta} \riso
      ]U[_{\X, \lambda} \cap [Z_1]_{\X, \eta}
$$
for any $\lambda, \eta \in |K^\times|_\Q \cap ]0, 1[$.
\end{enumerate}
\end{lemm}

\begin{proof}
Since $(\ZZ_1 \times_{\widehat{\A}_{\T}^{d-1}} \ZZ_1 \setminus \Delta(\ZZ_1))$
is removed, we get (1).
The other assertion (2) (resp. (3)) follow from \cite[1.3.1]{Be} and the fact that
$h$ is \'etale (resp. and $Z_1 \not\subset D$).
\end{proof}

\begin{prop}\label{suppiso} With the notation as above, we have the followings.
\begin{enumerate}
\item If $\H$ is a sheaf of Abelian groups on $]\widetilde{X}[_{\widetilde{\X}}$, then
$$
       \R h_{K*}\underline{\Gamma}_{]Z_1[_{\widetilde{\X}}}^\dag(\H)
       \cong h_{K*}\underline{\Gamma}_{]Z_1[_{\widetilde{\X}}}^\dag(\H).
$$
\item Let $\AA$ and $\B$ be a sheaf of rings on $]X[_\X$
and $]\widetilde{X}[_{\widetilde{\X}}$, respectively,
with a morphism $h_K^{-1}\AA \rightarrow \B$ of rings
such that $\AA|_{]Z_1[_\X} \riso \B|_{]Z_1[_{\widetilde{\X}}}$ under the isomorphism
in \ref{fib} (2). If $\H$ is an $\AA$-module, then the adjoint map
$$
       \underline{\Gamma}_{]Z_1[_{\widetilde{\X}}}^\dag(\H)
       \rightarrow h_{K*}\underline{\Gamma}_{]Z_1[_{\widetilde{\X}}}^\dag
       (\B \otimes_{h_K^{-1}\AA} h_K^{-1}\H).
$$
is an isomorphism of $\AA$-modules.
\end{enumerate}
\end{prop}

\begin{proof} Let us define a functor
$$
        \underline{\Gamma}_{]Z_1[_{\widetilde{X}}, \eta}^\dag(\H)
        = \mathrm{ker}\left(\H \rightarrow
            \underset{\mu \rightarrow \eta^-}{\mathrm{lim}}\,
            \alpha_{]\widetilde{Y}_1[_{\widetilde{\X}, \mu}*}
            (\H|_{]\widetilde{Y}_1[_{\widetilde{\X}, \mu}})
            \right)
$$
as same as in $2^\circ$ of the proof of \ref{section},
where $\alpha_{]\widetilde{Y}_1[_{\widetilde{\X}, \mu}} :
]\widetilde{Y}_1[_{\widetilde{\X}, \mu} \rightarrow ]\widetilde{X}[_{\widetilde{\X}}$
is the canonical open immersion.
Then the same of \ref{supp} and \ref{suppdag} hold.

(1) Since $\underline{\Gamma}_{]Z_1[_{\widetilde{X}}, \eta}^\dag
(\H)|_{]Y_1[_{\widetilde{X}, \eta}} = 0$, we have
$\R^q h_{K*}\underline{\Gamma}_{]Z_1[_{\widetilde{X}}, \eta}^\dag(\H) = 0$
for any $q \geq 1$ by \ref{fib} (2). Because the cohomological functor $\R^q h_{K*}$
commutes with filtered inductive limits
by the quasi-compactness and quasi-separateness of $h_K$, we have
$$
      \R^q h_{K*}\underline{\Gamma}_{]Z_1[_{\widetilde{\X}}}^\dag(\H)\,
      \cong\, \R^q h_{K*}\left(\underset{\eta \rightarrow 1^-}{\mathrm{lim}}\,
      \underline{\Gamma}_{]Z_1[_{\widetilde{\X}}, \eta}^\dag(\H)\right)\,
      \cong\, \underset{\eta \rightarrow 1^-}{\mathrm{lim}}\, \R^q h_{K*}
      \underline{\Gamma}_{]Z_1[_{\widetilde{\X}}, \eta}^\dag(\H) = 0
$$
for any $q \geq 1$ by \ref{supp}.

(2) Since $\H|_{[Z_1]_{\X, \eta}} \riso
(\B \otimes_{h_K^{-1}\AA} h_K^{-1}\H)|_{[Z_1]_{\widetilde{\X}, \eta}}$,
the assertion follows from \ref{supp} and \ref{fib}.
\end{proof}

Let $(\widetilde{E}, \widetilde{\nabla})$ be the inverse image of $(E, \nabla)$ by $h _k$,
i.e.,
$$
     \begin{array}{l}
        \widetilde{E} = h_K^* E
             = j_{\widetilde{U}}^\dag\O_{]\widetilde{X}[_{\widetilde{\X}}}
                 \otimes_{h_K^{-1}(j_U^\dag\O_{]X[_\X})} h_K^{-1}E \\
        \widetilde{\nabla} : \widetilde{E} \rightarrow
j_{\widetilde{U}}^\dag\Omega_{\widetilde{\X}_K^{\#}/\T_K}^1
\otimes_{j_{\widetilde{U}}^\dag\O_{]\widetilde{X}[_{\widetilde{\X}}}} \widetilde{E},
     \end{array}
$$
where $\widetilde{\nabla}$ is the induced $\O_{]T[_\T}$-linear connection
by $\nabla$ because of the \'etaleness of $h$. We also denote the induced basis of
$\Omega_{\widetilde{\X}_K^{\#}/\T_K}^1$ by $\frac{dz_1}{z_1}, \cdots, \frac{dz_s}{z_s},
dz_{s+1}, \cdots, dz_d$ and the dual basis of derivations by
$z_1\frac{\partial}{\partial z_1}, \cdots, z_s\frac{\partial}{\partial z_s},
\frac{\partial}{\partial z_{s+1}}, \cdots, \frac{\partial}{\partial z_d}$.

\begin{prop}\label{conv}
\begin{enumerate}
\item If we put $(\widetilde{\E}, \widetilde{\nabla}) = h_K^*(\mathcal E, \nabla)$,
then the natural morphism
$j_{\widetilde{U}}^\dag(\widetilde{\E}, \widetilde{\nabla})
\rightarrow (\widetilde{E}, \widetilde{\nabla})$
is an isomorphism.
\item The derivation $\widetilde{\partial}_{\#1} = \nabla(z_1\frac{\partial}{\partial z_1})$
on $\widetilde{\E}$ satisfies the overconvergent condition \ref{logoct}.
\end{enumerate}
\end{prop}

\begin{proof} (1) easily follows from the fact $\mathcal E$ is locally free.

(2) It is enough to check the overconvergent condition
for $\mathrm{pr}_{2 K}^*(\E, \nabla)$ along $z_1 = 0$. Fix
a complete $K$-algebra norm on the affinoid algebra
associated to $]X[_\X$. Then one can take a
contractive complete $K$-algebra norm on the affinoid algebra
associated to
$]Z_1 \times_{\A_k^{d-1}} X[_{\ZZ_1 \times_{\widehat{\A}_{\T}^{d-1}} \X}$
\cite[6.1.3, Prop. 3]{BGR},
The induced norms $||\mbox{-}||_\X$ on $\Gamma(]U[_{\X, \lambda}, \E)$
and $||\mbox{-}||_{\ZZ_1 \times \X}$ on
$\Gamma\left(\mathrm{pr}_{2 K}^{-1}(]U[_{\X, \lambda}), \mathrm{pr}_{2 K}^*\E\right)$
satisfy the inequality $||e||_{\ZZ_1 \times \X} \leq ||e||_\X$
for any $e \in \Gamma(]U[_{\X, \lambda}, \E)$.
The overconvergent condition for $\mathrm{pr}_{2 K}^*(\E, \nabla)$ along $z_1 = 0$
follows from the inequality.
\end{proof}

\begin{rema} The connection $(\widetilde{\E}, \widetilde{\nabla})$ satisfies the
overconvergent condition \ref{logovcon}. It should be called a log-isocrystal
on $\widetilde{U}^{\#}/\T_K$ overconvergent along $\widetilde{D}$.
\end{rema}

Since $(j_U^\dag\O_{]X[_\X})|_{]Z_1[_\X} \riso
(j_{\widetilde{U}}^\dag\O_{]\widetilde{X}[_{\widetilde{\X}}})|_{]Z_1[_{\widetilde{\X}}}$,
we have
$$
   \begin{array}{lll}
      \R g_{K *}\underline{\Gamma}_{]Z_1[_\X}^\dag(
      j_U^\dag\Omega_{\Xlogan/\T_K}^\bullet \otimes_{j_U^\dag\O_{]X[_\X}} E)
    &\cong
    &\R g_{K *}\left(h_{K*}\underline{\Gamma}_{]Z_1[_{\widetilde{\X}}}^\dag
    (j_{\widetilde{U}}^\dag\Omega_{\widetilde{\X}_K^{\#}/\T_K}^\bullet
\otimes_{j_{\widetilde{U}}^\dag\O_{]\widetilde{X}[_{\widetilde{\X}}}} \widetilde{E})\right) \\
    &\cong
    &\R g_{K *}\R h_{K*}\underline{\Gamma}_{]Z_1[_{\widetilde{\X}}}^\dag
    (j_{\widetilde{U}}^\dag\Omega_{\widetilde{\X}_K^{\#}/\T_K}^\bullet
\otimes_{j_{\widetilde{U}}^\dag\O_{]\widetilde{X}[_{\widetilde{\X}}}} \widetilde{E}) \\
    &\cong
    &\R\widetilde{g}_{K *}\underline{\Gamma}_{]Z_1[_{\widetilde{\X}}}^\dag
    (j_{\widetilde{U}}^\dag\Omega_{\widetilde{\X}_K^{\#}/\T_K}^\bullet
\otimes_{j_{\widetilde{U}}^\dag\O_{]\widetilde{X}[_{\widetilde{\X}}}} \widetilde{E})
    \end{array}
$$
by \ref{suppiso}. Hence we have only to prove the vanishing
$$
      \R\widetilde{g}_{K *}\underline{\Gamma}_{]Z_1[_{\widetilde{\X}}}^\dag
    (j_{\widetilde{U}}^\dag\Omega_{\widetilde{\X}_K^{\#}/\T_K}^\bullet
\otimes_{j_{\widetilde{U}}^\dag\O_{]\widetilde{X}[_{\widetilde{\X}}}} \widetilde{E})
= 0.
$$

\vspace*{3mm}

\noindent
$9^\circ$ {\it An argument of Gauss-Manin type.}

Let $\Omega_0^q$ (resp. $\Omega_1^q$)
be a free $\O_{]\widetilde{X}[_{\widetilde{\X}}}$-submodule
of $\Omega_{\widetilde{\X}_K^{\#}/\T_K}^q$
generated by wedge products of $\frac{dz_2}{z_2} , \cdots, \frac{dz_s}{z_s},
dz_{s+1}, \cdots, dz_d$ (resp. $\frac{dz_1}{z_1} \wedge \omega$ for $\omega \in
\Omega_{\widetilde{\X}_K^{\#}/\T_K}^{q-1}$).
Then $\Omega_0^q \riso \Omega_1^{q+1}$
by $\omega \mapsto \frac{dz_1}{z_1} \wedge \omega$.
Define
\begin{equation}
\label{GaussManin}
      \begin{array}{l}
      \widetilde{\nabla}_0 = \sum_{i=2}^s\, \frac{dz_i}{z_i} \otimes \partial_{\#i}
+ \sum_{i=s+1}^d\, dz_i \otimes \partial_i : \widetilde{E} \rightarrow
    j_{\widetilde{U}}^\dag\Omega_0^1
\otimes_{j_{\widetilde{U}}^\dag\O_{]\widetilde{X}[_{\widetilde{\X}}}} \widetilde{E} \\
\widetilde{\nabla}_1 = \mathrm{id} \otimes \partial_{\#1} :
j_{\widetilde{U}}^\dag\Omega_0^q
\otimes_{j_{\widetilde{U}}^\dag\O_{]\widetilde{X}[_{\widetilde{\X}}}}
\widetilde{E} \rightarrow  j_{\widetilde{U}}^\dag\Omega_1^q
\otimes_{j_{\widetilde{U}}^\dag\O_{]\widetilde{X}[_{\widetilde{\X}}}} \widetilde{E},
\end{array}
\end{equation}
where $\mathrm{id}$ is the identity of $j_{\widetilde{U}}^\dag\Omega_0^q$.
The definition of $\widetilde{\nabla}_0$ and $\widetilde{\nabla}_1$ is independent
of the choices of local parameters $z_1, z_2, \cdots, z_d$
of $\X$ over $\T$ as above.
Then the exterior power of $j_{\widetilde{U}}^\dag\Omega_0^1$ induces
a complex $(j_{\widetilde{U}}^\dag\Omega_0^\bullet
\otimes_{j_{\widetilde{U}}^\dag\O_{]\widetilde{X}[_{\widetilde{\X}}}} \widetilde{E},
\widetilde{\nabla}_0)$ and there is an isomorphism
\begin{equation}\label{gmquasi}
      j_{\widetilde{U}}^\dag\Omega_{\widetilde{\X}_K^{\#}/\T_K}^\bullet
\otimes_{j_{\widetilde{U}}^\dag\O_{]\widetilde{X}[_{\widetilde{\X}}}} \widetilde{E}\,
\riso
\left[(j_{\widetilde{U}}^\dag\Omega_0^\bullet
\otimes_{j_{\widetilde{U}}^\dag\O_{]\widetilde{X}[_{\widetilde{\X}}}} \widetilde{E},
\widetilde{\nabla}_0)\, \overset{\widetilde{\nabla}_1}{\longrightarrow}\,
(j_{\widetilde{U}}^\dag\Omega_1^\bullet
\otimes_{j_{\widetilde{U}}^\dag\O_{]\widetilde{X}[_{\widetilde{\X}}}} \widetilde{E},
\frac{dz_1}{z_1} \wedge \widetilde{\nabla}_0)\right]
\end{equation}
of complexes of $\O_{]T[_\T}$-modules.
Note that $\widetilde{\nabla}_1$ is the induced relative connection
$\widetilde{E} \rightarrow
j_{\widetilde{U}}^\dag\Omega_{\widetilde{\X}_K^{\#}/\ZZ_{1K}^{\#}}^1
\otimes_{j_{\widetilde{U}}^\dag\O_{]\widetilde{X}[_{\widetilde{\X}}}} \widetilde{E}$
by $\widetilde{\nabla}$.

One can easily see $(\widetilde{E}, \widetilde{\nabla}_1)$ satisfies the
hypothesis (a) and (b) along $z_1 = 0$ in \ref{Nobuorigid} by \ref{fib}
and the overconvergent condition in \ref{section}, so that
$$
   \R\widetilde{g}_{1K*}\left(\left[j_{\widetilde{U}}^\dag\Omega_0^q
\otimes_{j_{\widetilde{U}}^\dag\O_{]\widetilde{X}[_{\widetilde{\X}}}} \widetilde{E}\,
\overset{\widetilde{\nabla}_1}{\longrightarrow}\,
j_{\widetilde{U}}^\dag\Omega_1^{q+1}
\otimes_{j_{\widetilde{U}}^\dag\O_{]\widetilde{X}[_{\widetilde{\X}}}}
\widetilde{E}\right]\right)
= 0
$$
for any $q$ by \ref{section}. Hence,
$$
     \R\widetilde{g}_{K *}\underline{\Gamma}_{]Z_1[_{\widetilde{\X}}}^\dag
    (j_{\widetilde{U}}^\dag\Omega_{\widetilde{\X}_K^{\#}/\T_K}^\bullet
\otimes_{j_{\widetilde{U}}^\dag\O_{]\widetilde{X}[_{\widetilde{\X}}}} \widetilde{E})
=  \R\widetilde{g}'_{K *}\R\widetilde{g}_{1K *}
\underline{\Gamma}_{]Z_1[_{\widetilde{\X}}}^\dag
    (j_{\widetilde{U}}^\dag\Omega_{\widetilde{\X}_K^{\#}/\T_K}^\bullet
\otimes_{j_{\widetilde{U}}^\dag\O_{]\widetilde{X}[_{\widetilde{\X}}}} \widetilde{E})
= 0.
$$

\vspace{3mm}

This completes the proof of \ref{Nobuorigid}.
\hspace*{\fill} $\Box$

\vspace{3mm}

\begin{prop}\label{overconv} With the notation in \ref{Nobuorigid}, we assume furthermore
that $g : \X \rightarrow \T$ factors through an irreducible component $\ZZ_1$ of $\ZZ$ by
a smooth morphism $g_1 : \X \rightarrow \ZZ_1$ over $\T$
such that the composite $g_1 \circ i_ 1 : \ZZ_1 \rightarrow \ZZ_1$ of the closed immersion
$i_1 : \ZZ_1 \rightarrow \X$ and $g_1$ is the identity of $\ZZ_1$
and that the inverse image of the relatively
strict normal crossing divisor $\ZZ_1' = \cup_{i = 2}^s\, \ZZ_1 \cap \ZZ_i$ of $\ZZ_1$
by $g_1$ is $\cup_{i = 2}^s\, \ZZ_i$.
Let  $E$ be a log-isocrystal on $U^{\#}/\T_K$ overconvergent along $D$.
Then, for any nonnegative integer $m$, $g_{1K*}\widetilde{\nabla}_0$
(resp. $g_{1K*}(\frac{dz_1}{z_1} \wedge \widetilde{\nabla}_0)$)
in \ref{GaussManin} induces an integrable logarithmic $\O_{]T[_\T}$-connection of
the locally free $j_{Z_1 \cap U}^\dag\O_{]Z_1[_{\ZZ_1}}$-module $g_{1K*}(E(m\ZZ_1)/E)$
(resp. $g_{1K*}(j_U^\dag\Omega_{\Xlogan/\ZZ_{1K}^{\#}}^1
\otimes_{j_U^\dag\O_{]X[_\X}} E(m\ZZ_1)/E)$) on $(\ZZ_{1K}, \ZZ_{1K}')/\T_K$
which satisfies the overconvergent condition as a log-isocrystal
on $(Z_1 \cap U)^{\#}/\T_K$ overconvergent along $Z_1 \cap D$.

Suppose furthermore that $Z_1 \not\subset D$ and
that $E$ satisfies the conditions (a) and (b) in \ref{Nobuorigid}. Then
\begin{equation}
\label{overconv-iso}
   \R g_{1K*}\underline{\Gamma}^\dag_{]Z_1[_\X}
   (j_U^\dag\Omega_{\Xlogan/\ZZ_{1K}^{\#}}^\bullet
   \otimes_{j_U^\dag\O_{]X[_\X}} E)\,
   \cong\, \left[g_{1K*}(E(m\ZZ_1)/E)\,
   \displaystyle{\mathop{\longrightarrow}^{g_{1K*}\nabla}}\,
   g_{K*}(j_U^\dag\Omega_{\Xlogan/\ZZ_{1K}^{\#}}^1
   \otimes_{j_U^\dag\O_{]X[_\X}} E(m\ZZ_1)/E)\right][-1]
\end{equation}
and $g_{1K*}(E(m\ZZ_1)/E)$
(resp. $g_{1K*}(j_U^\dag\Omega_{\Xlogan/\ZZ_{1K}^{\#}}^1
\otimes_{j_U^\dag\O_{]X[_\X}} E(m\ZZ_1)/E)$) also satisfies the same conditions (a) and (b)
for any $m \geq \max\{ e\, |\, e\, \mbox{is a positive integral exponent
      of $\nabla$ along $Z_1$} \}\, \cup\, \{ 0 \}$.
\end{prop}

\begin{proof} The locally freeness has been already proved in the part $0^\circ$
of the proof of \ref{section}. From the definition
of $\widetilde{\nabla}_0$ in \ref{GaussManin}, it induces an integrable connection.
Since $\ZZ_1$ is a section of $\X$ over $\T$,
a complete $K$-algebra norm of subaffinoid variety of $]Z_1[_{\ZZ_1}$
induces a complete $K$-algebra norm of certain subaffinoid variety of $]X[_\X$
Hence the logarithmic connections on $g_{1K*}(E(m\ZZ_1)/E)$ and
$g_{1K*}(j_U^\dag\Omega_{\Xlogan/\ZZ_{1K}^{\#}}^1
\otimes_{j_U^\dag\O_{]X[_\X}} E(m\ZZ_1)/E)$ satisfy the overconvergent condition.
Their exponents along $Z_i$ are $m$ copies of those
of $E$ by the definition of $\widetilde{\nabla}_0$ for $i \ne 1$.
Therefore, the conditions (a) and (b) also hold.
\end{proof}

\begin{exam}\label{example} Let $\X = \widehat{\mathbb P}_\V^1
\times_{\mathrm{Spf}\, \V} \widehat{\mathbb P}_\V^1$
be a formal projective scheme over $\S = \mathrm{Spf}\, \V$
with homogeneous coordinates $(x_0, x_1), (y_0, y_1)$,
let $\ZZ_1$ (resp. $\ZZ_2$)
be a divisor defined by $x_1 = 0$ (resp. $y_1 = 0$) in $\X$
and put $\ZZ = \ZZ_1 \cup \ZZ_2$ and $\X^{\#} = (\X, \ZZ)$.
Let $U$ be an open formal subscheme of $X$ defined by $x_0 \ne 0$
and $y_0 \ne 0$,
let $z_1 = x_1/x_0, z_2 = y_1/y_0$ be the lift of coordinates of $U$,
and let $D$ be a closed subscheme of $X$ defined by $x_ 0 = 0$ or $y_0 = 0$.
For integers $e > 0$ and $h \geq 0$,
we define a log-isocrystal $E$ on $U^{\#}/\S_K$ of rank $2$ overconvergent
along $D$ $(E = j_U^\dag\O_{]X[_\X}v_1 \oplus j_U^\dag\O_{]X[_\X}v_2$) by
$$
      \nabla(v_1, v_2)
      = (v_1, v_2)\left(\begin{array}{cc} e &z_2^h \\ 0 &e \end{array}\right)\frac{dz_1}{z_1}
      + (v_1, v_2)\left(\begin{array}{cc} 0 &0 \\ 0 &h \end{array}\right)\frac{dz_2}{z_2}
$$
for some strict neighborhood of $]U[_\X$ in $]X[_\X$.
Indeed, since the exponents along $Z_1$ (resp. $Z_2$) are $e$ and $e$ (resp. $0$ and $h$),
the logarithmic connection satisfies the overconvergent condition
and is overconvergent along $D$. Moreover, it satisfies the conditions (a) and (b)
in \ref{Nobuorigid}. If $g_1 : \X \rightarrow \ZZ_1$ is the second projection
(note that the coordinate of $\ZZ_1 \cap \U$ is $z_2$), then
$$
   \begin{array}{l}
       \R g_{1K*}\underline{\Gamma}^\dag_{]Z_1[_\X}
       (j_U^\dag\Omega_{\Xlogan/\ZZ_{1K}^{\#}}^\bullet
   \otimes_{j_U^\dag\O_{]X[_\X}} E)\, \\
   \hspace*{20mm} \cong\, \left[g_{1K*}\left(E(m\ZZ_1)/E\right)\,
   \displaystyle{\mathop{\longrightarrow}^{g_{K*}(\frac{dz_1}{z_1}\otimes\partial_{\#1})}}\,
   g_{1K*}(j_U^\dag\Omega_{\Xlogan/\ZZ_{1K}^{\#}}^1
   \otimes_{j_U^\dag\O_{]X[_\X}} E(m\ZZ_1)/E)\right][-1]
   \end{array}
$$
for $m \geq e$
by \ref{section}. Hence $\R^qg_{1K*}\underline{\Gamma}^\dag_{]Z_1[_\X}
       (j_U^\dag\Omega_{\Xlogan/\ZZ_{1K}^{\#}}^\bullet
   \otimes_{j_U^\dag\O_{]X[_\X}} E) = 0$ for $q \ne 1, 2$ and
$$
     \R^qg_{1K*}\underline{\Gamma}^\dag_{]Z_1[_\X}
       (j_U^\dag\Omega_{\Xlogan/\ZZ_{1K}^{\#}}^\bullet
   \otimes_{j_U^\dag\O_{]X[_\X}} E)\, \cong\,
   \left\{ \begin{array}{ll}
   j_{Z_1 \cap U}^\dag\O_{]Z_1[_{\ZZ_1}}z_1^{-e}v_1 &\mathrm{if}\, \, q = 1, \\
   \left(j_{Z_1 \cap U}^\dag\O_{]Z_1[_{\ZZ_1}}/
   z_2^hj_{Z_1 \cap U}^\dag\O_{]Z_1[_{\ZZ_1}}\right)z_1^{-e}v_1 \oplus
   j_{Z_1 \cap U}^\dag\O_{]Z_1[_{\ZZ_1}}z_1^{-e}v_2 &\mathrm{if}\, \, q = 2.
   \end{array} \right.
$$
Therefore, $\R^2g_{1K*}\underline{\Gamma}^\dag_{]Z_1[_\X}
       (j_U^\dag\Omega_{\Xlogan/\ZZ_{1K}^{\#}}^\bullet
   \otimes_{j_U^\dag\O_{]X[_\X}} E)$ is not always locally free.
By \ref{gmquasi} and using a spectral sequence,
the dimensions of total cohomology groups are as follow:
$$
      \mathrm{dim}_K\, \mathbb H^q\left(]X[_\X, \underline{\Gamma}^\dag_{]Z_1[_\X}
       (j_U^\dag\Omega_{\Xlogan/\S_{K}}^\bullet
   \otimes_{j_U^\dag\O_{]X[_\X}} E)\right)
   = \left\{
   \begin{array}{ll}
   1 &\mathrm{if}\, q = 1, \\
   2\,  (\mbox{resp.}\, 3) &\mathrm{if}\, q = 2\, (\mbox{resp. and $h = 0$}),  \\
   1\, (\mbox{resp.}\, 2) &\mathrm{if}\, q = 3\, (\mbox{resp. and $h = 0$}),  \\
   0 &\mathrm{if}\, q \ne 1, 2, 3.
   \end{array} \right.
$$
\end{exam}

\subsection{Cohomological operations of arithmetic log-$\D$-modules}
\label{nota-1.2}
We will need later some basic properties on cohomological operations such as direct images and
extraordinary inverses images by morphisms of smooth log-formal $\V$-schemes.
We follow here Berthelot's procedure on the study of arithmetic $\D$-modules.
We recall that
in order to come down from the case of formal schemes to the case of schemes
(the latter case is technically much better),
the strategy of Berthelot was to develop a notion of {\it quasi-coherence} complexes on formal schemes (see \cite{Beintro2}).
We extend naturally below (see \ref{nota-hatD} and \ref{qcoh-stepII})
this Berthelot's notion of quasi-coherence in the case of formal log-schemes.
This will allow us for instance to check the transitivity of direct images and
extraordinary inverse images (see \ref{properties}), which is essential for our work.
\bigskip

First, let us fix some notation that we will keep in this section.
Let $\T$ be a smooth formal scheme over $\V$, $h$ : $\X' \rightarrow \X$ be a morphism of smooth formal schemes over $\T$,
let $\ZZ$ (resp. $\ZZ'$) be a relatively strict normal crossing divisor of $\X$ (resp. $\X '$) over $\T$
such that $h ^{-1} (\ZZ) \subset \ZZ'$,
let $D$ (resp. $D'$) be a divisor of $X$ (resp. $X'$) such that $h ^{-1} (D) \subset D'$.
We denote by $U:=X \setminus D$, $ \X ^{ \#}:= (\X,\ZZ)$, $ \X ^{\prime \#}:= (\X',\ZZ')$,
$u$ : $\X ^{ \#} \rightarrow \X$, $g ^\#$ : $\X ^\# \rightarrow \T$
the canonical morphisms,
and $h ^\#$ : $\X ^{\prime \#}\rightarrow \X ^{\#}$ the induced morphism of smooth formal log-schemes over $\T$.
We denote by $h _i ^\#$ : $X _i ^{\prime \#}\rightarrow X _i ^{\#}$ the reduction of $h ^\#$ modulo $\pi ^{i+1}$.
Berthelot has constructed in \cite[4.2.3]{Be1} the $\O _{X _i}$-algebra $\B _{X _i} ^{(m)} (D)$ which is endowed with
a compatible structure of left $\D ^{(m)} _{X _i ^\#} $-module.
We recall that when $f\in \O _{X_i} $ is a lifting of an equation of $D$ in $X$, then
$\B _{X _i} ^{(m)} (D) = \O _{X_i} [T]/(f ^{p ^{m+1}} T - p)$.
By abuse of notation,
we pose $\D ^{(m)} _{X _i ^\#} (D):= \B _{X _i} ^{(m)} (D) \otimes _{\O _{X _i}} \D ^{(m)} _{X ^\# _i}$,
$\D ^{(m)} _{X _i ^{\prime \#}}(D'):= \B _{X '_i} ^{(m)} (D') \otimes _{\O _{X '_i}} \D ^{(m)} _{X _i^{\prime \#}}$.
For any $\O _{X _i}$-module $\M _i$, we pose
$\M _i (Z _i) : = \O _{X _i} (Z _i) \otimes _{\O _{X _i}} \M _i$, where
$\O _{X _i} (Z _i):= \mathcal{H} om _{\O _{X _i}} ( \omega _{X _i} ,
\omega _{X _i ^\#})$. When $\M _i$ is even a $\D ^{(m)} _{X _i ^\#} (D)$-module then
$\M _i (Z _i)$ has a canonical structure of $\D ^{(m)} _{X _i ^\#} (D)$-module (see \cite[5.1]{caro_log-iso-hol}).

We check by functoriality that the sheaf $\B _{X '_i} ^{(m)} (D ') \otimes _{\O _{X '_i}} h_i ^* (\D ^{(m)} _{X _i ^\#})$ is
a $(\D ^{(m)} _{X _i ^{\prime \#}}(D '), h _i ^{-1}\D ^{(m)} _{X _i ^{\#}}(D ) )$-bimodule.
This bimodule will be denoted by
$\D ^{(m)} _{X _i ^{\prime \#}\rightarrow X _i ^{\#}} (D',D)$.
Also, we get a $(h _i ^{-1}\D ^{(m)} _{X _i ^{\#}}(D ) , \D ^{(m)} _{X _i ^{\prime \#}}(D '))$-bimodule
with:
$\D ^{(m)} _{X _i ^{\#}\leftarrow X _i ^{\prime \#}} (D,D')
:=
\B _{X '_i} ^{(m)} (D ') \otimes _{\O _{X _i}} (
\omega _{X _i ^{\prime \#}} \otimes _{\O _{X ' _i}} h_i ^{*l} (\D ^{(m)} _{X _i ^\#} \otimes _{\O _{X _i}} \omega ^{-1} _{X _i ^\#})$,
where the symbol `$l$' means that to compute the inverse image by $h _i$ we choose the left structure of
left $\D ^{(m)} _{X _i ^\#}$-module of
$\D ^{(m)} _{X _i ^\#} \otimes _{\O _{X _i}} \omega ^{-1} _{X _i ^\#}$.

Before proceeding, let us state the following lemma that we will need
to define the local cohomological functor with support in a closed subscheme (see \ref{locfunclog}).
\begin{lemm}
  \label{HomDDhash}
  Let $\E$ be a $\D ^{(m)} _{X _i }$-module and $\FF$ be a $\D ^{(m)} _{X _i ^\#}$-module.
  Then $\mathcal{H} om _{\O _{X _i}} (\E, \FF)$ is endowed with unique structure of
  $\D ^{(m)} _{X _i ^\#}$-module such that, for any morphism $\phi$ of
  $\mathcal{H} om _{\O _{X _i}} (\E, \FF)$, for any section $x$ on $\E$, we have
  \begin{equation}
\label{HomDDhash-formula}
  (\underline{\partial} ^{<\underline{k}>}_\# \cdot \phi )(x)
=
\sum _{\underline{h} \leq \underline{k}}
(-1) ^{|\underline{h}|}
\left \{ \begin{smallmatrix}   \underline{k} \\   \underline{h} \\ \end{smallmatrix} \right \}
\underline{t} ^{\underline{h}} \underline{\partial} ^{<\underline{k}-\underline{h}>}_\# \cdot
(\phi
(  \underline{\partial} ^{<\underline{h}>} \cdot x)
).
  \end{equation}
\end{lemm}
\begin{proof}
We denote by $\PP ^n _{X_i ^\# ,\,(m)}$ the $m$-PD-envelop of order $n$ of the diagonal immersion of $X _i ^\#$,
${d} ^{n} _{1*} \PP ^n _{X_i ^\#,\,(m)}$ (resp. ${d} ^{n} _{2*} \PP ^n _{X_i ^\#,\,(m)}$)
the induced $\O_{X _i}$-algebra for the left (resp. right) structure.
  Using the isomorphisms
  $\mathcal{H} om _{\O _{X_i}} (\E,\FF ) \otimes _{\O _{X_i}} {d} ^{n} _{i*} \PP ^n _{X_i ^\#,\,(m)}
\riso
\mathcal{H} om _{\PP ^n _{X_i ,\,(m)}}
(\E \otimes _{\O _{X_i}} {d} ^{n} _{i*} \PP ^n _{X_i ,\,(m)},
\FF \otimes _{\O _{X_i}} {d} ^{n} _{i*} \PP ^n _{X_i ^\#,\,(m)}),$
we pose $\epsilon _n ^{ \mathcal{H} om _{\O _{X_i}} (\E, \FF) }: =
\mathcal{H} om _{\PP ^n _{X_i ,\,(m)} } ((\epsilon _n ^{\E })^{-1}, \epsilon _n ^{\FF })$, where
$\epsilon _n ^{\E }$ is the $m$-PD-stratification of $\E$ with respect to $X _i ^\# /S _i $ corresponding to its
structure of $\D ^{(m)} _{X _i }$-module
and
$\epsilon _n ^{\FF }$ is the $m$-PD-stratification of $\FF$ with respect to $X _i ^\# /S _i $ corresponding to its
structure of $\D ^{(m)} _{X ^\# _i }$-module (see \cite[1.8]{caro_log-iso-hol})

To compute $(\epsilon _n ^{\E })^{-1}$ and $\epsilon _n ^{\FF }$, we use respectively
\cite[2.3.2.3]{Be1} (notice that this formula is not any more true with logarithmic structure)
and \cite[1.8.1]{caro_log-iso-hol}.
\end{proof}

\begin{empt}
[Quasi-coherence, step I]
\label{nota-hatD}
  Let $\B$ be a sheaf of $\O _{\X}$-algebras, $\E \in D ^- ( \B \overset{ ^\mathrm{r}}{} )$,
  $\FF \in D ^- (\overset{ ^\mathrm{l}}{} \B  )$, i.e. $\E$ (resp. $\FF$) is a bounded above complex of right (resp. left) $\D$-modules.
  We pose: $\B _i :=\B / \pi ^{i+1} \B$,
  $\E _i := \E \otimes ^\L _{\B} \B _i$, $\FF _i := \B _i\otimes ^\L _{\B} \FF$,
  $\E \widehat{\otimes} ^\L _{\B} \FF:=
  \R \underset{\underset{i}{\longleftarrow}}{\lim} \E _i \otimes ^\L _{\B _i} \FF _i$.

$\bullet$ We say that $\E$ (resp. $\FF$) is {\it $\B$-quasi-coherent} if
the canonical morphism
$\E \rightarrow \E \widehat{\otimes} ^\L _{\B} \B$
(resp. $\FF \rightarrow \B \widehat{\otimes} ^\L _{\B} \FF$)
is an isomorphism.
We denote by $D ^- _\mathrm{qc} (\overset{ ^*}{} \B  )$
(resp. $D ^\mathrm{b} _\mathrm{qc} (\overset{ ^*}{} \B  )$)
the full subcategory of quasi-coherent complexes of
$D ^- (\overset{ ^*}{} \B  )$
(resp. $D ^\mathrm{b} (\overset{ ^*}{} \B  )$),
where `$*$' is either `$\mathrm{r}$' or `$\mathrm{l}$'.

$\bullet$ We pose
$\smash{\widehat{\D}} _{\X ^\#} ^{(m)}(D):=
\underset{\underset{i}{\longleftarrow}}{\lim}\, \D_{X ^\# _i} ^{(m)}(D)$.
Since $\smash{\widehat{\D}} _{\X ^\#} ^{(m)}(D)$ is a flat
$\smash{\widehat{\B}} _{\X} ^{(m)}(D)$-module (for the right or the left structures),
a complex of $D ^\mathrm{*} (\overset{ ^*}{} \smash{\widehat{\D}} _{\X ^\#} ^{(m)}(D))$
is $\smash{\widehat{\D}} _{\X ^\#} ^{(m)}(D)$-quasi-coherent (and in particular when $\X ^\#$ is replaced by $\X$) if and only if it is
$\smash{\widehat{\B}} _{\X} ^{(m)}(D)$-quasi-coherent.
Then, the forgetful functor
$D ^\mathrm{*} (\overset{ ^*}{} \smash{\widehat{\D}} _{\X } ^{(m)}(D))
\rightarrow
D ^\mathrm{*} (\overset{ ^*}{} \smash{\widehat{\D}} _{\X ^\#} ^{(m)}(D))$
induces :
$D ^\mathrm{*} _\mathrm{qc}(\overset{ ^*}{} \smash{\widehat{\D}} _{\X } ^{(m)}(D))
\rightarrow
D ^\mathrm{*}_\mathrm{qc} (\overset{ ^*}{} \smash{\widehat{\D}} _{\X ^\#} ^{(m)}(D))$.
Also, it follows from \cite[4.3.3.(i)]{Be1}:
$\smash{\widehat{\B}} _{\X} ^{(m)}(D) \otimes ^\L _{\V} \V /\pi ^{i+1}
\riso
\smash{\widehat{\B}} _{\X} ^{(m)}(D) \otimes  _{\V} \V /\pi ^{i+1}
\riso \smash{\B} _{X _i} ^{(m)}(D) $.
Hence, a complex of
$D ^\mathrm{*} (\overset{ ^*}{} \smash{\widehat{\B}} _{\X} ^{(m)}(D))$
is
$\smash{\widehat{\B}} _{\X} ^{(m)}(D)$-quasi-coherent
if and only if it is
$\O _{\X} $-quasi-coherent,
if and only if it is $\V$-quasi-coherent.

$\bullet$ We get a $(\smash{\widehat{\D}} _{\X ^{\prime \#}} ^{(m)}(D') , h ^{-1} \smash{\widehat{\D}} _{\X ^\#} ^{(m)} (D))$-bimodule by posing
$\smash{\widehat{\D}} _{\X ^{\prime \#}\rightarrow \X ^{\#}} ^{(m)} ( D' , D):=
\underset{\underset{i}{\longleftarrow}}{\lim}\, \D_{X^{\prime \#}_i \rightarrow X ^\# _i} ^{(m)}( D' , D)$.
Also, we have the $(h ^{-1} \smash{\widehat{\D}} _{\X ^\#} ^{(m)} (D),\smash{\widehat{\D}} _{\X ^{\prime \#}} ^{(m)}(D')  )$-bimodule
$\smash{\widehat{\D}} _{\X ^{\#} \leftarrow \X ^{\prime \#}} ^{(m)} ( D , D'):=
\underset{\underset{i}{\longleftarrow}}{\lim}\, \D ^{(m)} _{X _i ^{\#}\leftarrow X _i ^{\prime \#}} (D,D')$.

\end{empt}

\begin{empt}
[Quasi-coherence, step II]
\label{qcoh-stepII}
  Let $\smash{\widehat{\D}} _{\X ^{\#}} ^{(\bullet)}(D) := (\smash{\widehat{\D}} _{\X ^{\#}} ^{(m)}(D))_{m \in \N}$
  be the canonical inductive system.
Localizing twice $D ^{\mathrm{b}} ( \smash{\widehat{\D}} _{\X ^{\#}} ^{(\bullet)}(D))$ (these localizations replace respectively
the foncteur $-\otimes _\Z \Q$ and the inductive limite on the level $m$),
we construct similarly to \cite[4.2.1, 4.2.2]{Beintro2} and \cite[1.1.3]{caro_courbe-nouveau}
a category denoted by
  $\smash{\underset{^{\longrightarrow}}{LD}} ^{\mathrm{b}} _{\Q }
  ( \smash{\widehat{\D}} _{\X ^{\#}} ^{(\bullet)}(D))$.
  Let $\E ^{(\bullet)}=(\E ^{(m)}) _{m \in \N} \in \smash{\underset{^{\longrightarrow}}{LD}} ^{\mathrm{b}} _{\Q}
  ( \smash{\widehat{\D}} _{\X ^{\#}} ^{(\bullet)}(D))$.
As for \cite[4.2.3]{Beintro2} and \cite[1.1.3]{caro_courbe-nouveau},
we say that $\E ^{(\bullet)}$ is quasi-coherent if for any $m $
$\E ^{(m)} $ is $\smash{\widehat{\D}} _{\X ^{\#}} ^{(m)}(D)$-quasi-coherent.
We denote the subcategory of quasi-coherent sheaves by
$\smash{\underset{^{\longrightarrow}}{LD}} ^{\mathrm{b}} _{\Q ,\mathrm{qc}}
  ( \smash{\widehat{\D}} _{\X ^{\#}} ^{(\bullet)}(D))$.
With the second point of \ref{nota-hatD}, we check that the canonical functor:
$  \smash{\underset{^{\longrightarrow}}{LD}} ^{\mathrm{b}} _{\Q }
  ( \smash{\widehat{\D}} _{\X } ^{(\bullet)}(D))
  \rightarrow
\smash{\underset{^{\longrightarrow}}{LD}} ^{\mathrm{b}} _{\Q }
  ( \smash{\widehat{\D}} _{\X ^{\#}} ^{(\bullet)}(D))
$
induces the following one:
$\smash{\underset{^{\longrightarrow}}{LD}} ^{\mathrm{b}} _{\Q ,\mathrm{qc}}
  ( \smash{\widehat{\D}} _{\X ^{\#}} ^{(\bullet)}(D))
  \rightarrow
  \smash{\underset{^{\longrightarrow}}{LD}} ^{\mathrm{b}} _{\Q,\mathrm{qc} }
  ( \smash{\widehat{\D}} _{\X } ^{(\bullet)}(D))$.

\end{empt}

\begin{empt}
[Extraordinary inverse image, direct image, tensor product]

Let $\E ^{(\bullet)} \in \smash{\underset{^{\longrightarrow}}{LD}} ^{\mathrm{b}} _{\Q ,\mathrm{qc}}
  ( \smash{\widehat{\D}} _{\X ^{\#}} ^{(\bullet)}(D))$,
  $\E ^{\prime(\bullet)}\in \smash{\underset{^{\longrightarrow}}{LD}} ^{\mathrm{b}} _{\Q ,\mathrm{qc}}
  ( \smash{\widehat{\D}} _{\X ^{\prime \#}} ^{(\bullet)}(D'))$.
The following functors extend that which were already defined without log-structure.

  \begin{itemize}

\item
The extraordinary inverse image of $\E ^{(\bullet)}$ by $h^\#$ is defined as follows:
\begin{equation}
  \label{defqcextinvimage}
  h ^{\# !} _{D',D} (\E ^{(\bullet)}):=
  ( \smash{\widehat{\D}} _{\X ^{\prime \#}\rightarrow \X ^{\#}} ^{(m)} ( D' , D)
  \widehat{\otimes} ^\L _{h ^{-1} \smash{\widehat{\D}} _{\X ^\#} ^{(m)} (D)}
h ^{-1} \E ^{(m)} [d_{X '/X}]) _{m \in \N }
\in
\smash{\underset{^{\longrightarrow}}{LD}} ^{\mathrm{b}} _{\Q ,\mathrm{qc}}
  ( \smash{\widehat{\D}} _{\X ^{\prime \#}} ^{(\bullet)}(D')).
\end{equation}

\item The direct image by $h ^\#$ of $\E ^{\prime(\bullet)}$ is defined as follows:
\begin{equation}
  \label{defqcdirectimage}
  h ^\# _{D,D'+} (\E ^{\prime(\bullet)}):=
  ( \R h _* ( \smash{\widehat{\D}} _{\X ^{\#} \leftarrow \X ^{\prime \#}} ^{(m)} ( D , D')
  \widehat{\otimes} ^\L _{\smash{\widehat{\D}} _{\X ^{\prime \#}} ^{(m)}(D') }
\E ^{\prime(m)})) _{m \in \N }
\in
\smash{\underset{^{\longrightarrow}}{LD}} ^{\mathrm{b}} _{\Q ,\mathrm{qc}}
  ( \smash{\widehat{\D}} _{\X ^{\#}} ^{(\bullet)}(D)).
\end{equation}

\item  Let $\widetilde{D}$ be a divisor of $X$ containing $D$.
We pose:
\begin{equation}
  \label{defhdagDtilde}
  (\hdag \widetilde{D}, D) (\E ^{(\bullet)}) :=
(\smash{\widehat{\D}} _{\X ^{\#}} ^{(m)}(\widetilde{D})
\widehat{\otimes} ^\L _{\smash{\widehat{\D}} _{\X ^{\#}} ^{(m)}(D) }
\E ^{(m)}) _{m\in \N}\in \smash{\underset{^{\longrightarrow}}{LD}} ^{\mathrm{b}} _{\Q ,\mathrm{qc}}
  ( \smash{\widehat{\D}} _{\X ^{\#}} ^{(\bullet)}(\widetilde{D})).
\end{equation}
We denote by $\mathrm{Forg} _{D, \widetilde{D}}$ :
$\smash{\underset{^{\longrightarrow}}{LD}} ^{\mathrm{b}} _{\Q ,\mathrm{qc}}
  ( \smash{\widehat{\D}} _{\X ^{\#}} ^{(\bullet)}(\widetilde{D}))
  \rightarrow
  \smash{\underset{^{\longrightarrow}}{LD}} ^{\mathrm{b}} _{\Q ,\mathrm{qc}}
  ( \smash{\widehat{\D}} _{\X ^{\#}} ^{(\bullet)}(D))$
  the forgetful functor.

\item  When $D$ or $D'$ are empty, we remove them in the notation. Also,
  when $D ' =h ^{-1} (D)$, we remove $D'$ in the notation.

\end{itemize}

Using the remark \cite[2.3.5.(iii)]{Be1}, we get the isomorphism
in
$\smash{\underset{^{\longrightarrow}}{LD}} ^{\mathrm{b}} _{\Q ,\mathrm{qc}}
  ( \smash{\widehat{\D}} _{\X ^{\#}} ^{(\bullet)}(\widetilde{D}))$:
\begin{equation}
\label{ODdivcohe}
\O _{\X} (\hdag \widetilde{D}) _\Q \smash{\overset{\L}{\otimes}}   ^{\dag}_{\O _{\X} (\hdag D) _\Q}\E ^{(\bullet)}
:=
(\widehat{\B} ^{(m)} _{\X} ( \widetilde{D})  \smash{\widehat{\otimes}} ^\L
_{\widehat{\B} ^{(m)}  _{\X} ( D) } \E ^{(m)}) _{m\in \N}
\riso
(\hdag \widetilde{D}, D) (\E ^{(\bullet)}).
\end{equation}
Since a flat $\D _{X^\# _i} ^{(m)}$-module (resp. a flat $\D _{X _i} ^{(m)}$-module) is also
a flat $\O _{X_i} ^{(m)}$-module, we check that
the functor $(\hdag \widetilde{D}, D)$ commutes
with the forgetful functor
$ \smash{\underset{^{\longrightarrow}}{LD}} ^{\mathrm{b}} _{\Q,\mathrm{qc} }
  ( \smash{\widehat{\D}} _{\X } ^{(\bullet)}(D))
  \rightarrow
\smash{\underset{^{\longrightarrow}}{LD}} ^{\mathrm{b}} _{\Q ,\mathrm{qc}}
  ( \smash{\widehat{\D}} _{\X ^{\#}} ^{(\bullet)}(D))$.
  Hence, by \cite[1.1.8]{caro_courbe-nouveau} and the associativity of tensor products,
  we deduce from \ref{ODdivcohe} that we have a canonical isomorphism:
  $( \hdag \widetilde{D},D) \riso ( \hdag \widetilde{D})  \circ \mathrm{Forg} _{D}$.
  Similarly, if $D _1$ and $D _2$ are two divisors of $X$ then
  $(\hdag D _1) \circ (\hdag D _2) \riso (\hdag D_1 \cup D _2) $ (we have omitted the forgetful functor).
  Then we notice that $(\hdag D _1) $ and $ (\hdag D_1 \cup D _2)$ are canonically isomorphic on
  $\smash{\underset{^{\longrightarrow}}{LD}} ^{\mathrm{b}} _{\Q ,\mathrm{qc}}
  ( \smash{\widehat{\D}} _{\X ^{\#}} ^{(\bullet)}(D_2))$.

\end{empt}

\begin{empt}
  [Local cohomological functor with support in a closed subscheme]
  \label{locfunclog}
Let $\widetilde{X}$ be a closed subscheme of $X$,
$\E ^{(\bullet)},\FF ^{(\bullet)}  \in \smash{\underset{^{\longrightarrow}}{LD}} ^{\mathrm{b}} _{\Q ,\mathrm{qc}}
  ( \smash{\widehat{\D}} _{\X ^{\#}} ^{(\bullet)}(D))$.
Let $\mathfrak{I}_i$ be the ideal of $\O _{X _i}$ defined by $\widetilde{X} \subset X _i$,
$\PP  _{(m) } ( \mathfrak{I} _i)$ the $m$-PD-envelop of $\mathfrak{I} _i$
(resp. $\PP ^n  _{(m) } ( \mathfrak{I} _i)$ the $m$-PD-envelop of order $n$ of $\mathfrak{I} _i$),
$\overline{\mathcal{I}} _i ^ {\{ n\}_{(m)}}$ its $m$-PD filtration (see \cite[1.3--4]{Be1}).
From \cite[4.4.4]{Beintro2},
$\PP  _{(m) } ( \mathfrak{I} _i)$ is a $\D _{X _i} ^{(m)}$-module such that,
for any integers $n$ and $n'$, for any $P\in \D _{X  _i,n} ^{(m)}$, $x \in
\overline{\mathcal{I}} _i ^ {\{ n'\} _{(m)}}$, we have $P \cdot x \in \overline{\mathcal{I}} _i ^ {\{ n'-n\}_{(m)}}$.
With the formula \ref{HomDDhash-formula}, this implies that
the sub-sheaf
\begin{equation}\label{defGamma1}
\smash{\underline{\Gamma} } ^{(m)} _{\widetilde{X}} ( \E _i) : =
\underset{\underset{n}{\longrightarrow}}{\lim}
\mathcal{H}om _{\O _{X _i}} (\PP ^n _{(m) } ( \mathfrak{I} _i) , \E _i )\notag
\end{equation}
of $\mathcal{H}om _{\O _{X _i}} (\PP  _{(m) } ( \mathfrak{I} _i) , \E _i)$
has an induced structure of $\D _{X^\# _i} ^{(m)}$-module.
We get a functor
$\R \smash{\underline{\Gamma} } ^{(m)} _{\widetilde{X}}\,:\,
D ^+ ( \D _{X^\# _i} ^{(m)}) \rightarrow
D ^+ ( \D _{X^\# _i} ^{(m)}) $, which is computed
using a resolution by injective $\D _{X^\# _i} ^{(m)}$-modules.
When the $\ZZ$ is empty (i.e., without log-poles), we retrieve the usual local cohomological functor
(e.g., see \cite[4.4.4]{Beintro2} or \cite[1.1.3]{caro_surcoherent}).
Since $\D _{X^\# _i} ^{(m)}$ is flat as $\O _{X _i}$-module, we notice
that an injective $\D _{X^\# _i} ^{(m)}$-module (resp. an injective $\D _{X _i} ^{(m)}$-module) is also
an injective $\O _{X_i} ^{(m)}$-module.
Then, this functor
$\R \smash{\underline{\Gamma} } ^{(m)} _{\widetilde{X}}$ commutes with
the forgetful functor $D ^+ ( \D _{X _i} ^{(m)}) \rightarrow
D ^+ ( \D _{X^\# _i} ^{(m)})$.

We construct then
$\R \underline{\Gamma} ^\dag _{\widetilde{X}}$ :
$\smash{\underset{^{\longrightarrow}}{LD}} ^{\mathrm{b}} _{\Q ,\mathrm{qc}}
  ( \smash{\widehat{\D}} _{\X ^{\#}} ^{(\bullet)})
  \rightarrow
  \smash{\underset{^{\longrightarrow}}{LD}} ^{\mathrm{b}} _{\Q ,\mathrm{qc}}
  ( \smash{\widehat{\D}} _{\X ^{\#}} ^{(\bullet)})$
the local cohomology with strict compact support
in $\widetilde{X}$ similarly to \cite[2.1--2]{caro_surcoherent}. Also,
as for \cite[2.2.6.1]{caro_surcoherent},
we have the canonical isomorphism:
\begin{equation}
  \label{Gammaotimes}
  \R \underline{\Gamma} ^\dag _{\widetilde{X}} ( \E ^{(\bullet)})
\smash{\overset{\L}{\otimes}}   ^{\dag}_{\O _{\X,\Q}} \FF ^{(\bullet)}
\riso
\R \underline{\Gamma} ^\dag _{\widetilde{X}} (
\E ^{(\bullet)} \smash{\overset{\L}{\otimes}}   ^{\dag}_{\O _{\X,\Q}}  \FF ^{(\bullet)}
).
\end{equation}
Finally, since its is known (e.g., see \cite[2.2.1]{caro_surcoherent})
when $\E ^{(\bullet)} = \O _{\X}^{(\bullet)}$ (in
$\smash{\underset{^{\longrightarrow}}{LD}} ^{\mathrm{b}} _{\Q ,\mathrm{qc}}
  ( \smash{\widehat{\D}} _{\X } ^{(\bullet)})$
  and then in $\smash{\underset{^{\longrightarrow}}{LD}} ^{\mathrm{b}} _{\Q ,\mathrm{qc}}
  ( \smash{\widehat{\D}} _{\X ^{\#}} ^{(\bullet)})$
  via the forgetful functor),
for any divisor $\widetilde{X}$ of $X$,
we get from \ref{Gammaotimes} and \ref{ODdivcohe}
the exact triangle of localization of $\E ^{(\bullet)}$ with respect to $\widetilde{X}$ as follows:
\begin{equation}
  \label{locexactri}
  \R \underline{\Gamma} ^\dag _{\widetilde{X}} (\E ^{(\bullet)})
  \rightarrow
  \E ^{(\bullet)}
  \rightarrow
  (\hdag \widetilde{X}) (\E ^{(\bullet)})
  \rightarrow
  \R \underline{\Gamma} ^\dag _{\widetilde{X}} (\E ^{(\bullet)}) [1].
\end{equation}
Similarly, we deduce from \ref{Gammaotimes} that the usual rules of composition of local cohomological functors
and Mayer-Vietoris exact triangles holds (more precisely, see \cite[2.2.8, 2.2.16]{caro_surcoherent}).
\end{empt}

\begin{empt}
[Transitivity]
  \label{properties}
Let $h'$ : $\X'' \rightarrow \X'$ be a second morphism of smooth formal schemes over $\T$,
let $\ZZ''$ be a relatively strict normal crossing divisor of $\X ''$ over $\T$
such that $h ^{\prime -1} (\ZZ') \subset \ZZ''$,
let $D''$ be a divisor of $X''$ such that $h ^{\prime -1} (D') \subset D''$.
We denote by $ \X ^{\prime \prime \#}:= (\X'',\ZZ'')$
and $h ^{\prime\#}$ : $\X ^{\prime\prime \#}\rightarrow \X ^{\prime \#}$
the induced morphism of smooth formal log-schemes over $\T$.

Then, we have the isomorphisms of functors:
\begin{gather}
  \label{trans+}
  h ^\# _{D,D'+} \circ h ^{\prime \#} _{D',D''+}
  \riso (h ^\# \circ h ^{\prime \#}) _{D,D''+} ,\\
\label{trans!}
h ^{\prime \# !} _{D'',D'}  \circ h ^{\# !} _{D',D}
\riso
\smash{(h ^\# \circ h ^{\prime \#})} ^{!} _{D'',D}
\end{gather}
Indeed, thanks to Berthelot's notion of quasi-coherence,
we come down to the case of log-schemes, which is classical.

\end{empt}

\begin{empt}
  Similarly to \cite[1.1.9]{caro_courbe-nouveau}, we check the canonical isomorphisms of functors:
\begin{equation}
  \label{forg+!}
  \mathrm{Forg} _{D} \circ h ^\# _{D,D'+}  \riso h ^\# _{+} \circ \mathrm{Forg} _{D'}  ,\hspace{2cm}
(\hdag D') \circ h ^{\# !} _{D',D} \riso h ^{\# !} \circ (\hdag D).
\end{equation}
\end{empt}

\begin{empt}
[Coherence and quasi-coherence]
We pose
$\D ^{\dag }_{\X  ^\# } (\hdag D) _{\Q}:=
\underset{\underset{m}{\longrightarrow}}{\lim}
\smash{\widehat{\D}} _{\X ^{\#}} ^{(m)} ( D) _\Q$.
We get a
$(\D ^{\dag } _{\X ^{\prime \#}} (\hdag D' )_{\Q},\, h ^{-1} \D ^{\dag }_{\X  ^\# } (\hdag D) _{\Q})$-bimodule
and respectively
a ($h ^{-1} \D ^{\dag }_{\X  ^\# } (\hdag D) _{\Q}$, $\D ^{\dag } _{\X ^{\prime \#}} (\hdag D' )_{\Q}$)-bimodule
with
$$\D ^{\dag} _{\X ^{\prime \#} \rightarrow \X ^\# } (\hdag D ',D )_{\Q}
:=\underset{\underset{m}{\longrightarrow}}{\lim}
\smash{\widehat{\D}} _{\X ^{\prime \#}\rightarrow \X ^{\#}} ^{(m)} ( D' , D) _\Q,\,
\D ^{\dag} _{\X ^\# \leftarrow \X ^{\prime \#}} (\hdag D ,D ')_{\Q}:=
\underset{\underset{m}{\longrightarrow}}{\lim}
\smash{\widehat{\D}} _{\X ^{\#}\leftarrow \X ^{\prime \#}} ^{(m)} ( D , D') _\Q.$$

  We have also the canonical functor $\underset{\longrightarrow}{\lim}$ :
$\smash{\underset{^{\longrightarrow}}{LD}} ^{\mathrm{b}} _{\Q ,\mathrm{qc}}
(\smash{\widehat{\D}} _{\X ^{\#}} ^{(\bullet)}(D))
\rightarrow
D  ( \D ^\dag _{\X ^{\#}} (\hdag D) _{\Q} )$ (see \cite[4.2.2]{Beintro2}). Remark that by abuse of notation
this functor is in fact the composition of the inductive limite
on the level with the functor $-\otimes _\Z \Q$.
This functor $\underset{\longrightarrow}{\lim}$ induces an equivalence
of categories between a subcategory of $\smash{\underset{^{\longrightarrow}}{LD}} ^{\mathrm{b}} _{\Q ,\mathrm{qc}}
(\smash{\widehat{\D}} _{\X ^{\#}} ^{(\bullet)}(D))$, denoted by
$\smash{\underset{^{\longrightarrow}}{LD}} ^{\mathrm{b}} _{\Q ,\mathrm{coh}}
(\smash{\widehat{\D}} _{\X ^{\#}} ^{(\bullet)}(D))$,
and $D ^\mathrm{b} _\mathrm{coh} ( \D ^\dag _{\X ^{\#}} (\hdag D) _{\Q} )$ (similarly to \cite[4.2.4]{Beintro2}).
Let $\E ^{(\bullet)}\in \smash{\underset{^{\longrightarrow}}{LD}} ^{\mathrm{b}} _{\Q ,\mathrm{coh}}
  ( \smash{\widehat{\D}} _{\X ^{\#}} ^{(\bullet)}(D))$,
  $\E ^{\prime(\bullet)}\in \smash{\underset{^{\longrightarrow}}{LD}} ^{\mathrm{b}} _{\Q ,\mathrm{coh}}
  ( \smash{\widehat{\D}} _{\X ^{\prime \#}} ^{(\bullet)}(D'))$.
We denote by
$\E :=\underset{\longrightarrow}{\lim}\E ^{(\bullet)}$, $\E ' := \underset{\longrightarrow}{\lim} \E ^{\prime(\bullet)}$.
Then we get :

\begin{gather}
\underset{\longrightarrow}{\lim}\circ
h ^{\# !} _{D',D} (\E ^{(\bullet)})
\riso
\D ^{\dag} _{\X ^{\prime \#} \rightarrow \X ^\# } (\hdag D ',D )_{\Q}
\otimes ^\L _{h ^{-1} \D ^{\dag} _{\X ^\# } (\hdag D )_{\Q} }
h ^{-1} \E [d_{X'/X}]=:
h ^{\# !} _{D',D} (\E ),
\\
\underset{\longrightarrow}{\lim}\circ
h ^\# _{D,D'+} (\E ^{\prime(\bullet)})
\riso
\R h _* (\D ^{\dag} _{\X ^\# \leftarrow \X ^{\prime \#}} (\hdag D ,D ')_{\Q}
\otimes ^\L _{\D ^{\dag} _{\X ^{\prime \#}} (\hdag D ')_{\Q}}
\E ')=:
h ^\# _{D,D'+} (\E '),
\\
\label{hdagDtildeDcoh}
\underset{\longrightarrow}{\lim}\circ
(\hdag \widetilde{D}, D) (\E ^{(\bullet)})
\riso
\D ^{\dag} _{\X ^{ \#}} (\hdag \widetilde{D} )_{\Q}
\otimes  _{\D ^{\dag} _{\X ^{ \#}} (\hdag D )_{\Q}}\E
=: (\hdag \widetilde{D}, D) (\E ).
\end{gather}
In the last isomorphism, we have removed the symbol ``$\L$'' since the extension
$\D ^\dag _{\X ^\#} (\hdag D) _\Q \rightarrow \D ^\dag _{\X ^\#} (\hdag \widetilde{D}) _\Q$ is flat
(this a consequence of \cite[4.7]{caro_log-iso-hol}).
Also, we can write
$\E (\hdag \widetilde{D}, D) :=(\hdag \widetilde{D}, D) (\E )$.

We pose
$\O _{\X} (\ZZ):= \mathcal{H} om _{\O _{\X}} ( \omega _{\X} , \omega _{\X ^\#})$ and
$\E (\ZZ) = \O _{\X} (\ZZ) \otimes _{\O _{\X}} \E$.
This functor $(-)(\ZZ)$ preserves $D ^\mathrm{b} _\mathrm{coh} (\D ^\dag _{\X ^\#} (\hdag D) _\Q)$
(see \cite[5.1]{caro_log-iso-hol}).
Moreover, because this is true when $\E =\D ^\dag _{\X ^\#} (\hdag D) _\Q$, we check
by functoriality the isomorphism in $D ^\mathrm{b} _\mathrm{coh} (\D ^\dag _{\X ^\#} (\hdag D) _\Q)$:
\begin{equation}
  \label{Z-hdagDcomm}
\E (\ZZ) (\hdag D) \riso  \E  (\hdag D) (\ZZ).
\end{equation}
Also, when $Z \subset D$, we compute $\E (\hdag D)\riso \E (\ZZ) (\hdag D)  $.
\end{empt}

\begin{empt}
\label{defidual}
Let $\E \in D ^\mathrm{b} _{\mathrm{coh}} (\D ^\dag _{\X ^\#,\Q}) $.
  The $\D ^\dag _{\X ^\#,\Q}$-linear dual of $\E$ is well defined as follows (see \cite[5.6]{caro_log-iso-hol}):
\begin{equation}
  \DD _{\X ^\#} (\E) =
\R \mathcal{H} om _{ \D ^\dag _{\X ^\#,\Q}}
(\E, \D ^\dag _{\X ^\#,\Q} ) \otimes _{\O _{\X}} \omega ^{-1} _{\X ^\#} [d _X].
\end{equation}
\end{empt}

\begin{empt}
[Direct image by a log-smooth morphism]
\label{h+smooth}
We suppose here that $h ^\#$ is log-smooth.
Then, as for \cite[4.2.1.1]{Beintro2}, we have the canonical quasi-isomorphism:
$\Omega _{\X ^{\prime \#}/\X ^\#,\Q} ^\bullet
\otimes _{\O _{\X ',\Q}} \D ^\dag _{\X ^{\prime \#},\Q} [d _{\X ^{\prime \#}/\X ^\#}]
  \riso
  \D ^\dag _{\X ^\# \leftarrow \X ^{\prime \#},\Q}$. This implies:
  $\Omega _{\X ^{\prime \#}/\X ^\#,\Q} ^\bullet
  \otimes _{\O _{\X ',\Q}} \D ^\dag _{\X ^{\prime \#}} (\hdag D ') _{\Q} [d _{\X ^{\prime \#}/\X ^\#}]
  \riso \D ^\dag _{\X ^\# \leftarrow \X ^{\prime \#}} (\hdag D, D ') _{\Q}$.
  Then, for any $\E' \in D ^\mathrm{b} _{\mathrm{coh}} (\D ^{\dag} _{\X ^{\prime \#}} (\hdag D ') _{\Q}) $:
  \begin{equation}
    \label{h+smoothiso}
    h _{D, D' +} ^\# ( \E') :=
    \R h _* (\D ^\dag _{\X ^\# \leftarrow \X ^{\prime \#}} (\hdag D, D ') _{\Q}
    \otimes ^\L  _{\D ^{\dag} _{\X ^{\prime \#}} (\hdag D ') _{\Q}} \E')
    \riso \R h _* (\Omega _{\X ^{\prime \#}/\X ^\#,\Q} ^\bullet \otimes _{\O _{\X ',\Q}} \E') [d _{\X ^{\prime \#}/\X ^\#}].
  \end{equation}
\end{empt}

\subsection{Interpretation of the comparison theorem with arithmetic log-$\D$-modules}
We keep the notation of \ref{nota-1.2}.
First, we give in this section the following interpretation of
convergent ($F$-)log-isocrystals on $(X,Z)$ over $\S$.
Moreover, we translate theorem \ref{Nobuorigid} and finally proposition \ref{overconv},
which will be respectively fundamental for the section \ref{subsection22} and \ref{subsection23}.

\begin{prop}
\label{Ddag=overcvcondition}
\begin{enumerate}
  \item \label{Ddag=overcvcondition1}
   The functors $\sp ^*$ and $\sp _*$ induce quasi-inverse equivalences between the category of
coherent $\D_{\X^\#}^\dag(\hdag D) _{ \Q}$-modules, locally projective of finite type
over $\O_{\X}(\hdag D)_{ \Q}$
and the category of
locally free $j_U^\dag\O_{]X[_\X}$-modules of finite type
with an integrable logarithmic connection
$\nabla : E \rightarrow j_U^\dag\Omega_{\Xlogan/\S_K}^1 \otimes_{j_U^\dag\O_{]X[_\X}} E$
satisfying the overconvergent condition of \ref{logovcon}.

\item \label{Ddag=overcvcondition2}
Denote by $I _{\mathrm{conv, et}}  ( (X,Z)/ \Spf \V )$,
the category of convergent log-isocrystals on $(X,Z)$ over $\S$
in the sense of Shiho (see \cite[2.1.5, 2.1.6]{Shi1} and \cite{Shiho-log-isocI}).
There exists an equivalence between $I _{\mathrm{conv, et}}((X,Z)/ \Spf \V)$
and the category of coherent $\D_{\X^\#, \Q}^\dag$-modules,
locally projective of finite type over $\O_{\X, \Q}$.
\end{enumerate}
\end{prop}

\begin{proof}
   We check the first equivalence of categories similarly to \cite[4.4.12]{Be1} (see also \cite[4.19]{caro_log-iso-hol}).
   We deduce the next one by Kedlaya's theorem \cite[6.4.1]{kedlaya-semistableI} (see also his definition \cite[2.3.7]{kedlaya-semistableI}).
\end{proof}

\begin{rema}
\label{remaDdagovcv}
$\bullet$ With the notation \ref{Ddag=overcvcondition},
since $D$ is a divisor, for any locally free $j_U^\dag\O_{]X[_\X}$-module $E$ of finite type,
for any integer $j \not =0$, $\H ^j \sp _* (E)=0$.

$\bullet$ Moreover, it follows from \ref{Ddag=overcvcondition}.\ref{Ddag=overcvcondition1} that for any coherent $\D_{\X^\#}^\dag(\hdag D) _{ \Q}$-module,
  locally projective of finite type
over $\O_{\X}(\hdag D)_{ \Q}$, $E:=\sp ^* (\E)$ is a locally free
$j_U^\dag\O_{]X[_\X}$-module of finite type with a logarithmic
connection $\nabla : E \rightarrow j_U^\dag\Omega_{\Xlogan/\T_K}^1
\otimes_{j_U^\dag\O_{]X[_\X}} E$ satisfying the overconvergent condition of \ref{logovcon}.
Of course the converse is not true unless $\T =\S$.
\end{rema}

\begin{empt}
[Inverse image]
  \label{def-inv image}
  Let $\V \rightarrow \V'$ be a morphism of mixed characteristic complete discrete valuation
rings, $k \rightarrow k'$ the induced morphism of perfect residue fields, $\X $ be a smooth formal $\V$-scheme,
$\X'$ be a smooth formal $\V'$-scheme
and $\ZZ$ (resp. $\ZZ'$) be a relatively strict normal crossing divisor of $\X$ over $\Spf \V$ (resp. $\X'$ over $\Spf \V '$).
Let $f  _0$ : $(X',Z') \rightarrow (X,Z)$ be a morphism of log-schemes over $\Spec k$.
We have a canonical inverse image functor under $f _0$ denoted by $f _0 ^*$ :
$I _{\mathrm{conv, et}}  ( (X,Z)/ \Spf \V )
\rightarrow I _{\mathrm{conv, et}}  ( (X',Z')/ \Spf \V ')$ (this is obvious from the definition
\cite[2.1.5, 2.1.6]{Shi1}).
We get from \ref{Ddag=overcvcondition}.\ref{Ddag=overcvcondition2}
an inverse image functor under $f _0$, also denoted by $f _0 ^*$, from the category of
coherent $\D_{(\X,\ZZ),\Q}^\dag$-modules, locally projective
of finite type over $\O_{\X,\Q}$
to the category of
coherent $\D_{(\X ',\ZZ')  ,\Q}^\dag$-modules, locally projective of finite type
over $\O_{\X',\Q}$.
When there exists a lifting $f$ : $(\X ',\ZZ') \rightarrow (\X,\ZZ)$ of
$(X',Z') \rightarrow (X,Z)$ then $ f_0 ^*$ is canonically isomorphic to the usual functor
$f ^*$.
\end{empt}

\begin{empt}[Frobenius structure]
\label{Frobenius-structure}
Suppose now that $\V \rightarrow \V'$ is $\sigma$ (which is a fixed lifting of the $a$th Frobenius power of $k$) and
$f _0$ is $F _{(X,Z)}$ (or simply $F$) the $a$th power of the absolute Frobenius of $(X,Z)$.
A ``coherent $F$-$\D_{(\X,\ZZ),\Q}^\dag$-module, locally projective
of finite type over $\O_{\X,\Q}$'' or ``coherent $\D_{(\X,\ZZ),\Q}^\dag$-module, locally projective
of finite type over $\O_{\X,\Q}$ and endowed with a Frobenius structure'' is a coherent $\D_{(\X,\ZZ),\Q}^\dag$-module $\E$, locally projective
of finite type over $\O_{\X,\Q}$ and endowed with a $\D_{(\X,\ZZ),\Q}^\dag$-linear isomorphism
$\E \riso F ^* (\E)$. This notion is compatible (via the equivalence of categories \ref{Ddag=overcvcondition}.\ref{Ddag=overcvcondition2})
with Shiho's notion of convergent $F$-log-isocrystal on $(X,Z)$ (see \cite[2.4.2]{Shi1}).
By \cite[2.4.3]{Shi1}, an $F$-log-isocrystal on $(X,Z)$ is strikingly locally free.
\end{empt}

The following lemma indicates that the equivalence of categories of
\ref{Ddag=overcvcondition}.\ref{Ddag=overcvcondition1}
is compatible with the most useful functors (see also \ref{commf_0*} for inverse images).

\begin{lemm}
\label{spcommDD'}
Let $D \subset D ' $ be a second divisor of $X$ and
$U':= X \setminus D'$.
  Let $\E$ be a coherent
$\D_{\X^\#}^\dag(\hdag D) _{ \Q}$-module which is a locally projective
$\O_{\X}(\hdag D)_{ \Q}$-module of finite type and $E := \sp ^* (\E)$.
Then
\begin{align}
\label{spcommDD'iso1}
\E (\hdag D')& =\D_{\X ^\#}^\dag(\hdag D ') _{ \Q}  \otimes _{\D_{\X ^\#}^\dag(\hdag D) _{ \Q}} \E \riso
\sp _* (j ^\dag _{U'} E),
\\
\label{spcommDD'iso2}
\R \underline{\Gamma} ^\dag _{D'} (\E )& \riso
\R \sp _* \circ \underline{\Gamma}_{]D'[_\X}^\dag  (E).
\end{align}
\end{lemm}

\begin{proof}
  We have the canonical isomorphism:
$\sp _* (j ^\dag _{U'} E) \riso
\O_{\X}(\hdag D')_{ \Q} \otimes _{\O_{\X}(\hdag D)_{ \Q}}  \E$.
Since $j ^\dag _{U'} E$ satisfies the overconvergent condition,
$\O_{\X}(\hdag D')_{ \Q} \otimes _{\O_{\X}(\hdag D)_{ \Q}}  \E$ is
then a coherent $\D_{\X ^\#}^\dag(\hdag D ') _{ \Q}  $-module which is also
a locally projective
$\O_{\X}(\hdag D')_{ \Q}$-module of finite type.
Then, we get a morphism of coherent $\D_{\X ^\#}^\dag(\hdag D ') _{ \Q}  $-modules:
$\O_{\X}(\hdag D')_{ \Q} \otimes _{\O_{\X}(\hdag D)_{ \Q}}  \E
\rightarrow
\D_{\X ^\#}^\dag(\hdag D ') _{ \Q}  \otimes _{\D_{\X ^\#}^\dag(\hdag D) _{ \Q}} \E $.
Since this morphism is an isomorphism outside $D'$, this is an isomorphism (see \cite[4.8]{caro_log-iso-hol}).
Thus, we have proved \ref{spcommDD'iso1}.

By applying the functor $\R \sp _*$ to an exact sequence of the form \ref{exact-trngl-rig}, we get
the exact triangle (and with the first remark of \ref{remaDdagovcv}):
\begin{equation}
  \notag
         \R \sp _* \circ \underline{\Gamma}_{]D'[_\X}^\dag(E)\,
     \longrightarrow\, \sp _* (E)\, \longrightarrow\, \sp _* ( j_{U'}^\dag (E) )\, \longrightarrow\,
     \R \sp _* \circ \underline{\Gamma}_{]D'[_\X}^\dag(E)[1].
\end{equation}
Since $\sp _* (E)\, \longrightarrow\, \sp _* ( j_{U'}^\dag (E) )$ is canonically isomorphic
to $\E \rightarrow \E( \hdag D')$, it follows from the exact triangle of localization of
$\E$ with respect to $D'$ (see \ref{locexactri}), that
$\R \underline{\Gamma} ^\dag _{D'} (\E )\riso
\R \sp _* \circ \underline{\Gamma}_{]D'[_\X}^\dag  (E)$.
\end{proof}

An exponent of a coherent
$\D_{\X^\#}^\dag(\hdag D) _{ \Q}$-module,
locally projective of finite type
over $\O_{\X}(\hdag D)_{ \Q}$-module means
an exponent of the associated overconvergent log-isocrystal by \ref{Ddag=overcvcondition}.\ref{Ddag=overcvcondition1}.
The comparison theorem \ref{Nobuorigid} can be reformulated as follows:
\begin{theo}
\label{Nobuo}
Let $\E$ be a coherent
$\D_{\X^\#}^\dag(\hdag D) _{ \Q}$-module which is a locally projective
$\O_{\X}(\hdag D)_{ \Q}$-module of finite type. Suppose that
\begin{list}{}{}
\item[\mbox{\rm (a)}] none of differences of exponents
is a $p$-adic Liouville number, and
\item[\mbox{\rm (b')}] any exponent
is neither a $p$-adic Liouville number nor a positive integer
\end{list}

\noindent
along each irreducible component $Z_i$
of $Z$ such that $Z_i \not\subset D$. Then the natural morphism
\begin{equation}
\label{Nobuo-iso}
\R g_*\left(\Omega_{\X^\#/\T, \Q}^\bullet
   \otimes_{\O_{\X, \Q}} \E\right)
       \rightarrow
       \R g_*\left(\Omega_{\X/\T, \Q}^\bullet \otimes_{\O_{\X, \Q}}
       \E(\hdag Z)\right)
\end{equation}
is an isomorphism.
\end{theo}

\begin{proof}
  Using \ref{Ddag=overcvcondition} (and the first remark \ref{remaDdagovcv}),
  we have only to apply the functor $\sp _*$ in \ref{Nobuorigid}
  (with $E := \sp ^* (\E)$).
\end{proof}

\begin{rema}
\label{remaTheo135}
With the notation of \ref{Nobuo}, since
$\R g_*\left(\Omega_{\X ^\#/\T, \Q}^\bullet \otimes_{\O_{\X, \Q}}
       \E(\hdag Z)\right) =
       \R g_*\left(\Omega_{\X/\T, \Q}^\bullet \otimes_{\O_{\X, \Q}}
       \E(\hdag Z)\right)$, it follows from \ref{h+smoothiso} and \ref{locexactri} that
the fact that the morphism \ref{Nobuo-iso} is an isomorphism is equivalent to the fact that
$g _{D,+} ^\# \circ \R \underline{\Gamma} ^\dag _{Z}  (\E)=0$.
We will see also that this is equivalent to the fact that $g _+ (\rho)$ is an isomorphism. But first, we need to recall
the construction of $\rho$.
\end{rema}
\vspace{5mm}

\begin{empt}
[The morphism $\rho$]
  \label{defi-rho}
Let $\E \in D ^\mathrm{b} _\mathrm{coh} (\D ^\dag _{\X ^\#} (\hdag D) _\Q)$.

$\bullet $ From \cite[5.2.4]{caro_log-iso-hol},
we get the canonical isomorphism of $(\D ^\dag _{\X } (\hdag D) _\Q, \D ^\dag _{\X ^\#} (\hdag D) _\Q)$-bimodules:
$\D ^\dag _{\X \leftarrow \X ^\#} (\hdag D) _{\Q}
\riso
\D ^{\dag} _{\X } (\hdag D) _{\Q}\otimes _{\O _{\X}} \O _{\X}(\ZZ)$, where to compute the tensor product we take
the right structure of $\D ^{\dag} _{\X } (\hdag D) _{\Q}$-module (and then the right structure of $\O _{\X}$-module)
of $\D ^{\dag} _{\X } (\hdag D) _{\Q}$.
Hence, the canonical inclusion
$\D ^{\dag} _{\X} (\hdag D) _{\Q} \otimes _{\O _{\X}} \O _{\X}(\ZZ)
\subset
\D ^{\dag} _{\X} (\hdag D \cup Z) _{\Q}$ induces the morphism
$$u_{D+} (\E) =
\D ^\dag _{\X \leftarrow \X ^\#} (\hdag D) _{\Q}
\otimes ^\L  _{\D ^{\dag} _{\X ^\#} (\hdag D) _{\Q}} \E
\rightarrow
\D ^{\dag} _{\X } (\hdag D \cup Z) _{\Q}
\otimes   _{\D ^{\dag} _{\X ^\#} (\hdag D) _{\Q}} \E =
\E (\hdag Z).$$
This canonical morphism is
denoted by
$\rho$ : $u_{D+} (\E) \rightarrow \E (\hdag Z)$.

$\bullet $ From $\D ^\dag _{\X \leftarrow \X ^\#} (\hdag D) _{\Q}
\riso
\D ^{\dag} _{\X } (\hdag D) _{\Q}\otimes _{\O _{\X}} \O _{\X}(\ZZ)$
(and also \cite[6.2.1]{caro_log-iso-hol}), we get
\begin{equation}
\label{u+621}
  u _{D +} (\E)
\riso
\D ^\dag  _{\X } (\hdag D) _{\Q} \otimes ^\L _{\D ^\dag  _{\X ^\#} (\hdag D) _{\Q}} \E (\ZZ).
\end{equation}

$\bullet$ Finally, by \cite[5.25]{caro_log-iso-hol}, when $\E$ is furthermore a log-isocrystal on $\X ^\#$ overconvergent along $D$,
for any $j\not =0$, $\H ^j (u_{D+} (\E) )=0$, i.e.,
$u_{D+} (\E) \riso
\D ^\dag  _{\X } (\hdag D) _{\Q} \otimes _{\D ^\dag  _{\X ^\#} (\hdag D) _{\Q}} \E (\ZZ)$.
This will be essential in the proof of \ref{rhoisolemm2}.
\end{empt}

\begin{rema}
  With the notation \ref{defi-rho}, since the canonical morphism
$(\hdag Z) \circ u _+ (\E) \rightarrow  \E (\hdag Z)$ of coherent $\D ^\dag  _{\X } (\hdag D \cup Z) _{\Q} $-modules
is an isomorphism (this is obvious outside $D \cup Z$ and so we can apply \cite[4.3.12]{Be1}),
the localization triangle of $u_{D+} (\E)$ with respect to $Z$ is canonically isomorphic to
  \begin{equation}
    \label{exacttriu+F}
    \R \underline{\Gamma} ^\dag _Z \circ u_{D+} (\E)  \rightarrow u_{D+} (\E)
    \overset{\rho}{\rightarrow} \E (\hdag Z)
    \rightarrow \R \underline{\Gamma} ^\dag _Z \circ u_{D+} (\E)  [1].
  \end{equation}
Hence,
$\R \underline{\Gamma} ^\dag _Z \circ u _+ (\E)  =0$ if and only if
$\rho$ is an isomorphism.

\end{rema}

We will need the following two lemmas of commutativity:
\begin{lemm}
\label{u+commhdag}
  Let $\widetilde{D}$ be a second divisor of $X$,
$\E ^{(\bullet)} \in \smash{\underset{^{\longrightarrow}}{LD}} ^{\mathrm{b}} _{\Q ,\mathrm{qc}}
  ( \smash{\widehat{\D}} _{\X ^{\#}} ^{(\bullet)}(D))$.
  We have:
\begin{equation}
\label{u+commhdag-equ}
  (u _{D+} (\E ^{(\bullet)} ))(\hdag \widetilde{D})
  \riso
  u _{D+} (\E ^{(\bullet)} (\hdag \widetilde{D}))
  \riso
  u _{\widetilde{D}+} (\E ^{(\bullet)} (\hdag \widetilde{D}))
  .
\end{equation}
\end{lemm}

\begin{proof}
Since, over $\smash{\underset{^{\longrightarrow}}{LD}} ^{\mathrm{b}} _{\Q ,\mathrm{qc}}
  ( \smash{\widehat{\D}} _{\X ^{\#}} ^{(\bullet)}(D))$, $(\hdag \widetilde{D}) \riso (\hdag D \cup \widetilde{D}) $,
we can suppose that $D \subset \widetilde{D}$.
According to our notation (see the beginning of \ref{nota-1.2}),
$u _i$ : $X ^\# _i \rightarrow X _i$ denotes the reduction modulo $\pi ^{i+1}$ of $u$ and
$\E ^{(m)} _i := \O _{X _i} \otimes ^\L _{\O _{\X ^\#}} \E ^{(m)}$.
By posing $\FF  ^{(\bullet)} :=\E ^{(\bullet)} (\hdag \widetilde{D})$,
we get:
$\FF ^{(m)} _i \riso
\D ^{(m)} _{X ^\# _i } (\widetilde{D}) \otimes ^\L _{\D ^{(m)} _{X _i ^\#} (D)} \E ^{(m)} _i $.
By \cite[5.2.4]{caro_log-iso-hol},
$\D ^{(m)} _{X _i \leftarrow X ^\# _i} (D) \riso \D ^{(m)} _{X _i } (D) \otimes _{\O _{X _i}}  \O _{X _i} (Z _i)$.
Hence, using \cite[5.1.2]{caro_log-iso-hol},
we obtain:
$\D ^{(m)} _{X _i \leftarrow X ^\# _i} (D) \otimes _{\D ^{(m)} _{X _i ^\#} (D)} ^\L \FF _i
\riso
\D ^{(m)} _{X _i } (D) \otimes ^\L _{\D ^{(m)} _{X _i ^\#} (D)}
   ( \FF ^{(m)} _i (Z _i))$.
Via the canonical isomorphism of transposition
$\gamma$ :
$\D ^{(m)} _{X ^\# _i } (\widetilde{D})
\otimes _{\O _{X _i}} \O _{X _i} (Z _i)
\riso
\O _{X _i} (Z _i) \otimes _{\O _{X _i}} \D ^{(m)} _{X ^\# _i } (\widetilde{D})$
(see \cite[1.24]{caro_log-iso-hol})
and via \cite[5.1.2]{caro_log-iso-hol},
we get:
$\FF ^{(m)} _i (Z _i)
\riso
   \D ^{(m)} _{X ^\# _i } (\widetilde{D}) \otimes ^\L _{\D ^{(m)} _{X _i ^\#} (D)} (\E ^{(m)} _i (Z _i))$.
   Thus:
$\D ^{(m)} _{X _i \leftarrow X ^\# _i} (D) \otimes _{\D ^{(m)} _{X _i ^\#} (D)} ^\L \FF _i
\riso
\D ^{(m)} _{X _i } (D) \otimes ^\L _{\D ^{(m)} _{X _i ^\#} (D)}
   \D ^{(m)} _{X ^\# _i } (\widetilde{D}) \otimes ^\L _{\D ^{(m)} _{X _i ^\#} (D)} (\E ^{(m)} _i (Z _i))$.
   Since $\D ^{(m)} _{X _i } (D) $ and $\D ^{(m)} _{X ^\# _i } (D) $ are
   $\B _{X _i} ^{(m)} (D)$-flat, we check:
   $\D ^{(m)} _{X ^\# _i } (\widetilde{D}) \riso \D ^{(m)} _{X ^\# _i } (D)
   \otimes ^\L _{\B _{X _i} ^{(m)} (D)} \B _{X _i} ^{(m)} (\widetilde{D})$
   (and also without $\#$). This gives the following $(\D ^{(m)} _{X _i } (D) , \D ^{(m)} _{X ^\# _i } (\widetilde{D}))$-linear isomorphism:
   $\D ^{(m)} _{X _i } (D) \otimes ^\L _{\D ^{(m)} _{X _i ^\#} (D)}
   \D ^{(m)} _{X ^\# _i } (\widetilde{D})
   \riso
  \D ^{(m)} _{X _i } (\widetilde{D}) $, which furnishes
  the second isomorphism:
\begin{gather}
  \notag
  \D ^{(m)} _{X _i \leftarrow X ^\# _i} (D) \otimes _{\D ^{(m)} _{X _i ^\#} (D)} ^\L \FF _i
\riso
\D ^{(m)} _{X _i } (D) \otimes ^\L _{\D ^{(m)} _{X _i ^\#} (D)}
   \D ^{(m)} _{X ^\# _i } (\widetilde{D}) \otimes ^\L _{\D ^{(m)} _{X _i ^\#} (D)} (\E ^{(m)} _i (Z _i)) \riso
  \\ \label{u+commhdagiso}
\riso \D ^{(m)} _{X _i } (\widetilde{D}) \otimes ^\L _{\D ^{(m)} _{X _i ^\#} (D)} \E ^{(m)} _i (Z _i)
\riso
\D ^{(m)} _{X _i } (\widetilde{D})
\otimes ^\L _{\D ^{(m)} _{X _i } (D)}
(\D ^{(m)} _{X _i } (D) \otimes ^\L _{\D ^{(m)} _{X _i ^\#} (D)} \E ^{(m)} _i (Z _i)).
\end{gather}
So we have checked: $u _{D+} (\E ^{(\bullet)} (\hdag \widetilde{D}))
  \riso
  (u _{D+} (\E ^{(\bullet)} ))(\hdag \widetilde{D})$. By \ref{forg+!}, the second isomorphism was known
  (we can also use the second isomorphism of \ref{u+commhdagiso}).
\end{proof}

\begin{lemm}
\label{u+commhdag-coro}
  Let $\widetilde{D}$ be a second divisor of $X$,
$\E ^{(\bullet)} \in \smash{\underset{^{\longrightarrow}}{LD}} ^{\mathrm{b}} _{\Q ,\mathrm{qc}}
  ( \smash{\widehat{\D}} _{\X ^{\#}} ^{(\bullet)}(D))$.
  We have:
\begin{equation}
\label{u+commhdag-coro-equ}
  u _{D+} \circ \R \underline{\Gamma} ^\dag _{\widetilde{D}} (\E ^{(\bullet)} )
  \riso
  \R \underline{\Gamma} ^\dag _{\widetilde{D}} \circ u _{D+} (\E ^{(\bullet)} ).
\end{equation}
\end{lemm}
\begin{proof}
  This is a consequence of \ref{u+commhdag}.
  Indeed, following \ref{locexactri}, the mapping cone of
  $\R \underline{\Gamma} ^\dag _{\widetilde{D}} \circ u _{D+} \circ \R \underline{\Gamma} ^\dag _{\widetilde{D}} (\E ^{(\bullet)} )
  \rightarrow
  u _{D+} \circ \R \underline{\Gamma} ^\dag _{\widetilde{D}} (\E ^{(\bullet)} )$
  is isomorphic to
  $(\hdag \widetilde{D} ) \circ u _{D+} \circ \R \underline{\Gamma} ^\dag _{\widetilde{D}} (\E ^{(\bullet)} ) =0 $ by \ref{u+commhdag}.
Also, the mapping cone of
$\R \underline{\Gamma} ^\dag _{\widetilde{D}} \circ u _{D+} \circ \R \underline{\Gamma} ^\dag _{\widetilde{D}} (\E ^{(\bullet)} )
  \rightarrow
  \R \underline{\Gamma} ^\dag _{\widetilde{D}} \circ u _{D+}  (\E ^{(\bullet)} )$
  is isomorphic to
$\R \underline{\Gamma} ^\dag _{\widetilde{D}} \circ u _{D+} \circ (\hdag \widetilde{D} )  (\E ^{(\bullet)} ) =0 $ by \ref{u+commhdag}.
\end{proof}

\begin{coro}
\label{Nobuobis}
Let $\E$ be a coherent
$\D_{\X^\#}^\dag(\hdag D) _{ \Q}$-module which is a locally projective
$\O_{\X}(\hdag D)_{ \Q}$-module of finite type and which satisfies
the conditions (a) and (b') of \ref{Nobuo}.
Then, the morphism
$  g _{D,+}  ( u_{D+} (\E )) \underset{g _+ (\rho)}{\longrightarrow} g _{D\cup Z,+} (\E (\hdag Z))$
is an isomorphism and $g  _{+}  \R \underline{\Gamma} ^\dag _{Z} \circ u _{D,+}   (\E)=0$.
\end{coro}

\begin{proof}
By the exact triangle \ref{exacttriu+F}, this is sufficient to
check that $g _{D,+} \circ \R \underline{\Gamma} ^\dag _{Z} \circ u _{D+} (\E )=0$.
But $g  ^\# _{D,+} \riso g  _{D,+} \circ u _{D,+}$ (see \ref{trans+}).
Hence, by \ref{remaTheo135}, we get
$g  _{D,+} \circ u _{D,+} \circ \R \underline{\Gamma} ^\dag _{Z}  (\E)=0$.
We finish the proof by using \ref{u+commhdag-coro-equ}.
\end{proof}

Finally, we finish with the following version of \ref{overconv}:
\begin{theo}
\label{overconvbis}
We assume
that $g : \X \rightarrow \T$ factors through an irreducible component $\ZZ_1$ of $\ZZ$ by
a smooth morphism $g_1 : \X \rightarrow \ZZ_1$ over $\T$
such that the composite $g_1 \circ i_ 1 : \ZZ_1 \rightarrow \ZZ_1$
of the closed immersion
$i_1 : \ZZ_1 \rightarrow \X$ and $g_1$ is the identity.
Moreover, we suppose that $D \cap Z _1$ is a divisor of $Z _1$.
Let $\ZZ_1' = \cup_{i = 2}^s\, \ZZ_1 \cap \ZZ_i$ be a strict normal crossing divisor of $\ZZ_1$,
$\ZZ_1^{\#} := (\ZZ_1, \ZZ_1')$.
We suppose that
$g _1 ^{-1} (\ZZ_1') = \cup_{i = 2}^s\, \ZZ_i$ and let
$g _1 ^\# \,:\, \X ^\# \rightarrow \ZZ_1^{\#} $ be the canonical induced morphism.

Let $\E$ be a coherent
$\D_{\X^\#}^\dag(\hdag D) _{ \Q}$-module which is a locally projective
$\O_{\X}(\hdag D)_{ \Q}$-module of finite type and
which satisfies the conditions (a) and (b) in \ref{Nobuorigid}.
Then the complex
\begin{equation}
  \label{overconvbis-cone}
\mathrm{Cone} \left( g _{1+} ^\# ( \E) \rightarrow g _{1+} ^\# ( \E (\hdag Z _1))\right)
\end{equation}
is isomorphic to a complex of
coherent $\D_{\ZZ _1 ^\#}^\dag(\hdag D \cap Z _1) _{ \Q}$-modules, locally projective
of finite type as $\O_{\ZZ _1}(\hdag D \cap Z _1)_{ \Q}$-modules
and satisfying the conditions (a) and (b) of \ref{Nobuorigid}.
\end{theo}
\begin{proof}
  We pose $E := \sp ^* (\E)$ and $Y _1 := X \setminus Z _1$. Then,
  since the functor $\underline{\Gamma}^\dag_{]Z_1[_\X}$ is exact, since
  mapping cones commute with the functor $\R g_{1K*}(\Omega_{\Xlogan/\ZZ_{1K}^{\#}}^\bullet
   \otimes_{\O_{]X[_\X}} - )$
  and
  $j_U^\dag\Omega_{\Xlogan/\ZZ_{1K}^{\#}}^\bullet
   \otimes_{j_U^\dag\O_{]X[_\X}} E
   \cong
   \Omega_{\Xlogan/\ZZ_{1K}^{\#}}^\bullet
   \otimes_{\O_{]X[_\X}} E$, we obtain
\begin{equation}
  \label{overconvbis-iso1}
  \R g_{1K*}\underline{\Gamma}^\dag_{]Z_1[_\X}\left
   (j_U^\dag\Omega_{\Xlogan/\ZZ_{1K}^{\#}}^\bullet
   \otimes_{j_U^\dag\O_{]X[_\X}} E\right)
   \cong
\mathrm{Cone} \left(
\R g_{1K*}(\Omega_{\Xlogan/\ZZ_{1K}^{\#}}^\bullet
   \otimes_{\O_{]X[_\X}} E )\rightarrow
   \R g_{1K*}(\Omega_{\Xlogan/\ZZ_{1K}^{\#}}^\bullet
   \otimes_{\O_{]X[_\X}} j_{Y _1} ^\dag E )
   \right)[-1].
\end{equation}
By applying the functor $\R \sp _*$ in the right term of \ref{overconvbis-iso1},
since $\R \sp _* \circ \R g_{1K*} \riso \R g_{1*} \circ \R \sp _* $ and
using the first remark of \ref{remaDdagovcv},
  we get
the complex
\begin{equation}
    \label{overconvbis-iso2}
\mathrm{Cone} \left(
\R g_{1*}(\Omega_{\X ^\# /\ZZ_{1}^{\#},\Q}^\bullet
   \otimes_{\O_{\X,\Q}} \sp _*(E ))\rightarrow
   \R g_{1*}(\Omega_{\X ^\#/\ZZ_{1}^{\#},\Q}^\bullet
   \otimes_{\O_{\X,\Q}} \sp _*(j_{Y _1} ^\dag E ))
   \right)[-1].
\end{equation}
Following \ref{h+smoothiso}, \ref{Ddag=overcvcondition}.\ref{Ddag=overcvcondition1}
and \ref{spcommDD'iso1},
the complex \ref{overconvbis-iso2} is isomorphic (up to a shift) to
\ref{overconvbis-cone}.

On the other hand, by applying the functor $\R \sp _*$ in the left term of \ref{overconvbis-iso1},
using the isomorphism \ref{overconv-iso} and the first remark of \ref{remaDdagovcv}
(and of course \ref{Ddag=overcvcondition}.\ref{Ddag=overcvcondition1}),
we get a complex isomorphic to a complex of
coherent $\D_{\ZZ _1 ^\#}^\dag(\hdag D \cap Z _1) _{ \Q}$-modules, locally projective
of finite type as $\O_{\ZZ _1}(\hdag D \cap Z _1)_{ \Q}$-modules
and satisfying the conditions (a) and (b) in \ref{Nobuorigid}

\end{proof}

\begin{rema}
With the notation \ref{overconvbis}, we have the isomorphism (see \ref{locexactri}):
  \begin{equation}
  \label{overconvbis-cone2}
g _{1+} ^\#  \circ \R \underline{\Gamma} ^\dag _{Z _1}( \E) \riso
\mathrm{Cone} \left( g _{1+} ^\# ( \E) \rightarrow g _{1+} ^\# ( \E (\hdag Z _1))\right)
[-1].
\end{equation}
\end{rema}
\section{Application to the study of overconvergent $F$-isocrystals and arithmetic $\D$-modules}

\subsection{Kedlaya's semi-stable reduction theorem}

We recall the following Kedlaya's definitions (see \cite[3.2.1, 3.2.4]{kedlaya-semistableII}):

\begin{defi}
\label{unipotent}
Let $X$ be a smooth irreducible variety over $\Spec k$, $Z$ be a strict normal crossing divisor of $X$,
  and let $E$ be a convergent isocrystal on $X \setminus Z$. We say that $E$ is {\it log-extendable} on $X$
  if
  there exists
  a log-isocrystal with nilpotent residues convergent on the log-scheme $(X,Z)$ (see \cite[2.1.5, 2.1.6]{Shi1})
  whose induced convergent isocrystal on $X \setminus Z$ is $E$.
When $E$ is even an isocrystal on $X \setminus Z$ overconvergent along $Z$
  then $E$ is log-extendable if and only if $E$ has unipotent monodromy along $Z$ (see definition \cite[4.4.2]{kedlaya-semistableI}
  and theorem \cite[6.4.5]{kedlaya-semistableI}).
\end{defi}

\begin{defi}
  \label{defi-qunipt}
Let $Y$ be a smooth irreducible variety over $\Spec k$, let $X$ be
a partial compactification of $Y$,
and let $E$ be an $F$-isocrystal on $Y$ overconvergent along $X\setminus Y$.
We say that $E$ {\it admits semistable reduction}
if there exists
\begin{enumerate}
  \item a proper, surjective, generically \'etale morphism $f$ : $X _1 \rightarrow X$,
  \item an open immersion $X _1 \hookrightarrow \overline{X} _1$ into a smooth projective variety over $k$
  such that $D _1 := f ^{-1} (X \setminus Y) \cup (\overline{X} _1 \setminus X _1)$
  is a strict normal crossing divisor of $\overline{X} _1 $
\end{enumerate}
such that the isocrystal $f ^* (E)$ on $Y _1:= f^{-1}(Y)$ overconvergent along $D _1 \cap X _1$
is log-extendable on $X _1$ (see \ref{unipotent}).

\end{defi}

With the previous definitions,
Kedlaya has proved in \cite[2.4.4]{kedlaya-semistableIV} (see also \cite{kedlaya-semistableI}, \cite{kedlaya-semistableII}, \cite{kedlaya-semistableIII})
the following theorem which answers positively to Shiho's conjecture in \cite[3.1.8]{Shi1}:
\begin{theo}[Kedlaya]
\label{semi-stable}
Let $Y$ be a smooth irreducible $k$-variety, $X$ be
a partial compactification of $Y$, $Z:=X \setminus Y$,
$E$ be an $F$-isocrystal on $Y$ overconvergent along $Z$.
Then $E$ admits semistable reduction.
\end{theo}

\begin{rema}
This conjecture was previously checked by Tsuzuki when $E$ is unit-root in \cite{Tsu-mono} and
by Kedlaya in the case of curves (see \cite{kedlaya_semi-stable}).
\end{rema}

\subsection{A comparison theorem between log-de Rham complexes and de Rham complexes}
\label{subsection22}
Let $\X$ be a smooth formal $\V$-scheme, $D$ be a divisor of $X$, $Y := X \setminus D$,
$\ZZ$ be a strict normal crossing divisors of $\X$,
$\X ^\#:=(\X, \ZZ) $ be the induced smooth logarithmic formal $\V$-scheme,
$u$ : $\X ^\#\rightarrow \X$ be the canonical morphism.

\begin{lemm}
  \label{preZZ'u}
Let $\ZZ'$ be a strict normal crossing divisor of $\X$
  such that $\ZZ \cup \ZZ'$ is a strict normal crossing divisor of $\X$
  and such that $Z \cap Z ' $ is of codimension $2$ in $X$ (i.e., the irreducible components of
  $Z$ and $Z'$ are different).
  We pose  $\X ^{\# \prime} =(\X, \ZZ \cup \ZZ')$. Then the canonical morphism
$\D ^\dag _{\X ^{\# \prime}} (\hdag D\cup Z ') _\Q \rightarrow \D ^\dag _{\X ^\#} (\hdag D\cup Z ') _\Q $
is an isomorphism.
\end{lemm}

\begin{proof}
The assertion is local in $\X$. We can suppose that there exists local coordinates $t _1,\dots, t_d$ of $\X$
such that $\ZZ \cup \ZZ' = V(t_1\dots t_r)$ and $\ZZ = V (t _{s+1} \dots t_r)$ for some $0\leq s \leq r$.
For any integer $m$, we have the canonical inclusion:
$\smash{\widehat{\D}} ^{(m)} _{\X ^{\# \prime}} (D\cup Z ') _\Q
\subset
\smash{\widehat{\D}} ^{(m)} _{\X ^{\#}} (D\cup Z ') _\Q$ (see the notation of \ref{nota-hatD}).
A fortiori, by direct limit on the level, we obtain
$\D ^\dag _{\X ^{\# \prime}} (\hdag D\cup Z ') _\Q \subset
\D ^\dag _{\X ^\#} (\hdag D\cup Z ') _\Q $.

Less obviously, let us check the converse.
  For any integer $k$, we denote by
  $ q _k ^{(m)}$, $ q _k ^{(m+1)}$, $ r _k ^{(m)}$, $r _k ^{(m+1)}$, $\tilde{r} _k ^{(m)}$
  the integers satisfying the following conditions:
  $k = p ^m q _k ^{(m)} + r _k ^{(m)}$, $0\leq r _k ^{(m)}<p ^{m}$,
  $k = p ^{m+1} q _k ^{(m+1)} + r _k ^{(m+1)}$, $0\leq r _k ^{(m+1)}<p ^{m+1}$,
  $q _k ^{(m)} =p q _k ^{(m+1)} + \tilde{r} _k ^{(m)}$, $0\leq \tilde{r} _k ^{(m)}<p$.
  We recall that the $p$-adic valuation of $k!$ is
  $v _p (k!)=(k - \sigma (k))/(p-1)$, where $\sigma (k) = \sum _i a _i$ if $k = \sum _i a _i p ^i $ with $0\leq a _i <p$.
We compute: $ v _p (q _k ^{(m)}!) -v _p (q _k ^{(m+1)}!)=(q _k ^{(m)} -q _k ^{(m+1)} - \tilde{r} _k ^{(m)})/(p-1)
=q _k ^{(m+1)}$.
By \cite[2.2.3.1]{Be1} (and $\smash{\widehat{\D}} ^{(m)} _{\X,\Q} \subset
\smash{\widehat{\D}} ^{(m+1)} _{\X,\Q}$), we have:
$\partial _i ^{<k> _{(m)}}=
q _k ^{(m)}!/ q _k ^{(m+1)}! \partial _i ^{<k> _{(m+1)}}$.
Then, there exists a unit $u$ of $\Z _p$ such that for every $0\leq i\leq s$,
we get:
$\partial _i ^{<k> _{(m)}}
=u p ^{q _k ^{(m+1)}} \partial _i ^{<k> _{(m+1)}}
=\frac{u}{ t ^{r _k ^{(m+1)}}}
\left(\frac{p}{t _i ^{p ^{m+1}}}\right) ^{q _k ^{(m+1)}}  t _i ^k \partial _i ^{<k> _{(m+1)}}$.
Since for any $k$ we have
$\frac{u}{ t ^{r _k ^{(m+1)}}}
\left(\frac{p}{t _i ^{p ^{m+1}}}\right) ^{q _k ^{(m+1)}} \in
\frac{1}{t ^{(p^{m+1}-1)}}\widehat{\B} ^{(m)} _{\X} (D \cup Z')$, we obtain the inclusion
$\smash{\widehat{\D}} ^{(m)} _{\X ^{\#}} (D\cup Z ') _\Q
\subset
\frac{1}{t ^{(p^{m+1}-1)}}
\smash{\widehat{\D}} ^{(m+1)} _{\X ^{\# \prime}} (D\cup Z ') _\Q$.
Since $\frac{1}{t ^{(p^{m+1}-1)}}$ is invertible in
$\smash{\widehat{\D}} ^{(m+1)} _{\X ^{\# \prime}} (D\cup Z ') _\Q$, this implies:
$\smash{\widehat{\D}} ^{(m)} _{\X ^{\#}} (D\cup Z ') _\Q
\subset
\smash{\widehat{\D}} ^{(m+1)} _{\X ^{\# \prime}} (D\cup Z ') _\Q$.
Then, by taking the direct limit on the level,
$\D ^\dag _{\X ^\#} (\hdag D\cup Z ') _\Q \subset
\D ^\dag _{\X ^{\# \prime}} (\hdag D\cup Z ') _\Q $.
\end{proof}

\begin{lemm}
  \label{ZZ'u}
With the same notation as in \ref{preZZ'u}, let
$v$ : $\X ^{\# \prime} \rightarrow \X$ be
  the canonical morphism.
  For any $\E \in D ^\mathrm{b} _\mathrm{coh} (\D ^\dag _{\X ^\#} (\hdag D) _\Q)$ and
  $\E '\in D ^\mathrm{b} _\mathrm{coh} (\D ^\dag _{\X ^{\# \prime}} (\hdag D) _\Q)$,
  we have the isomorphisms in $D ^\mathrm{b} _\mathrm{coh} (\D ^\dag _{\X  }(\hdag D\cup Z') _\Q)$ :
\begin{gather}
  \label{ZZ'u-equ1}
v _{D \cup Z'+} (\E (\hdag Z')) \riso u _{D \cup Z'+} (\E (\hdag Z')) \riso (u _{D +} (\E) ) (\hdag Z'),
\\ \label{ZZ'u-equ2}
u _{D \cup Z'+} (\E '(\hdag Z')) \riso v _{D \cup Z'+} (\E '(\hdag Z'))  \riso (v _{D +} (\E') ) (\hdag Z').
\end{gather}

\end{lemm}
\begin{proof}
First, since $\D ^\dag _{\X ^{\# \prime}} (\hdag D\cup Z ') _\Q=\D ^\dag _{\X ^\#} (\hdag D\cup Z ') _\Q $
(see \ref{preZZ'u}),
the left terms of \ref{ZZ'u-equ1} and \ref{ZZ'u-equ2} are well defined.
Also, as the proof of \ref{ZZ'u-equ2} is similar, we will only check
\ref{ZZ'u-equ1}.

By \ref{u+621},
$u _{D +} (\E)
\riso
\D ^\dag  _{\X } (\hdag D) _{\Q} \otimes ^\L _{\D ^\dag  _{\X ^\#} (\hdag D) _{\Q}} \E (\ZZ)$.
Then, we get by associativity of the tensor product:
  $$(u _{D +} (\E) ) (\hdag Z') \riso
  \D ^\dag _{\X } (\hdag D \cup Z') _\Q
  \otimes ^{\L} _{\D ^\dag _{\X ^\#} (\hdag D) _\Q} \E (\ZZ)
  \riso
  \D ^\dag _{\X } (\hdag D \cup Z') _\Q
\otimes ^{\L} _{\D ^\dag _{\X ^\#} (\hdag D\cup Z ') _\Q}
  \E (\ZZ) (\hdag Z').$$
On the other hand, by \ref{u+621}
(and, for the second isomorphism, since $\D ^\dag _{\X ^{\# \prime}} (\hdag D\cup Z ') _\Q=\D ^\dag _{\X ^\#} (\hdag D\cup Z ') _\Q $),
we get:
\begin{align}
\notag
  u _{D \cup Z'+} (\E (\hdag Z')) & \riso
\D ^\dag _{\X } (\hdag D \cup Z') _\Q
\otimes ^{\L} _{\D ^\dag _{\X ^{\# }} (\hdag D\cup Z ') _\Q}
  \E  (\hdag Z')(\ZZ),
  \\
  \notag
  v _{D \cup Z'+} (\E (\hdag Z')) & \riso
\D ^\dag _{\X } (\hdag D \cup Z') _\Q
\otimes ^{\L} _{\D ^\dag _{\X ^{\# }} (\hdag D\cup Z ') _\Q}
  \E  (\hdag Z')(\ZZ\cup \ZZ').
\end{align}
Since $\E (\hdag Z') (\ZZ \cup \ZZ') \riso
\E (\hdag Z') (\ZZ ') (\ZZ) \riso
\E  (\hdag Z')(\ZZ)
\riso \E  (\ZZ) (\hdag Z') $ (see \ref{Z-hdagDcomm}),
we conclude the proof of \ref{ZZ'u-equ1}.

\end{proof}

\begin{prop}
  \label{rhoisolemm}
  Let $\mathfrak{A} = \Spf \V \{ t _1,\dots, t_n\}$, $ D$ be a divisor of $\Spec k [ t _1,\dots, t _n]$
  and for $i =1,\dots, n$ let
  $\mathfrak{H} _i$ be the formal closed subscheme of $\mathfrak{A}$ defined by $t _i =0$, i.e.,
  $\mathfrak{H} _i  = \Spf \V \{ t _1,\dots, \widehat{t} _i ,\dots, t_n\}$. Let $\mathfrak{H} _0$ be the empty set.
  Fix an integer $r \in \{0,\dots, n\}$ and pose $\mathfrak{H} := \mathfrak{H} _0 \cup \mathfrak{H} _1 \cup \cdots \cup \mathfrak{H} _r$.
  Let $\mathfrak{A} ^\# := (\mathfrak{A},\mathfrak{H})$ and
$w$ : $\mathfrak{A} ^\# \rightarrow \mathfrak{A}$ be the canonical morphism.
Let $\E$ be a coherent
$\D_{\mathfrak{A}^\#}^\dag(\hdag D) _{ \Q}$-module which is a locally projective
$\O_{\mathfrak{A}}(\hdag D)_{ \Q}$-module of finite type such that
the conditions (a) and (b') in \ref{Nobuo} holds.
Then the canonical morphism
  $\rho$ : $w _{D+} (\E) \rightarrow \E (\hdag H)$ (see \ref{defi-rho}) is an isomorphism.
\end{prop}

\begin{proof}
We have to check $\R \underline{\Gamma} ^\dag _H w _{D+} (\E)=0$ (thanks to the exact triangle \ref{exacttriu+F}).
  To prove it, we will proceed by induction on $r$. When $r =0$, this is obvious.
  Suppose $r \geq 1$, pose $\mathfrak{H} ' = \cup _{r \geq i \geq 2} \mathfrak{H} _i$ (when $r=1$, $\mathfrak{H}'$ is empty) and $\G :=w _{D+} (\E)$.
  We get the Mayer-Vietoris exact triangle (see \cite[2.2.16]{caro_surcoherent}):
  \begin{equation}
    \label{MayerVietoris-rhoiso1}
\R \underline{\Gamma} ^\dag _{H _1 \cap H ' } \G (\hdag H _1)
\rightarrow
\R \underline{\Gamma} ^\dag _{H _1 } \G (\hdag H _1)
\oplus
\R \underline{\Gamma} ^\dag _{H ' } \G (\hdag H _1)
\rightarrow
\R \underline{\Gamma} ^\dag _{H _1 \cup H ' } \G (\hdag H _1)
\rightarrow
\R \underline{\Gamma} ^\dag _{H _1 \cap H ' } \G (\hdag H _1) [1].
  \end{equation}
  Since $\R \underline{\Gamma} ^\dag _{H _1 } \G (\hdag H _1)=0$
  and $\R \underline{\Gamma} ^\dag _{H _1 \cap H '} \G (\hdag H _1)=0$, we obtain
$\R \underline{\Gamma} ^\dag _{H ' } \G (\hdag H _1)
\riso
\R \underline{\Gamma} ^\dag _{H  } \G (\hdag H _1)$.

Let $\mathfrak{A} ^{\# \prime }:= (\mathfrak{A}, \mathfrak{H}')$,
$w'\,:\,\mathfrak{A} ^{\# \prime } \rightarrow \mathfrak{A}$
be the canonical map
and
$E := \sp ^* (\E)$.
By \ref{spcommDD'},
$\E (\hdag H _1)
\riso
\sp _* (j ^\dag _{Y _1\cap U} E)$,
where $U = \A _k ^n \setminus D$ and $Y_1= \A _k ^n \setminus H _1$.
Moreover, from \ref{preZZ'u}, $\D_{\mathfrak{A} ^{\#\prime}}^\dag(\hdag D \cup H _1) _{ \Q} =
\D_{\mathfrak{A} ^\#}^\dag(\hdag D \cup H _1) _{ \Q} $.
Then $\E (\hdag H _1)$ is
a coherent
$\D_{\mathfrak{A}^{\#\prime}}^\dag(\hdag D\cup H _1) _{ \Q}$-module which is a locally projective
$\O_{\mathfrak{A}}(\hdag D\cup H _1)_{ \Q}$-module of finite type satisfying
both conditions (a) and (b').
Using the induction hypothesis, this implies
$\R \underline{\Gamma} ^\dag _{H ' }  w ' _{D\cup H_1,+} (\E (\hdag H _1))=0$.
We get from \ref{ZZ'u-equ2} the isomorphism:
$w' _{D\cup H_1,+} (\E (\hdag H _1)) \riso (w _{D+} (\E )) (\hdag H _1)$.
Since
$\R \underline{\Gamma} ^\dag _{H ' } \G (\hdag H _1)
\riso
\R \underline{\Gamma} ^\dag _{H  } \G (\hdag H _1)$,
we obtain:
$\R \underline{\Gamma} ^\dag _{H  } \G (\hdag H _1)=0$.
Symmetrically, for any $i =1, \dots, r$, we check that
$\R \underline{\Gamma} ^\dag _{H  } \G (\hdag H _i)=0$.
With the exact triangle of localization of
$\R \underline{\Gamma} ^\dag _{H  } \G $ with respect to $H _i$, this means that the canonical morphism
$\R \underline{\Gamma} ^\dag _{H _i } \R \underline{\Gamma} ^\dag _{H  } \G
\rightarrow
\R \underline{\Gamma} ^\dag _{H  } \G $
is an isomorphism.
By \cite[2.2.8]{caro_surcoherent}, this implies:
$\R \underline{\Gamma} ^\dag _{H _1 \cap \cdots \cap H _r } \G \riso
\R \underline{\Gamma} ^\dag _{H  } \G $.

It remains to prove that $\R \underline{\Gamma} ^\dag _{H _1 \cap \cdots \cap H _r } \G = 0$.
When $D$ contains $H _1 \cap \cdots \cap H _r$, this is obvious.
This reduces us to the
case where $D \cap (H _1 \cap \cdots \cap H _r)$ is a divisor of $H _1 \cap \cdots \cap H _r$.

Let $\iota$ be the canonical closed immersion
$\mathfrak{H} _1 \cap \cdots \cap \mathfrak{H} _r= \Spf \V \{ t _{r+1},\dots, t_n\} \hookrightarrow
\Spf \V \{ t _{1},\dots, t_n\} =\mathfrak{A}$ and $g$ : $\mathfrak{A} \rightarrow \Spf \V \{ t _{r+1},\dots, t_n\} $ be
the canonical projection. We notice that $g \circ \iota$ is the identity.
Since $\E$ satisfies the conditions (a) and (b') and
$\G= w_{D+} (\E) $,
it follows from \ref{Nobuobis}
that
$g _{D+} \R \underline{\Gamma} ^\dag _{H } (\G)=0$
(notice that we do need here the relative case of \ref{Nobuobis},
i.e., $\T$ is not necessary equal to $\S$).
Hence, $g _{D+} \R \underline{\Gamma} ^\dag _{H _1 \cap \cdots \cap H _r } (\G)=0$.
By \cite[4.4.5]{Beintro2},
$\R \underline{\Gamma} ^\dag _{H _1 \cap \cdots \cap H _r } (\G)
\riso
\iota_{+} \iota ^! (\G) $.
Then:
$g _{D+} \R \underline{\Gamma} ^\dag _{H _1 \cap \cdots \cap H _r } (\G)
\riso
g _{+} \iota_{+} \iota ^! (\G) \riso \iota^! (\G)$.
Hence $\iota^! (\G)=0$
and then
$\R \underline{\Gamma} ^\dag _{H _1 \cap \cdots \cap H _r } (\G)=0$, which finishes the proof.

\end{proof}

We will need to extend \cite[6.11]{caro_log-iso-hol}, which will be essential
(in the proof of \ref{theorhoiso} or \ref{theorhoiso2}).
As for \cite[6.11]{caro_log-iso-hol}, we need a preliminary result:
\begin{lemm}
\label{I.3.6}
With the same notation as in \ref{preZZ'u},
let $X _i ^\#$ and $X _i ^{\# \prime}$ be respectively the reductions of
$\X ^\#$ and $\X ^{\# \prime}$ modulo $\pi ^{i+1}$.
  Let $\B _{X_i}$ be a $\D ^{(m)} _{X _i ^\#}$-module endowed with a compatible structure of $\O _{X _i}$-algebra.
We pose $\smash{\widetilde{\D}} ^{(m)} _{X _i ^\#}:= \B _{X _i} \otimes _{\O _{X _i}} \D ^{(m)} _{X _i ^\#}$,
$\smash{\widetilde{\D}} ^{(m)} _{X _i ^{\#\prime}}:= \B _{X _i} \otimes _{\O _{X _i}}  \D ^{(m)} _{X _i ^{\#\prime}}$.
Let $\E '$ be a left $\widetilde{\D} ^{(m)} _{X _i^{\# \prime}}$-module and $\E $ be a left $\widetilde{\D} ^{(m)} _{X _i ^\#}$-module.
Then the canonical morphism of $\widetilde{\D} ^{(m)} _{X _i ^\#}$-modules:
\begin{equation}
  \label{I.3.6iso}
\widetilde{\D} ^{(m)} _{X _i^\#} \otimes _{\widetilde{\D} ^{(m)} _{X _i^{\# \prime}}} (\E '\otimes _{ \B _{X _i}} \E)
  \rightarrow
(\widetilde{\D} ^{(m)} _{X _i^\#} \otimes _{\widetilde{\D} ^{(m)} _{X _i^{\# \prime}}} \E ' )\otimes _{ \B _{X _i}} \E
\end{equation}
is an isomorphism.
\end{lemm}

\begin{proof}
  Similar to \cite[3.6]{caro_log-iso-hol}.
\end{proof}

\begin{prop}
  \label{I.6.11}
With the same notation as in \ref{preZZ'u}, let $\tilde{u}$ : $\X ^{\# \prime}\rightarrow \X ^\#$ be the canonical morphism.
  Let $\E$ be a coherent $\D ^\dag _{\X ^\#} (\hdag D) _\Q$-module which is a locally projective $\O _{\X} (\hdag D) _\Q$-module of finite type.
Then $\E$ is also a coherent $\D ^\dag _{\X ^{\#\prime}} (\hdag D) _\Q$-module
  which is a locally projective $\O _{\X} (\hdag D) _\Q$-module of finite type.
  Furthermore we have the isomorphism of
$\D ^\dag _{\X ^\# }(\hdag D) _\Q$-modules
\begin{gather}
\label{I.6.11-iso}
\tilde{u} _{D +} (\E) \riso \D ^\dag _{\X ^\# }(\hdag D\cup Z' ) _\Q \otimes _{\D ^\dag _{\X ^\# }(\hdag D) _\Q}
\E =\E (\hdag Z').
\end{gather}
In particular, $\tilde{u} _{D +} (\E) $ (resp. $\E (\hdag Z')$)
can be endowed with a canonical structure of
coherent $\D ^\dag _{\X ^\# }(\hdag D\cup Z') _\Q$-modules (resp.
coherent $\D ^\dag _{\X ^\# }(\hdag D) _\Q$-modules).
\end{prop}

\begin{proof}
By \ref{Ddag=overcvcondition},
$\sp ^* (\E)$ is a locally free $j_U^\dag\O_{]X[_\X}$-module of finite type
with a logarithmic connection
$\nabla : E \rightarrow j_U^\dag\Omega_{\Xlogan/\S_K}^1 \otimes_{j_U^\dag\O_{]X[_\X}} E$
satisfying the overconvergent condition (see \ref{Ddag=overcvcondition}).
Then, we check that the
induced logarithmic connection
$\nabla ' : E \rightarrow j_U^\dag\Omega_{\X _K ^{\# \prime} /\S_K}^1 \otimes_{j_U^\dag\O_{]X[_\X}} E$
satisfies the overconvergent condition. So,
$\E$ is a coherent
$\D ^\dag _{\X ^{\#\prime}} (\hdag D) _\Q$-module
  which is a locally projective $\O _{\X} (\hdag D) _\Q$-module of finite type.

As for \cite[6.8]{caro_log-iso-hol}, we compute:
$\tilde{u} _{D +} (\O _{\X} (\hdag D) _\Q) \riso \O _{\X} (\hdag D \cup Z') _\Q .$
Then, in the same way as for the proof of \cite[6.11]{caro_log-iso-hol},
we deduce from \ref{I.3.6} that the isomorphism \ref{I.6.11-iso} holds.
\end{proof}

\begin{rema}
  \label{I.6.11-rema}
With the notation \ref{I.6.11}, it comes from \ref{ODdivcohe} and \ref{hdagDtildeDcoh}
that there is no ambiguity in writing $\E (\hdag Z')$. More precisely,
$$\D ^\dag _{\X ^{\# \prime } }(\hdag D\cup Z' ) _\Q \otimes _{\D ^\dag _{\X ^{\# \prime } }(\hdag D) _\Q}
\E\riso
\D ^\dag _{\X ^\# }(\hdag D\cup Z' ) _\Q \otimes _{\D ^\dag _{\X ^\# }(\hdag D) _\Q} \E
\riso \E (\hdag Z').$$

\end{rema}

\begin{lemm}
\label{h+finiteetalecondab}
  Let $h$ : $ \X ' \rightarrow \X$ be a finite \'etale morphism of smooth formal $\V$-schemes,
  $D ' = h ^{-1} (D)$, $\X ^{\prime \#}:=(\X ', h^{-1} (\ZZ))$,
  $h ^\#$ : $\X ^{\prime \#} \rightarrow \X ^\# $ be the induced morphism by $h$.
Let $\E'$ be a coherent
$\D_{\X^{\prime \#}}^\dag(\hdag D') _{ \Q}$-module which is a locally projective
$\O_{\X ^\prime }(\hdag D')_{ \Q}$-module of finite type.
Then $h ^\# _{D+}(\E')$ is a coherent
$\D_{\X^\#}^\dag(\hdag D) _{ \Q}$-module which is a locally projective
$\O_{\X}(\hdag D)_{ \Q}$-module of finite type.
Furthermore if
$\E'$ satisfies the conditions (a) and (b') of \ref{Nobuo},
so is $h ^\# _{D+}(\E')$.
\end{lemm}

\begin{proof}
Since $h ^\#$ is smooth, we have the canonical isomorphism
$\Omega _{\X^{\prime \#}/\X ^\#,\Q} ^\bullet \otimes _{\O _{\X,\Q}} \D ^\dag _{\X^{\prime \#},\Q}
[d _{\X^{\prime \#}/\X ^\#}]
  \riso
  \D ^\dag _{\X^{\#}\leftarrow \X^{\prime \#},\Q}$
  (see \ref{h+smooth}).
Since $h$ is even \'etale, we get $\Omega ^1 _{\X^{\prime \#}/\X ^\#} =0$ and then
$\D ^\dag _{\X^{\prime \#},\Q} \riso
\D ^\dag _{\X^{\#}\leftarrow \X^{\prime \#},\Q}$.
But $\R h _* = h _*$ because $h$ is finite.
This implies that $h ^\# _{D+}(\E')$ is canonically isomorphic to
  $h _* (\E')$.
  Pose $U ' := \X' \setminus D'$.
Recall that by \ref{Ddag=overcvcondition}
that $E ' := \sp ^* (\E')$ is
a locally free $j_{U'}^\dag\O_{]X'[_{\X'}}$-module of finite type
endowed with a logarithmic connection
$\nabla : E ' \rightarrow j_{U'}^\dag\Omega_{\X ^{\prime \# } _K /\S_K}^1 \otimes_{j_{U'}^\dag\O_{]X'[_{\X'}}} E'$
satisfying the overconvergent condition of \ref{logovcon}.
By hypothesis, $E'$ satisfies the conditions (a) and (b') of \ref{Nobuo}.
By \ref{remexp}.\ref{remexp-6}, then so is $h _* (E')$.
We conclude with the isomorphism:
$\sp _* h _* (E') \riso h _* \sp _*  (E') \riso h_* (\E')$.
\end{proof}

\begin{lemm}
  \label{h+=0is0}
  Let $h$ : $\PP \rightarrow \PP'$ be a finite and \'etale morphism of smooth formal $\V$-schemes,
  $D'$ be a divisor of $X'$, $D := h ^{-1}(D')$,
$\E \in
\smash{\underset{^{\longrightarrow}}{LD}} ^{\mathrm{b}} _{\Q,\mathrm{qc}}
( \smash{\widehat{\D}} _{\PP} ^{(\bullet)}(D))$.
Then $ h _+ (\E) = 0$ if and only if $\E= 0$.
\end{lemm}

\begin{theo}
\label{theorhoiso}
Let $\E$ be a coherent
$\D ^\dag _{\X^\#,\Q}$-module which is a locally projective
$\O_{\X,\Q}$-module of finite type such that
the conditions (a) and (b') in \ref{Nobuo} holds.
Then the canonical morphism
  $\rho$ : $u _{+} (\E) \rightarrow \E (\hdag Z)$ (see \ref{defi-rho}) is an isomorphism.
\end{theo}

\begin{proof}
This is equivalent to prove that $\R \underline{\Gamma} ^\dag _Z u _{+} (\E) =0$ (see \ref{exacttriu+F}).
We proceed by induction on the dimension of $X$.

\noindent $1^\circ$ {\it How to use the case \ref{rhoisolemm} of affine spaces.}

Let $x$ be a point of $\X$ and let $\ZZ _1, \dots, \ZZ _r$ be the irreducible components of $\ZZ$ which contain
$x$. By \cite[Theorem2]{Kedlaya-coveraffinebis}, there exist
an open dense subset $\U$ of $\X$ containing $x$ and
a finite \'etale morphism
$h_0$ : $U \rightarrow \A _k ^n$
such that
$\ZZ \cap \U = \ZZ _1 \cap \dots \cap \ZZ _r$ and
$Z _1,\dots, Z _r$ map by $h_0$ to coordinate hyperplanes $H _1,\dots, H _r$.
Since the theorem is local in $\X$, we can suppose that $\U = \X$.

Let $h$ : $\X \rightarrow \Spf \V \{ t _1,\dots, t_n\}$
be a lifting of $h _0$. Denote by $\HH _1,\dots, \HH _n$ the coordinate hyperplanes of
$\Spf \V \{ t _1,\dots, t_n\}$,
$\HH := \HH _1 \cup \cdots \cup \HH _r$, $\ZZ '':= h ^{-1} (\HH)$.
Let $\ZZ'$ be the union of the irreducible components of $\ZZ''$ which are not an irreducible component
of $\ZZ$.
Denote by $\X ^{\# \prime}= (\X, \ZZ'')$,
$\smash{\widehat{\A}} ^n _\V=\Spf \V \{ t _1,\dots, t_n\}$,
$\smash{\widehat{\A}} ^{n\#} _\V  = (\Spf \V \{ t _1,\dots, t_n\}, \HH)$,
$h ^\#$ : $\X ^{\# \prime} \rightarrow \smash{\widehat{\A}} ^{n\#} _\V$,
$w$ : $\smash{\widehat{\A}} ^{n\#} _\V \rightarrow \smash{\widehat{\A}} ^{n} _\V $,
$v$ : $\X ^{\# \prime} \rightarrow \X$.
We get the commutative diagram:
\begin{equation}
  \notag
  \xymatrix{
  {\X} \ar@{=}[r]  & {\X} \ar[r] ^-{h} & {\smash{\widehat{\A}} ^n _\V} \\
{(\X ,\ZZ)} \ar[u] ^-{u}  & {(\X ,\ZZ'')} \ar[l] _-{\tilde{u}} \ar[u] ^-{v} \ar[r] ^-{h ^\#} &
{\smash{\widehat{\A}} ^{n\#} _\V.} \ar[u] ^-{w}
}
\end{equation}

\noindent $2^\circ$ {\it The canonical morphism
$\R \underline{\Gamma} ^\dag _{Z\cap Z'} u _{+} (\E) \rightarrow
\R \underline{\Gamma} ^\dag _Z u _{+} (\E)$
is an isomorphism.}

We notice (for example see \ref{I.6.11}) that
$\E$ is also a coherent
$\D ^\dag _{\X^{\#\prime},\Q}$-module which is a locally projective
$\O_{\X,\Q}$-module of finite type.
By \ref{h+finiteetalecondab},
since $h $ is finite and \'etale,
$h ^\#  _{+}  (\E )$ is a coherent
$\D^\dag _{\smash{\widehat{\A}} ^{n\#} _\V,\Q}$-module which is a locally projective
$\O_{\smash{\widehat{\A}} ^{n} _\V ,\Q}$-module of finite type and which satisfies
both conditions (a) and (b') of \ref{Nobuo}.
Hence, by \ref{rhoisolemm},
$\R \underline{\Gamma} ^\dag _{H} w _{+} h ^\#  _{+}  (\E)=0$.
We have:
$h _{+} (\R \underline{\Gamma} ^\dag _{Z''} v _{+} (\E ))
\riso
\R \underline{\Gamma} ^\dag _{H} h _{+} v _{+} (\E)
\riso
\R \underline{\Gamma} ^\dag _{H} w _{+} h ^\#  _{+}  (\E)$
(see \cite[2.2.18.2]{caro_surcoherent} for the first isomorphism and
\ref{trans+} for the second one).
Then, by \ref{h+=0is0}:
$\R \underline{\Gamma} ^\dag _{Z''} v _{+} (\E )=0$.
Since $Z \subset Z''$, we get:
$\R \underline{\Gamma} ^\dag _{Z} v _{+} (\E )=0$.

It follows from \ref{I.6.11-iso}:
$\E (\hdag Z')\riso\tilde{u} _{+} (\E)  $. Then, by \ref{trans+}:
$u _+ (\E (\hdag Z')) \riso u _+ \tilde{u} _{+} (\E) \riso v _+ (\E)$.
This implies $\R \underline{\Gamma} ^\dag _{Z} u _{+} (\E (\hdag Z'))=0$.
By \ref{u+commhdag-equ},
$u _{+} (\E (\hdag Z')) \riso (u _{+} (\E )) (\hdag Z').$
Hence:
$\R \underline{\Gamma} ^\dag _Z (\hdag Z')  u _{+} (\E)=0$.
Using the exact triangle of localization of $\R \underline{\Gamma} ^\dag _{Z} u _{+} (\E) $ with respect to $Z'$,
this means that the canonical morphism
$\R \underline{\Gamma} ^\dag _{Z} \R \underline{\Gamma} ^\dag _{Z'} u _{+} (\E)
\rightarrow
\R \underline{\Gamma} ^\dag _Z u _{+} (\E)$
is an isomorphism.
Since $\R \underline{\Gamma} ^\dag _{Z\cap Z'} u _{+} (\E)\riso
\R \underline{\Gamma} ^\dag _{Z} \R \underline{\Gamma} ^\dag _{Z'} u _{+} (\E)$
(see \cite[2.2.8]{caro_surcoherent}),
we come down to prove
$\R \underline{\Gamma} ^\dag _{Z ' \cap Z} u _{+} (\E)=0$.

\noindent $3^\circ$ {\it We check that $\R \underline{\Gamma} ^\dag _{Z ' \cap Z} u _{+} (\E)=0$.}

When $Z \cap Z'$ is empty, this is obvious.
It remains to deal with the case where $Z\cap Z'$ is not empty.
Let $x$ be a closed point of $Z\cap Z'$,
$\ZZ _1,\dots , \ZZ _r$ be the irreducible components of $\ZZ$ containing $x$,
$\ZZ _{r+1}, \dots , \ZZ _s$ be the irreducible components of $\ZZ'$ containing $x$.
Since $\R \underline{\Gamma} ^\dag _{Z ' \cap Z} u _{+} (\E)$ is zero outside $Z \cap Z'$,
it is sufficient to prove its nullity around $x$.
Then, we can suppose that
$\ZZ = \ZZ _1 \cup \cdots \cup \ZZ _r$ and $\ZZ ' = \ZZ _{r+1} \cup \cdots \cup \ZZ _s$.

To end the proof,
we need the following lemma.
\begin{lemm2}
\label{lemmtheorhoiso}
With the above notation,
$\X'$ be an intersection of some irreducible components of $\ZZ'$.
Let $\X ^{\prime \#} := (\X', \X' \cap \ZZ)$,
$\iota$ : $\X' \hookrightarrow \X$,
$\iota ^\#$ : $\X ^{\prime \#} \hookrightarrow \X ^\#$,
$u'$ : $\X ^{\prime \#} \rightarrow \X '$
be the canonical morphisms.
For any $\E ^{(\bullet)} \in \smash{\underset{^{\longrightarrow}}{LD}} ^{\mathrm{b}} _{\Q ,\mathrm{qc}}
  ( \smash{\widehat{\D}} _{\X ^{\#}} ^{(\bullet)})$,
   we have the canonical isomorphism:
   $\iota ^! u _+ (\E^{(\bullet)}) \riso u ' _+  \iota ^{\# ! }( \E^{(\bullet)}) $.
\end{lemm2}

\begin{proof}
We keep the notation
of the section \ref{nota-1.2}, e.g., $X ' _i$ means the reduction modulo $\pi ^{i+1}$ of $\X'$ etc.
From $\D ^{(m)} _{X _i \leftarrow X _i^\#}  \riso \D ^{(m)} _{X _i} (Z _i)$ (see \cite[5.2.4]{caro_log-iso-hol})
and by \cite[5.1.2]{caro_log-iso-hol},
we get $\D ^{(m)} _{X _i \leftarrow X _i^\#} \otimes ^\L _{\D ^{(m)} _{X _i^\#}} \E _i ^{(m)}
\riso
\D ^{(m)} _{X _i} \otimes ^\L _{\D ^{(m)} _{X _i^\#}} \E _i ^{(m)} (Z _i)$.
Thus:
$$\D ^{(m)} _{X'   _{ i} \rightarrow X _i} \otimes ^\L
_{\iota^{\text{-}1}\D ^{(m)} _{X _i}} \iota^{\text{-}1} (\D ^{(m)}
_{X _i \leftarrow X _i^\#} \otimes ^\L _{\D ^{(m)} _{X _i^\#}} \E _i ^{(m)}) \riso \D ^{(m)} _{X'  _{ i} \rightarrow X _i} \otimes ^\L
_{\iota^{\text{-}1}  \D ^{(m)} _{X _i^\#}} \iota^{\text{-}1}  \E _i ^{(m)}
(Z_i).$$

  The canonical morphism
  $\D ^{(m)} _{X _i^\#} \rightarrow \D ^{(m)} _{X _i}$ induces
  the morphism of $(\D ^{(m)} _{X _{ i} ^{\prime \#}}, \iota^{-1}\D ^{(m)} _{X _i^\#})$-bimodules:
$\D ^{(m)} _{X ^{\prime \#}  _{i} \rightarrow X _i^\#}
\rightarrow \D ^{(m)} _{ X' _{ i} \rightarrow  X _i}$.
We get:
$\D ^{(m)} _{X'  _{ i} } \otimes _{\D ^{(m)} _{X ^{\prime \#} _{ i} }}
\D ^{(m)} _{X ^{\prime \#}  _{ i} \rightarrow X _i^\#} \rightarrow \D ^{(m)} _{ X' _{ i} \rightarrow  X _i}$.
By a computation in local coordinates, we check that this morphism is an isomorphism.
Since $\D ^{(m)} _{X ^{\prime \#}  _{ i} \rightarrow X _i^\#} $ is locally free over
$\D ^{(m)} _{X ^{\prime \#}  _{ i} }$,
we obtain:
$\D ^{(m)} _{X ' _{ i} } \otimes ^\L _{\D ^{(m)} _{X ^{\prime \#} _{ i} }}
\D ^{(m)} _{X ^{\prime \#}  _{ i} \rightarrow X _i^\#} \riso \D ^{(m)} _{ X ' _{ i} \rightarrow  X _i}$.
This implies:
$$\D ^{(m)} _{X'  _{ i} \rightarrow X _i} \otimes ^\L _{\iota^{\text{-}1}  \D ^{(m)} _{X _i^\#}} \iota^{\text{-}1}  \E _i ^{(m)} (Z_i)
\riso
(\D ^{(m)} _{X'  _{ i} } \otimes ^\L _{\D ^{(m)} _{X ^{\prime \#}  _{ i} }}
\D ^{(m)} _{X ^{\prime \#}  _{ i} \rightarrow X _i^\#} )
\otimes ^\L _{\iota^{\text{-}1}\D ^{(m)} _{X _i^\#}}
\iota^{\text{-}1}  (\E _i ^{(m)}  (Z _i ))
.$$
Moreover,
$\D ^{(m)} _{X ^{\prime \#}  _{ i} \rightarrow X _i^\#}
\otimes ^\L _{\iota^{\text{-}1}\D ^{(m)} _{X _i^\#}}
\iota^{\text{-}1}  (\E _i ^{(m)}  (Z _i ))
\riso
(\D ^{(m)} _{X ^{\prime \#} _{ i} \rightarrow X _i^\#} \otimes ^\L _{\iota^{\text{-}1}\D ^{(m)} _{X _i^\#}}
\iota^{\text{-}1}  \E _i ^{(m)} ) (Z _i \cap X ' _{ i})$.
From
$\D ^{(m)} _{X'  _{ i} \leftarrow X ^{\prime \#}  _{ i} }\riso
\D ^{(m)} _{X '  _{ i} } (Z _i \cap X ' _{ i})$ (see \cite[5.2.4]{caro_log-iso-hol})
and using the commutation of the functor `$-(Z _i \cap X ' _{ i})$'
with `$-\otimes ^\L _{\D ^{(m)} _{X ^{\prime \#}  _{ i} }}-$'
(see \cite[5.1.2]{caro_log-iso-hol}),
we obtain:
$$
\D ^{(m)} _{X'  _{ i} } \otimes ^\L _{\D ^{(m)} _{X ^{\prime \#}  _{ i} }}
\left ( (\D ^{(m)} _{X ^{\prime \#}  _{ i}
\rightarrow X _i^\#} \otimes ^\L _{\iota^{\text{-}1}\D ^{(m)} _{X _i^\#}}
\iota^{\text{-}1}  \E _i ^{(m)} ) (Z _i \cap X ' _{ i})
\right)
\riso
\D ^{(m)} _{X '  _{ i} \leftarrow X ^{\prime \#}  _{ i} }
\otimes  ^\L _{\D ^{(m)} _{X ^{\prime \#}  _{ i} }}
(\D ^{(m)} _{X ^{\prime \#}  _{ i} \rightarrow X _i^\#} \otimes ^\L _{\iota^{\text{-}1}\D ^{(m)} _{X _i^\#}}
\iota^{\text{-}1}  \E _i ^{(m)} ).
$$
Then, we get by composition:
$\D ^{(m)} _{X ' _{i} \rightarrow X _i} \otimes ^\L _{\iota^{\text{-}1}\D ^{(m)} _{X _i}}
\iota^{\text{-}1} (\D ^{(m)} _{X _i \leftarrow X _i^\#} \otimes ^\L _{\D ^{(m)} _{X _i^\#}} \E _i ^{(m)})
\riso
\D ^{(m)} _{X'  _{ i} \leftarrow X ^{\prime \#}  _{i} }
\otimes  ^\L _{\D ^{(m)} _{X ^{\prime \#}  _{i} }}
(\D ^{(m)} _{X ^{\prime \#}  _{i} \rightarrow X _i^\#} \otimes ^\L _{\iota^{\text{-}1}\D ^{(m)} _{X _i^\#}}
\iota^{\text{-}1}  \E _i ^{(m)} )
$,
which is up to a shift the required isomorphism at the level $m$.
\end{proof}
In particular, let $\ZZ _s ^\# := (\ZZ _s, \ZZ _s \cap \ZZ)$,
$\iota$ : $\ZZ _s \hookrightarrow \X$,
$\iota ^\#$ : $\ZZ _s ^\# \hookrightarrow \X ^\#$,
$u'$ : $\ZZ _s ^\# \rightarrow \ZZ _s$
be the canonical morphisms.
We obtain: $\R \underline{\Gamma} ^\dag _{Z _s \cap Z} u _{+} (\E)
\riso \R \underline{\Gamma} ^\dag _{Z} \iota_+ \iota^! u _{+} (\E)
\underset{\ref{lemmtheorhoiso}}{\riso} \R \underline{\Gamma} ^\dag
_{Z} \iota_+ u ' _{+} \iota^{\#!}  (\E) \riso \iota_+  \R
\underline{\Gamma} ^\dag _{Z\cap Z_s } u'  _{+} \iota^{\#!}  (\E)$
(see \cite[4.4.5]{Beintro2} for the first isomorphism).
Since $\E$ is flat over $\O _{\X,\Q}$, then: $\iota^{\#!} (\E)
[1] \riso \iota^{\#*} (\E) $. Since $\iota^{\#*} (\E)$ is a coherent
$\D ^\dag _{\ZZ ^\# _s ,\Q}$-module which is a locally projective
$\O_{\ZZ _s,\Q}$-module of finite type and which satisfies
conditions (a) and (b') of \ref{Nobuo} (see the proof of \ref{overconv}), since $\dim Z _s <
\dim X$, the induction hypothesis implies that $\R
\underline{\Gamma} ^\dag _{Z\cap Z_s } u' _{+} \iota^{\#!}  (\E) =0$.
Then: $\R \underline{\Gamma} ^\dag _{Z _s \cap Z} u _{+} (\E)=0$.
Similarly, we check that, for any $j$ between $r +1$ and $s$, $\R
\underline{\Gamma} ^\dag _{Z _j \cap Z} u _{+} (\E)=0$. Hence,
using Mayer-Vietoris exact triangles (see \cite[2.2.16]{caro_surcoherent}),
$\R \underline{\Gamma} ^\dag _{Z ' \cap Z} u _{+} (\E)=0$.
\end{proof}

\begin{exam}
The exponents of an overconvergent isocrystals with nilpotent residues (see \ref{unipotent}) are zero.
Then it follows from \ref{theorhoiso} the holonomicity of
overconvergent isocrystals with unipotent monodromy along $Z$.
\end{exam}

\begin{prop}
\label{lemm-diag-coroNobuo3}
Let $\E \in D ^\mathrm{b} _\mathrm{coh} (\D ^\dag _{\X ^\#} (\hdag D) _\Q)$.
Suppose that there exist a smooth morphism $\X \rightarrow \T$ of smooth formal $\V$-schemes over $\S$
such that $\ZZ $ is a relatively strict normal crossing divisor of $\X$ over $\T$. Then, we have the canonical quasi-isomorphism:
  \begin{equation}
  \label{diag-coroNobuo3}
  \Omega _{\X ^\#/\T,\Q} ^\bullet \otimes _{\O _{\X,\Q}}  \E
  \riso
  \Omega _{\X /\T,\Q} ^\bullet \otimes _{\O _{\X,\Q}}
  u_{D+} (\E ).
\end{equation}
\end{prop}
\begin{proof}
The proof is similar to that of \cite[6.3]{caro_log-iso-hol}.
\end{proof}

The second part of the next corollary improves the statements of \ref{Nobuorigid} (or \ref{Nobuo}):
\begin{theo}
\label{coro-theorhoiso}
  Let $\E$ be a coherent
$\D_{\X^\#,\Q}^\dag$-module which is a locally projective
$\O_{\X,\Q}$-module of finite type and which satisfies conditions (a) and (b') of
\ref{Nobuo}.
Then $\E (\hdag Z)$ is a holonomic
$\D ^\dag _{\X,\Q}$-module.

Moreover, suppose that there exist a smooth morphism $\X \rightarrow \T$ of smooth formal $\V$-schemes over $\S$
such that $\ZZ $ is a relatively strict normal crossing divisor of $\X$ over $\T$.
Then the canonical morphism
$\Omega _{\X ^\#/\T,\Q} ^\bullet \otimes _{\O _{\X,\Q}}  \E
\rightarrow
\Omega _{\X /\T,\Q} ^\bullet \otimes _{\O _{\X,\Q}}
  \E (\hdag Z)$
is a quasi-isomorphism.
\end{theo}

\begin{proof}
  The first assertion is a consequence of \cite[5.25]{caro_log-iso-hol}
  and the second one follows from
  \ref{theorhoiso} and \ref{lemm-diag-coroNobuo3}.
\end{proof}

We finish this section by checking that the conclusions of theorems \ref{theorhoiso} (and then \ref{coro-theorhoiso})
are stable under inverse image by smooth morphisms.

\begin{prop}
\label{u+f+smooth}
  Let $f$ : $\X '\rightarrow \X$ be a smooth morphism of smooth formal $\V$-schemes,
  $\ZZ':= f ^{-1} (\ZZ)$,
   $\X ^{\prime \#}= (\X', \ZZ')$,
   $u'$ : $\X ^{\prime \#} \rightarrow \X '$ be the canonical morphisms,
  $f ^\#$ : $\X ^{\prime \#} \rightarrow \X ^\#$ be the morphism induced by $f$.
Let $\E$ be a coherent
$\D ^\dag _{\X^\#,\Q}$-module which is a locally projective
$\O_{\X,\Q}$-module of finite type.
Then we have the canonical isomorphism:
\begin{equation}
  f ^* u_+ (\E) \riso u' _+ f ^{\# *} (\E).
\end{equation}
\end{prop}

\begin{proof}
  We have:
  $u' _+ f ^{\# *} (\E) \riso \D ^\dag _{\X^{\prime },\Q}
  \otimes _{\D ^\dag _{\X^{\prime \#},\Q}} (\D ^\dag _{\X ^{\prime \#} \rightarrow \X^\#,\Q}
  \otimes _{f ^{-1} \D ^\dag _{\X^\#,\Q}} f ^{-1}  \E) (\ZZ')$ (see \ref{defi-rho} for the direct image).
  The canonical morphism $\D ^\dag _{\X ^{\prime \#} \rightarrow \X^\#,\Q}
  \rightarrow \D ^\dag _{\X ^{\prime } \rightarrow \X,\Q}$ induces
  the morphism of coherent $\D ^\dag _{\X^{\prime },\Q}$-modules (which are also
  $(\D ^\dag _{\X^{\prime },\Q}, f ^{-1} \D ^\dag _{\X^{\# },\Q})$-bimodules)
$\D ^\dag _{\X^{\prime },\Q}
  \otimes _{\D ^\dag _{\X^{\prime \#},\Q}} \D ^\dag _{\X ^{\prime \#} \rightarrow \X^\#,\Q}
  \rightarrow
\D ^\dag _{\X ^{\prime } \rightarrow \X,\Q}$.
We compute that this morphism is an isomorphism (we come down to the case of log-schemes which corresponds to
a computation in local coordinates).
Then:
$$u' _+ f ^{\# *} (\E) \riso \D ^\dag _{\X ^{\prime } \rightarrow \X,\Q}
  \otimes _{f ^{-1} \D ^\dag _{\X^\#,\Q}}  f ^{-1}  \E  (\ZZ')
  \riso
  \D ^\dag _{\X ^{\prime } \rightarrow \X,\Q}
  \otimes _{f ^{-1} \D ^\dag _{\X,\Q}}
  f ^{-1} (\D ^\dag _{\X,\Q}  \otimes _{\D ^\dag _{\X^\#,\Q}} \E  (\ZZ))
  \riso
  f ^* u_+ (\E) .$$
\end{proof}

\begin{coro}
\label{coro-u+f+smooth}
  With the notation of \ref{u+f+smooth}, if the morphism
  $u_+ (\E) \rightarrow \E (\hdag Z)$ is an isomorphism then
  so is
  $u ' _+ ( f ^{\# *} (\E) ) \rightarrow f ^{\# *} (\E)  (\hdag Z')$.
\end{coro}

\subsection{Overholonomicity of overconvergent $F$-isocrystals}
\label{subsection23}

\begin{defi}
\label{defi-smthdevi}
Let $\X$ be a smooth formal $\V$-scheme.
\begin{enumerate}
\item \label{defi-smthdevi-1}
Let
$\E ^{(\bullet)} \in \smash{\underset{^{\longrightarrow}}{LD}} ^{\mathrm{b}} _{\Q ,\mathrm{qc}}
  ( \smash{\widehat{\D}} _{\X } ^{(\bullet)})$.
Let $Y$ be a subscheme of $X$ such that there exists a divisor $T$ of $X$
satisfying $Y = \overline{Y} \setminus T$, where $\overline{Y}$ is the closure of $Y$ in $X$.
The complexe $\E ^{(\bullet)}$ {\it is smoothly devissable over $Y$ in partially overconvergent isocrystals}
if there exist some divisors $T _1, \dots,T _{r}$ containing $T$ with $T _r =T$ such that,
for any $i: =0,\dots, r-1$ and posing $T _0:=\smash{\overline{Y}}$,
$Y _i := T _0 \cap T _1 \cap \dots \cap T _i \setminus T _{i+1}$,
we have $\overline{Y } _i$ smooth and the cohomological spaces of
  $\underset{\longrightarrow}{\lim} \R \underline{\Gamma} ^\dag _{Y _i} (\E ^{(\bullet)})$ (see \cite[3.2.1]{caro-2006-surcoh-surcv})
  are in the essential image of the functor $\sp _{\overline{Y} _i \hookrightarrow \X, T _{i+1}, +}$,
  where $\sp _{\overline{Y} _i \hookrightarrow \X, T _{i+1}, +}$ is the canonical fully faithful
  functor from the category
  of isocrystals on $Y _i$ overconvergent along $\overline{Y} _i  \setminus Y _i$ to the category
  of coherent $\D ^\dag _{\X} (\hdag T _{i+1}) _{\Q}$-modules (see \cite{caro-construction}).
  To simplify the notation, it is possible to avoid
  $\underset{\longrightarrow}{\lim} $ indicating.

More precisely, we can say that the complex {\it $\E ^{(\bullet)}$ is smoothly devissable over the stratification
$Y =\sqcup _{i= 0,\dots , r-1} Y _i$ in partially overconvergent isocrystals}
or {\it $(T _1, \dots, T_r)$ gives a smooth devissage over $Y$ of $\E ^{(\bullet)}$
in partially overconvergent isocrystals}.

\item  \label{defi-smthdevi-2} \label{rema-rhoisolemm2}
Let $D$ be a divisor of $X$, $\E \in D ^\mathrm{b} _\mathrm{coh} (\D ^\dag _{\X } (\hdag D) _{\Q})$ and
$\E ^{(\bullet)} \in \smash{\underset{^{\longrightarrow}}{LD}} ^{\mathrm{b}} _{\Q ,\mathrm{coh}}
  ( \smash{\widehat{\D}} _{\X } ^{(\bullet)}(D))$
  such that $\underset{\longrightarrow}{\lim} (\E ^{(\bullet)} ) \riso \E$
  (this has a meaning since $\underset{\longrightarrow}{\lim}$ induces
  the equivalence of categories $\smash{\underset{^{\longrightarrow}}{LD}} ^{\mathrm{b}} _{\Q ,\mathrm{coh}}
  ( \smash{\widehat{\D}} _{\X } ^{(\bullet)}(D)) \cong D ^\mathrm{b} _\mathrm{coh} (\D ^\dag _{\X } (\hdag D) _{\Q})$).

  We say that $\E$ is smoothly devissable in partially overconvergent isocrystals
  if $\E ^{(\bullet)} $ is smoothly devissable over $X \setminus D$ in partially overconvergent isocrystals.

Let $T _1,\dots , T _r$ be some divisors of $X$ such that $T _r$ is empty.
We pose, for $i=0,\dots, r$, $T '_i := T_i \cup D$.
We say that
  {\it $(T _1, \dots, T_r)$ (resp. $(T '_1,\dots , T '_r)$) gives a smooth devissage of $\E $
  over $X$ (resp. $X \setminus D$) in
partially overconvergent isocrystals}
if $(T _1, \dots, T_r)$ (resp. $(T '_1,\dots , T '_r)$)
gives a smooth devissage over $X$ (resp. $X \setminus D$) of $\E ^{(\bullet)}$ in
partially overconvergent isocrystals.
\end{enumerate}
\end{defi}

\begin{rema}
\label{rema-smthdevi}
\begin{enumerate}
  \item With the notation \ref{defi-smthdevi}.\ref{defi-smthdevi-1}, for any $i = 0, \dots ,r$,
  let $X _i := T _0 \cap T _1 \cap \dots \cap T _i$. Then, for any $i = 0, \dots ,r-1$,
  the exact triangle of localization of
$\R \underline{\Gamma} ^\dag _{X _i} (\E ^{(\bullet)})$ with respect to $T _{i+1}$ is
$$\R \underline{\Gamma} ^\dag _{X _{i+1}} (\E ^{(\bullet)})
\to \R \underline{\Gamma} ^\dag _{X _i} (\E ^{(\bullet)}) \to
\R \underline{\Gamma} ^\dag _{Y _i } (\E ^{(\bullet)})\to
\R \underline{\Gamma} ^\dag _{X _{i+1}} (\E ^{(\bullet)}) [1],$$
which explains the word ``devissage''.

\item With the notation \ref{defi-smthdevi}.\ref{defi-smthdevi-2},
we pose $T _0:= X$.
Since $\E ^{(\bullet)} \riso (\hdag D) (\E ^{(\bullet)} ) $, we notice that
$  \R \underline{\Gamma} ^\dag _{T _0 \cap T _1 \cap \dots \cap T _i }
  \circ
  (\hdag T _{i+1}) (\E ^{(\bullet)} )
  \riso
  \R \underline{\Gamma} ^\dag _{T _0 \cap T '_1 \cap \dots \cap T '_i }
  \circ  (\hdag T '_{i+1}) (\E ^{(\bullet)} )$.
  Then $(T _1, \dots, T_r)$ gives a smooth devissage of $\E $
  over $X$ in partially overconvergent isocrystals iff
 $(T '_1,\dots , T '_r)$ gives a smooth devissage of $\E $
  over $X \setminus D$
  in partially overconvergent isocrystals.

\end{enumerate}
\end{rema}

\begin{empt}
  \label{n1n2dev}
Similarly to \cite[3.2.7--8]{caro-2006-surcoh-surcv}, we have the following result.
Let $\X$ be a smooth formal $\V$-scheme,
$Y$ a subscheme of $X$. We suppose that there exists a divisor $T$ of $X$
such that $Y = \overline{Y} \setminus T$.
Let $\E \in F \text{-} \smash[b]{\underset{^{\longrightarrow }}{LD }}  ^\mathrm{b} _{\Q, \mathrm{qc}}
(\overset{^\mathrm{g}}{} \smash{\widehat{\D}} _{\PP} ^{(\bullet)})$.
Let $T _1,\dots, T _{r}$ be some divisors of $P$ containing $T$ with $T _r =T$ and,
for any $i: =0,\dots, r-1$,
$Y _i := T _0 \cap T _1 \cap \dots \cap T _i \setminus T _{i+1}$ where $T _0:=\smash{\overline{Y}}$.

If, for any $i: =0,\dots, r-1$, $\E$
is smoothly devissable over $Y _i$ in partially overconvergent isocrystals then
so is $\E$ over $Y$.

More precisely,
for any $i =0,\dots , r-1$,
let $T _{(i, 1)}, \dots,T _{(i, r _i)}$ be some divisors containing
$T _{i+1}$ with $T _{(i, r _i)}=T _{i+1}$ such that,
if $T _{(i, 0)} :=\overline{Y} _i $ and,
for any $h=0,\dots , r_i-1$,
  $Y _{(i,h)}:= T_{(i,0)} \cap \dots \cap T _{(i,h)} \setminus T _{(i,h+1)}$,
  then $\overline{Y _{(i,h)}}$ is smooth and, for any integer $j$,
  $\mathcal{H} ^j (\underset{\longrightarrow}{\lim}
  \R \underline{\Gamma} ^\dag _{Y _{(i,h)}}  \E  )$ is in the essential image of
  $\sp _{\overline{Y _{(i,h)}} \hookrightarrow X, T _{(i,h+1)},+}$.

Then $(T _{(0, 1)}, \dots,T _{(0, r _0)},
T _{(1, 1)}, \dots,T _{(1, r _1)},\dots ,
T _{(r-1, 1)}, \dots,T _{(r-1, r _{r-1} )})$
gives a smooth devissage of $\E$ in partially overconvergent isocrystals
over the stratification
\small
\begin{equation}
\label{strat-comp}
  Y = Y _{(0,0)} \sqcup \dots \sqcup Y _{(0,r _0 -1)} \bigsqcup
Y _{(1,0)} \sqcup \dots \sqcup Y _{(1,r _1 -1)} \bigsqcup
\dots \bigsqcup Y _{(r -1,0)} \sqcup \dots \sqcup Y _{(r- 1,r _{r-1} -1)}.
\end{equation}
\normalsize
\end{empt}

\begin{prop}
  \label{rhoisolemm2}
  Let $\mathfrak{A} = \Spf \V \{ t _1,\dots, t_n\}$
  and, for $i =1,\dots, n$, let
  $\mathfrak{H} _i$ be the formal closed subscheme of $\mathfrak{A}$ defined by $t _i =0$, i.e.,
  $\mathfrak{H} _i  = \Spf \V \{ t _1,\dots, \widehat{t} _i ,\dots, t_n\}$.
  Let $I$ and $I'$ be two subsets of $\{ 1,\dots , n\}$ such that $I \cap I'$ is empty.
  We pose $\mathfrak{H} := \cup _{i \in I} \mathfrak{H} _i$ and
  $\mathfrak{H} ' := \cup _{i '\in I'} \mathfrak{H} _{i'}$.
  Let $\mathfrak{A} ^\# := (\mathfrak{A},\mathfrak{H})$ and
$w$ : $\mathfrak{A} ^\# \rightarrow \mathfrak{A}$ be the canonical morphism.

Then there exist some divisors $T _1, \dots, T_N$,
only depending on $I$ and $I'$,
which satisfies the following property:
if $\E ^\bullet $ is any bounded complex of coherent
$\D_{\mathfrak{A}^\#}^\dag(\hdag H') _{ \Q}$-modules, locally projective of finite type
as $\O_{\mathfrak{A}}(\hdag H')_{ \Q}$-module and such that conditions (a) and (b) of
\ref{Nobuorigid} holds,
then $T _1, \dots, T_N$ gives
a
smooth devissage of $w _{H '+} (\E ^\bullet )$ in partially overconvergent isocrystals
over $\A _k^n $.

Moreover $T _1 = H$ and any divisor $T_1,\dots, T_N$ is a sub-divisor of $H$.

\end{prop}

\begin{proof}

\noindent $0^\circ$ {\it Induction.}

For the sake of convenience, we add the case $n=0$ where $\mathfrak{A}=\Spf \V$ (and then $I$ and $I'$ are empty).
  We proceed by induction on the lexicographic order $(n, |I|)$, with $n\geq 0$.
  The case $n=0$ is obvious. So we can suppose that $n \geq1$ and the proposition is checked for $n-1$.
  Moreover, the case where $|I|=0 $ means that $H $ is empty. This case is thus straightforward.
  So, we come down to treat the case $|I|\geq 1$. Up to a re-indexation, we can suppose $1 \in I$.

\noindent $1^\circ$ {\it We come down to the case where $\E^\bullet $ is a module.}

So, suppose here that there exist some divisors $T _1, \dots, T_N$ such that, for any
coherent
$\D_{\mathfrak{A}^\#}^\dag(\hdag H') _{ \Q}$-module $\E$, locally projective of finite type
as $\O_{\mathfrak{A}}(\hdag H')_{ \Q}$-module and satisfying above $(a), (b)$ conditions,
$T _1, \dots, T_N$ give a
smooth devissage of $w _{H '+} (\E  )$ in partially overconvergent isocrystals
over $\A _k^n$.

Following \cite[5.25.1]{caro_log-iso-hol},
for any coherent $\D_{\mathfrak{A}  ^\#}^\dag(\hdag H') _{ \Q}$-module $\E $,
locally projective of finite type as $\O_{\mathfrak{A}}(\hdag H')_{ \Q}$-module,
for any $j\not =0$, $\H ^j (w_{+} (\E) )=0$.
We pose $\FF ^\bullet : = w _{H '+} (\E ^\bullet )$. Then, for any integer $r$,
$\FF ^r = w _{H '+} (\E ^r )$.

For any $i: =0,\dots, r-1$, let
$Y _i := T _0 \cap T _1 \cap \dots \cap T _i \setminus T _{i+1}$ (with $T _0:=\smash{\overline{Y}}$)
and pose $\Phi:= \underline{\Gamma} ^\dag _{Y _i}=\underline{\Gamma} ^\dag _{T _0 \cap T _1 \cap \dots \cap T _i }
  \circ
  (\hdag T _{i+1})$.
Then, the first spectral sequence of hypercohomology of $\Phi$
gives $E _{1} ^{r,s}=\H ^s (\R \phi (\FF ^r)) \Rightarrow \H ^n (\R  \phi (\FF^\bullet ))$.
If for any $r,s$, $\H ^s (\R \phi (\FF ^r))$ is an isocrystal on $Y _i$ overconvergent along $\overline{Y _i} \setminus Y_i$,
then so is $\H ^n (\R  \phi (\FF^\bullet ))$.
Then we can suppose that $\FF^\bullet$ has only term.
Thus, $\E ^\bullet$ has only a term.
From now, we will write $\E$ instead of $\E ^\bullet$.

\noindent $2^\circ$ {\it Devissage.}

Via the exact triangle of localization of $w _{H '+} (\E)$
with respect to $H $, it is sufficient to check that
$\R \underline{\Gamma} ^\dag _H w _{H '+} (\E )$
is
smoothly devissable in partially overconvergent isocrystals.

  The exact triangle of localization of $\R \underline{\Gamma} ^\dag _H w _{H '+} (\E)$ with respect to
  $H _1$ is of the form:
  \begin{equation}
\label{rhoisolemm2-triloc}
\R \underline{\Gamma} ^\dag _{H _1} w _{H '+} (\E)
\rightarrow
\R \underline{\Gamma} ^\dag _H w _{H '+} (\E)
\rightarrow
(\hdag H _1) \R \underline{\Gamma} ^\dag _H w _{H '+} (\E )
\rightarrow
\R \underline{\Gamma} ^\dag _{H_1} w _{H '+} (\E) [1].
  \end{equation}
From the exact triangle \ref{rhoisolemm2-triloc} and using \ref{n1n2dev},
it is sufficient to check
the following two last steps.
\medskip

\noindent $3^\circ$ {\it $(\hdag H _1) \R \underline{\Gamma} ^\dag _H w _{H '+} (\E )$
is smoothly devissable in partially overconvergent isocrystals.}

Let $\widetilde{\mathfrak{H}}:= \cup \mathfrak{H} _{i \in I \setminus\{1\}}$,
$\widetilde{w}\,:\,(\mathfrak{A}, \widetilde{\mathfrak{H}}) \rightarrow \mathfrak{A}$
be the canonical map.
Similarly to the begin of the proof of \ref{rhoisolemm} (i.e., using a Mayer-Vietoris exact triangle),
we get the second isomorphism:
$(\hdag H _1) \R \underline{\Gamma} ^\dag _H w _{H '+} (\E )
\riso
\R \underline{\Gamma} ^\dag _H \circ (\hdag H _1)  \circ w _{H '+} (\E )
\riso
\R \underline{\Gamma} ^\dag _{\widetilde{H}}\circ  (\hdag H _1) \circ  w _{H '+} (\E )$.
We get from \ref{ZZ'u-equ2} the isomorphism:
$(\hdag H _1) (w _{H'+} (\E )) \liso \widetilde{w} _{H'\cup H_1,+} (\E (\hdag H _1)) $.
Thus,
$(\hdag H _1) \R \underline{\Gamma} ^\dag _H w _{H '+} (\E )
\riso
\R \underline{\Gamma} ^\dag _{\widetilde{H}}
\widetilde{w} _{H'\cup H_1,+} (\E (\hdag H _1)).$
By the induction hypothesis,
$\R \underline{\Gamma} ^\dag _{\widetilde{H}}
\widetilde{w} _{H'\cup H_1,+} (\E (\hdag H _1))$ is
smoothly devissable in partially overconvergent isocrystals.
\medskip

\noindent $4^\circ$ {\it $\R \underline{\Gamma} ^\dag _{H _1} w _{H '+} (\E)$
is smoothly devissable in partially overconvergent isocrystals.}

Let $\mathfrak{H} ^\# _1 = (\mathfrak{H} _1, \mathfrak{H} _1\cap \widetilde{\mathfrak{H}})$,
$i _1\,:\, \mathfrak{H} _1 \hookrightarrow \mathfrak{A}$,
$g _1\,:\, \ \mathfrak{A} \rightarrow \mathfrak{H} _1$,
$g _1 ^\# \,:\, \ \mathfrak{A} ^\#  \rightarrow \mathfrak{H} ^\# _1 $,
$w _1\,:\, \mathfrak{H} ^\# _1 \rightarrow \mathfrak{H} _1$
be the canonical morphisms.

By \ref{overconvbis} (and with the remark \ref{overconvbis-cone2}),
$g _{1+} ^\#  \circ \R \underline{\Gamma} ^\dag _{H _1}( \E)$
is a complex of
coherent $\D_{\mathfrak{H} _1 ^\#}^\dag(\hdag H_1 \cap H') _{ \Q}$-modules,
locally projective of finite type as $\O_{\mathfrak{H}_1}(\hdag H_1 \cap  H')_{ \Q}$-modules
and satisfying conditions (a) and (b).
Then, by induction hypothesis,
$w _{1+} \circ g _{1+} ^\#  \circ \R \underline{\Gamma} ^\dag _{H _1}( \E)$
is
smoothly devissable in partially overconvergent isocrystals.
Moreover,
\begin{gather}
\label{rhoisolemm2-iso1311}
w _{1,+} \circ g _{1+} ^\#  \circ \R \underline{\Gamma} ^\dag _{H _1}( \E)
\riso
g _{1+} \circ w _{H '+}  \circ \R \underline{\Gamma} ^\dag _{H _1}( \E)
\underset{\ref{u+commhdag-coro}}{\riso}
g _{1+} \circ \R \underline{\Gamma} ^\dag _{H _1} \circ w _{H '+}  ( \E)
\riso
i _1 ^! w _{H '+} (\E).
\end{gather}
Thus,
$i _1 ^! w _{H '+} (\E)$
is smoothly devissable in partially overconvergent isocrystals
and so is
$\R \underline{\Gamma} ^\dag _{H _1} w _{H '+} (\E)$.
\end{proof}

\begin{defi}
\label{defi-strgcoh}
  Let $\X$ be a smooth formal $\V$-scheme, $D$ a divisor of $X$ and
  $\E \in D (\D ^\dag _{\X } (\hdag D) _{\Q})$.
  To avoid the confusion with the coherence over $\D ^\dag _{\X } (\hdag D) _{\Q}$,
  we will say that $\E$ is
  ``{\it $-1$-overholonomic}'' if
  $\E \in D ^\mathrm{b} _\mathrm{coh} (\D ^\dag _{\X,\Q})$.
\end{defi}

\begin{lemm}
\label{rhoisolemm2-complm}
    Let $\mathfrak{A} = \Spf \V \{ t _1,\dots, t_n\}$,
  and, for $i =1,\dots, n$, let
  $\mathfrak{H} _i$ be the formal closed subscheme of $\mathfrak{A}$ defined by $t _i =0$.
  Let $I$ be a subset of $\{ 1,\dots , n\}$.
  We pose $\mathfrak{H} := \cup _{i \in I} \mathfrak{H} _i$.
  Let $\mathfrak{A} ^\# := (\mathfrak{A},\mathfrak{H})$,
$w$ : $\mathfrak{A} ^\# \rightarrow \mathfrak{A}$ be the canonical morphism.
Let $\E$ be coherent
$\D ^\dag _{\mathfrak{A}^\#,\Q}$-module, locally projective of finite type
as $\O_{\mathfrak{A},\Q}$-module and satisfying the conditions (a) and (b) of
\ref{Nobuorigid}.
Then the partially overconvergent isocrystals which appear in the smooth devissage of
$w _{+} (\E  )$ given by the divisors $T _1,\dots ,T _N$ of \ref{rhoisolemm2}
are $-1$-overholonomic.
\end{lemm}

\begin{proof}
  First, we prove by induction in $n$ that, for any subset $J \subset I$,
  $\R \underline{\Gamma} ^\dag _{H _J} w _{+} (\E  )\in D ^\mathrm{b} _\mathrm{coh} (\D ^\dag _{\X,\Q})$,
  where $\mathfrak{H} _J = \cap _{j \in J} \mathfrak{H} _j$.

Let $J$ a subset of $I$.
The case where $J$ is empty is obvious.
  So, we come down to treat the case $|J|\geq 1$. Up to a re-indexation, we can suppose $1 \in J$.
From \ref{rhoisolemm2-iso1311} and with its notation, we get
$w _{1,+} \circ g _{1+} ^\#  \circ \R \underline{\Gamma} ^\dag _{H _1}( \E)
\riso
i _1 ^! w _{+} (\E)$,
where
$g _{1+} ^\#  \circ \R \underline{\Gamma} ^\dag _{H _1}( \E)$
is a complex of
coherent $\D_{\mathfrak{H} _1 ^\#,\Q}^\dag$-modules,
locally projective of finite type as $\O_{\mathfrak{H}_1, \Q}$-modules
and satisfying the conditions (a) and (b).
Then, by the induction hypothesis,
$\R \underline{\Gamma} ^\dag _{H _J} i _1 ^! w _{+} (\E  )
\in
D ^\mathrm{b} _\mathrm{coh} (\D ^\dag _{\mathfrak{H}_1,\Q})$.
Since $\R \underline{\Gamma} ^\dag _{H _J} w _{+} (\E  )
\riso
i _{1+ } i _1 ^!  \R \underline{\Gamma} ^\dag _{H _J} w _{+} (\E  )
\riso
i _{1+ } \R \underline{\Gamma} ^\dag _{H _J} i _1 ^! w _{+} (\E  )$, it follows that
$\R \underline{\Gamma} ^\dag _{H _J} w _{+} (\E  ) \in
D ^\mathrm{b} _\mathrm{coh} (\D ^\dag _{\X,\Q})$.

Secondly, let $J$ and $J'$ be two subsets of $I$.
Then, using a Mayer-Vietoris exact sequence,
since $H _J \cap H _{J'} = H _{J\cup J'}$,
we check that
$\R \underline{\Gamma} ^\dag _{H _J \cup H _{J'}} w _{+} (\E  ) \in
D ^\mathrm{b} _\mathrm{coh} (\D ^\dag _{\X,\Q})$.
Similarly,
we obtain by induction on $r\geq 1$ that,
for any subsets
$J_1 ,\dots, J_r$ of $I$,
the complex
$\R \underline{\Gamma} ^\dag _{\cup _{s=1,\dots, r} H _{J_s}} w _{+} (\E  ) $
belongs to
$D ^\mathrm{b} _\mathrm{coh} (\D ^\dag _{\X,\Q})$.
If $D _1$ and $D _2$ are some divisors which are a finite union of some divisors of the form $H _J$ with $J$ as subset of $I$,
by the exact triangle of localization
of $\R \underline{\Gamma} ^\dag _{D_1} w _{+} (\E  ) $ with respect to
$D_2$, we get
$(\hdag D _2 )\circ \R \underline{\Gamma} ^\dag _{D_1} w _{+} (\E  ) \in
D ^\mathrm{b} _\mathrm{coh} (\D ^\dag _{\X,\Q})$.

\end{proof}

\begin{lemm}
  \label{h+=0is0overhol}
Let $r\geq -1$ be an integer,
$h$ : $\X \rightarrow \X'$ be a finite and \'etale morphism of smooth formal $\V$-schemes,
  $D'$ be a divisor of $X'$, $D := h ^{-1}(D')$,
$\E \in D ^\mathrm{b} _\mathrm{coh} (\D ^\dag _{\X } (\hdag D) _{\Q})$.
If $ h _+ (\E)$ is smoothly devissable in $r$-overholonomic (see \cite[3.1]{caro_surholonome}) partially overconvergent isocrystals
then
$\E$ is smoothly devissable in $r$-overholonomic partially overconvergent isocrystals.
\end{lemm}

\begin{proof}
  Let $Z'$ be a smooth closed subscheme of
  $X'$, $T'$ a divisor which contains $D'$ such that
  $T' \cap X'$ is a divisor of $Z'$ and
  the cohomological spaces of
  $\R \underline{\Gamma} ^\dag _{Z'}(\hdag T') (h _+ (\E))$
  are $r$-overholonomic and in the essential image of the functor $\sp _{Z' \hookrightarrow \X', T', +}$.
Pose $T := h ^{-1}(T')$ and $Z := h ^{-1}(Z')$.
Then, $h _+  (\R \underline{\Gamma} ^\dag _{Z}(\hdag T)(\E))
\riso
\R \underline{\Gamma} ^\dag _{Z'}(\hdag T') (h _+ (\E))$.
With this remark, we check that it is sufficient by smooth devissage of $h _+ (\E)$ to prove that if
$\E \riso \R \underline{\Gamma} ^\dag _{Z}(\hdag T)(\E)$,
$\E \in D ^\mathrm{b} _\mathrm{coh} (\D ^\dag _{\X } (\hdag T) _{\Q})$
and
the cohomological spaces of $h _+ (\E)$ are
$r$-overholonomic and in the essential image of the functor $\sp _{Z' \hookrightarrow \X', T', +}$
then the cohomological spaces of $\E$ are
$r$-overholonomic and in the essential image of the functor $\sp _{Z \hookrightarrow \X, T, +}$.
Since $h _+$ is exact, we can also suppose that $\E$ is a coherent $\D ^\dag _{\X } (\hdag T) _{\Q}$-module.
Since this is local in $\X$ and $h$ is affine, we can suppose $\X$ and $\X'$ affine.
Then, there exists respectively some liftings $a\,:\, \ZZ \rightarrow \ZZ'$,
$\iota\,:\, \ZZ \hookrightarrow \X$,
$\iota '\,:\, \ZZ '\hookrightarrow \X'$ of $Z \rightarrow Z'$, $Z \hookrightarrow X$, $Z' \hookrightarrow X'$.
Since $h _+$ commutes with the local cohomological functor with compact support,
$\iota '_+ \iota ^{\prime !}  (h _+ (\E)) \riso h _+ \iota _+ \iota ^!  (\E)$.
Because the direct images of arithmetic $\D$-modules do not depend (up to a canonical isomorphism) on
the choice of the lifting,
$h _+ \iota _+ \iota ^!  (\E) \riso
\iota '_+ a _+  \iota ^!  (\E) $.
Hence $\iota '_+ \iota ^{\prime !}  (h _+ (\E)) \riso \iota '_+ a _+  \iota ^!  (\E) $.
Since $\iota ^{\prime !}  \iota '_+ \riso Id$,
$\iota ^{\prime !}  (h _+ (\E)) \riso a _+  \iota ^!  (\E) $.
This means that $\iota ^!  (\E) $ is a coherent $\D ^\dag _{\ZZ } (\hdag T\cap Z) _{\Q}$-module
(for the coherence, recall that $\E$ has its support in $Z$)
such that $a _+  \iota ^!  (\E) $ is $r$-overholonomic and $\O _{\ZZ '} (\hdag T'\cap Z') _{\Q}$-coherent.
Let $\Y:= \ZZ \setminus T$, $\Y':= \ZZ '\setminus T'$.
Since the morphism $\Y \rightarrow \Y'$ induced by $a$ is finite (and \'etale),
the fact that $a _+  \iota ^!  (\E) $ is $\O _{\ZZ '} (\hdag T'\cap Z') _{\Q}$-coherent
implies that
$\Gamma (\Y, \iota ^! (\E))$ is of finite type over $\Gamma (\Y, \O _{\Y,\Q})$.
Then, by \cite[2.2.12--13]{caro_courbe-nouveau},
$\iota ^! (\E)$ is associated to an isocrystal on $Y$ overconvergent along $T\cap Z$.
Since $a$ is finite and \'etale, $a _+ =a_*$ and thus
$\iota ^! (\E)$ is a direct factor of $ a ^* a _+ \iota ^! (\E)$.
Then, since $a _+  \iota ^!  (\E) $ is $r$-overholonomic
and that the $r$-overholonomicity is stable under extraordinary inverse image
(e.g., under $a^!=a^*$), we get the $r$-overholonomicity of $\iota ^! (\E)$.
Since $\E \riso \iota _+ \iota ^! (\E)$, $\E$ is $r$-overholonomic and is
in the essential image of
$\sp _{Z \hookrightarrow \X, T,+}$, which finishes the proof.
\end{proof}

\begin{nota}
\label{nota-f*}
  Let $\X$, $\X'$ be two smooth formal $\V$-schemes, $f _0$ : $X' \rightarrow X$ a
  morphism of $k$-schemes, $Z$ (resp. $Z'$) a divisor of $X$ (resp. $X'$) such that
  $f _0 ^{-1} ( Z) \subset Z'$.

>From \cite[2.1.6]{Be2}, we have a functor:
$f _{0} ^! $ :
$\smash{\underset{^{\longrightarrow}}{LD}} ^{\mathrm{b}} _{\Q ,\mathrm{qc}}
  ( \smash{\widehat{\D}} _{\X } ^{(\bullet)})
\rightarrow
\smash{\underset{^{\longrightarrow}}{LD}} ^{\mathrm{b}} _{\Q ,\mathrm{qc}}
  ( \smash{\widehat{\D}} _{\X ^{\prime}} ^{(\bullet)})$.
We obtain:
$f _{0, Z', Z} ^! := (\hdag Z') \circ f _{0} ^!  \circ \mathrm{Forg} _Z$:
$\smash{\underset{^{\longrightarrow}}{LD}} ^{\mathrm{b}} _{\Q ,\mathrm{qc}}
  ( \smash{\widehat{\D}} _{\X } ^{(\bullet)}(Z))
\rightarrow
\smash{\underset{^{\longrightarrow}}{LD}} ^{\mathrm{b}} _{\Q ,\mathrm{qc}}
  ( \smash{\widehat{\D}} _{\X ^{\prime}} ^{(\bullet)}(Z'))$.
  When there exists a lifting $f$ : $\X' \rightarrow \X$ of $f _0$, we retrieve
$f _{Z', Z} ^!$.
We pose $f _{0,Z',Z} ^* = \H ^0 \circ f _{0, Z', Z} ^! [-d _{X'/X}]$
and $f _{Z',Z} ^* = \H ^0 \circ  f _{Z', Z} ^! [-d _{X'/X}]$,
where
$d _{X'/X}$ is the relative dimension of $X'$ over $X$.
We keep the previous notation when we work with coherent complexes.
Remark that if $f _0 ^{-1} ( Z) =Z'$ then
$f _{Z',Z} ^*=f^*$, where $f ^*$ is the usual inverse image functor (as $\O _{\X}$-modules).
\end{nota}

\begin{lemm}
\label{commf_0*}
  Let $\X $, $\X'$ be two smooth formal $\V$-schemes,
$\ZZ$ (resp. $\ZZ'$) be a strict normal crossing divisor of $\X$ (resp. $\X'$). Let $f _0$ : $X'\rightarrow X$
be a morphism of $k$-schemes such that $f_0 ^{-1} (Z) \subset Z'$.
We note $f ^\# _0$ : $(X',Z') \rightarrow (X,Z)$ the induced
morphism. Let $\E$ (resp. $\FF$) be a coherent $F$-$\D_{(\X,\ZZ),\Q}^\dag$-module (resp. $\D_{(\X,\ZZ),\Q}^\dag$-module), locally
projective of finite type over $\O_{\X,\Q}$ (see
\ref{Frobenius-structure}).
\begin{enumerate}
  \item We have the isomorphism of coherent
$F$-$\D ^\dag _{\X'} (\hdag Z ') _{\Q}$-modules, $\O _{\X'} (\hdag Z') _{\Q}$-coherent:
\begin{equation}
  \label{commf_0*-iso1}
(\hdag Z')  (f _0 ^{\#*} (\E))  \riso  f _{0,Z',Z} ^{*} (\E  (\hdag Z)) ,
\end{equation}
where the first (resp. second) inverse image is defined in
\ref{def-inv image} (resp. \ref{nota-f*}).

\item Suppose that there exists a lifting $f $ : $\X' \rightarrow \X$ of $f _0$ which induces a lifting
$f ^\#$ : $(\X' ,\ZZ')\rightarrow (\X ,\ZZ)$ of $f _0 ^\#$.
Then, we have the isomorphism of coherent $\D ^\dag _{\X'} (\hdag Z ') _{\Q}$-modules, $\O _{\X'} (\hdag Z') _{\Q}$-coherent:
\begin{equation}
  \label{commf_0*-iso2}
  (\hdag Z')  (f  ^{\#*} (\FF))  \riso  f _{Z',Z} ^{*} (\FF  (\hdag Z)).
\end{equation}

\end{enumerate}

\end{lemm}

\begin{proof}
  The sheaf $f _0 ^{\#*} (\E)$ is a coherent $F$-$\D_{(\X',\ZZ'),\Q}^\dag$-module, locally
projective of finite type over $\O_{\X',\Q}$. By both Kedlaya's fully faithfulness theorems \cite[6.4.5]{kedlaya-semistableI}
and \cite[4.2.1]{kedlaya-semistableII}, it is sufficient to check
the isomorphism \ref{commf_0*-iso1} outside $Z'$, which is obvious.
Using \ref{spcommDD'iso1}, the isomorphism \ref{commf_0*-iso2} becomes straightforward.
\end{proof}

\begin{rema}
  In the proof of \ref{commf_0*-iso1} we use the Frobenius structure
  (more precisely, the second Kedlaya's fully faithfulness theorem, i.e., \cite[4.2.1]{kedlaya-semistableII},
  needs a Frobenius structure). But, the isomorphism \ref{commf_0*-iso1} should be true without a Frobenius structure on $\E$.
  This check is technical (we have to paste local isomorphisms)
  and we avoid it because this is not really useful in this paper.
\end{rema}

\begin{empt}
[log-relative duality isomorphism]
\label{relativeduality}
We recall in this paragraph the isomorphism \cite[5.25.2]{caro_log-iso-hol}
and give a version of this.
This isomorphism will be essential in the next theorem.
Let $\X$ be a smooth formal $\V$-scheme,
$\ZZ$ a strict normal crossing divisors of $\X$,
$\X ^\#:=(\X, \ZZ) $ the induced smooth logarithmic formal $\V$-scheme,
$u$ : $\X ^\#\rightarrow \X$ the canonical morphism.
Let $\E$ be a coherent
$\D ^\dag _{\X^\#,\Q}$-module which is a locally projective
$\O_{\X,\Q}$-module of finite type.
It follows from \cite[5.25.2]{caro_log-iso-hol} that
$\DD_{\X } \circ u_+ (\E) \riso u_+ \circ \DD _{\X ^\#} (\E (\ZZ))$ (see the notation \ref{defidual}).
By \cite[5.22]{caro_log-iso-hol}, $\DD _{\X ^\#} (\E (\ZZ)) \riso (\E (\ZZ)) ^\vee \riso
\E ^\vee (-\ZZ)$.
Then:
\begin{equation}
  \label{relativeduality-iso}
\DD_{\X } \circ u_+ (\E) \riso u_+ (\E ^\vee (-\ZZ)).
\end{equation}
\end{empt}

\begin{theo}
\label{theorhoiso2}
Let $\X$ be a smooth formal $\V$-scheme,
$\ZZ$ a strict normal crossing divisors of $\X$,
$\X ^\#:=(\X, \ZZ) $ the induced smooth formal $\V$-scheme,
$u$ : $\X ^\#\rightarrow \X$ the canonical morphism.
Let $\E$ be a coherent
$\D ^\dag _{\X^\#,\Q}$-module which is a locally projective
$\O_{\X,\Q}$-module of finite type satisfying the following condition:
\begin{list}{}{}
\item[\mbox{\rm (c)}] none of elements of $\mathrm{Exp}(\E)^\gr$
(see the definition in \ref{remexp})
is a $p$-adic Liouville number.
\end{list}
Then $u _{+} (\E)$ is overholonomic.
\end{theo}

\begin{proof}
Let $r \geq -1$, $n\geq 0$ be two integers and
let us consider the next properties:
  \begin{enumerate}
\item [$(P _{n,r})$] If $\dim X \leq n$ then $u _{+} (\E)$ is $r$-overholonomic ;

  \item [$(Q _{n,r})$]
  If $\dim X \leq n$ then $\R \underline{\Gamma} ^\dag _Z  u _{+} (\E)$ is $r$-overholonomic ;

    \item [$(R _{n,r})$]
  If $\dim X \leq n$ then $\E (\hdag Z) $ is $r$-overholonomic.
\end{enumerate}
\medskip

\noindent $\mathrm{(I)}$ {\it First, for any $n\geq 1$, $r \geq -1$, we check that
$(P _{n-1,r}) \Rightarrow (Q _{n,r})$.}

$1^\circ$ {\it How to use the case \ref{rhoisolemm2-complm} of affine spaces.}

Let $\HH _1,\dots, \HH _n$ be the coordinate hyperplanes of
$\Spf \V \{ t _1,\dots, t_n\}$,
$\HH := \HH _1 \cup \cdots \cup \HH _r$ for some $r \leq n$,
$\smash{\widehat{\A}} ^n _\V:=\Spf \V \{ t _1,\dots, t_n\}$ and
$\smash{\widehat{\A}} ^{n\#} _\V  := (\Spf \V \{ t _1,\dots, t_n\}, \HH)$.
Since $r$-overholonomicity in local in $\X$,
similarly to the first step of the proof of theorem \ref{theorhoiso},
we come down to the case where
there exists a commutative diagram of the form:
\begin{equation}
  \notag
  \xymatrix{
  {\X} \ar@{=}[r]  & {\X} \ar[r] ^-{h} & {\smash{\widehat{\A}} ^n _\V} \\
{(\X ,\ZZ)} \ar[u] ^-{u}  & {(\X ,\ZZ'')} \ar[l] _-{\tilde{u}} \ar[u] ^-{v} \ar[r] ^-{h ^\#} &
{\smash{\widehat{\A}} ^{n\#} _\V,} \ar[u] ^-{w}
}
\end{equation}
where
$h$ is a finite \'etale morphism,
$\ZZ '':= h ^{-1} (\HH)$ and where
$h ^\#$, 
$w$, 
$v$, 
$\tilde{u}$
are the canonical induced morphisms.
Moreover,
denote by $\X ^{\# \prime}:= (\X, \ZZ'')$ and $\ZZ'$
the union of the irreducible components of $\ZZ''$ which are not an irreducible component of $\ZZ$.

 $2^\circ$ {\it $\R \underline{\Gamma} ^\dag _{H } w _{+} h ^\#  _{+}  (\E)$
is $r$-overholonomic.}

The case where $r=-1$ is already known from \ref{rhoisolemm2-complm}.
Suppose now $r \geq 0$.
We notice (for example see \ref{I.6.11}) that
$\E$ is also a coherent
$\D ^\dag _{\X^{\#\prime},\Q}$-module which is a locally projective
$\O_{\X,\Q}$-module of finite type.
Since $h $ is finite and \'etale,
$h ^\#  _{+}  (\E )$ is a coherent
$\D^\dag _{\smash{\widehat{\A}} ^{n\#} _\V,\Q}$-module which is a locally projective
$\O_{\smash{\widehat{\A}} ^{n} _\V ,\Q}$-module of finite type and such that
the condition (c) holds (see \ref{remexp}.\ref{remexp-6}).
Hence, by \ref{rhoisolemm2},
$\R \underline{\Gamma} ^\dag _{H} w _{+} h ^\#  _{+}  (\E)$
is smoothly devissable in partially overconvergent isocrystals.
Also, in the proof of \ref{rhoisolemm2} (see \ref{rhoisolemm2-iso1311})
and with its notation,
we have checked that
$i_1 ^! w _{+} h ^\#  _{+}  (\E)$
is isomorphic to the image by $w _{1+}$
of a complex of
coherent $\D^\dag _{\mathfrak{H} _1 ^\#, \Q}$-module which are locally projective
$\O_{\mathfrak{H}_1, \Q}$-modules of finite type satisfying the condition (c)
by \ref{overconv}.
The hypothesis $( P _{n-1,r})$ implies that
$i_1 ^! w _{+} h ^\#  _{+}  (\E)$ is $r$-overholonomic.
Then
$i_{1+} i_1 ^! w _{+} h ^\#  _{+}  (\E)\riso
\R \underline{\Gamma} ^\dag _{H _1} w _{+} h ^\#  _{+}  (\E)$
is $r$-overholonomic.
Symmetrically, we obtain for any $i=1,\dots, r$ that
$\R \underline{\Gamma} ^\dag _{H _i} w _{+} h ^\#  _{+}  (\E)$
is $r$-overholonomic.
Using Mayer-Vietoris exact triangles and the stability of $r$-overholonomicity
by local cohomological functors, this implies that
$\R \underline{\Gamma} ^\dag _{H } w _{+} h ^\#  _{+}  (\E)$
is $r$-overholonomic.

 $3^\circ$ {\it $(\hdag Z') \R \underline{\Gamma} ^\dag _Z u _{+} (\E)$
is $r$-overholonomic.}

We have:
$h _{+} (\R \underline{\Gamma} ^\dag _{Z''} v _{+} (\E ))
\riso
\R \underline{\Gamma} ^\dag _{H} h _{+} v _{+} (\E)
\riso
\R \underline{\Gamma} ^\dag _{H} w _{+} h ^\#  _{+}  (\E)$
(see \cite[2.2.18.2]{caro_surcoherent} for the first isomorphism and
\ref{trans+} for the second one).
Then, by \ref{h+=0is0overhol}:
$\R \underline{\Gamma} ^\dag _{Z''} v _{+} (\E )$
is $r$-overholonomic.
We have checked in the proof of \ref{theorhoiso} that
$u _+ (\E (\hdag Z')) \riso v _+ (\E)$.
This implies $\R \underline{\Gamma} ^\dag _{Z''} (\hdag Z')  u _{+} (\E)$
is $r$-overholonomic.
Using a Mayer-Vietoris exact triangle (similarly to \ref{MayerVietoris-rhoiso1}), we obtain
$\R \underline{\Gamma} ^\dag _{Z''} (\hdag Z')  u _{+} (\E) \riso
\R \underline{\Gamma} ^\dag _{Z} (\hdag Z')  u _{+} (\E)$.

Using the exact triangle of localization of $\R \underline{\Gamma} ^\dag _{Z} u _{+} (\E) $
with respect to $Z'$,
we come down to prove
$\R \underline{\Gamma} ^\dag _{Z ' \cap Z} u _{+} (\E)$
is $r$-overholonomic, which is the last step of the proof of $\mathrm{(I)}$.

 $4^\circ$ {\it $\R \underline{\Gamma} ^\dag _{Z ' \cap Z} u _{+} (\E)$
is $r$-overholonomic.}

When $Z \cap Z'$ is empty, this is obvious.
It remains to deal with the case where $Z\cap Z'$ is not empty.
Let $x$ be a closed point of $Z\cap Z'$,
$\ZZ _1,\dots , \ZZ _r$ be the irreducible components of $\ZZ$ containing $x$,
$\ZZ _{r+1}, \dots , \ZZ _s$ be the irreducible components of $\ZZ'$ containing $x$.
Since $\R \underline{\Gamma} ^\dag _{Z ' \cap Z} u _{+} (\E)$ is zero outside $Z \cap Z'$,
it is sufficient to prove its nullity around $x$.
Then, we can suppose that
$\ZZ = \ZZ _1 \cup \cdots \cup \ZZ _r$ and $\ZZ ' = \ZZ _{r+1} \cup \cdots \cup \ZZ _s$.

Let $I$ a subset of $\{r+1,\dots, s\}$,
$\X':=\cap _{i\in I} \ZZ _i$,
$\X ^{\prime \#} := (\X', \X' \cap \ZZ)$,
$\iota$ : $\X' \hookrightarrow \X$,
$\iota ^\#$ : $\X ^{\prime \#} \hookrightarrow \X ^\#$,
$u'$ : $\X ^{\prime \#} \rightarrow \X '$
be the canonical morphisms.
Then,
$\R \underline{\Gamma} ^\dag _{X' \cap Z} u _{+} (\E)
\riso \R \underline{\Gamma} ^\dag _{Z} \iota_+ \iota^! u _{+} (\E)
\underset{\ref{lemmtheorhoiso}}{\riso} \R \underline{\Gamma} ^\dag
_{Z} \iota_+ u ' _{+} \iota^{\#!}  (\E) \riso \iota_+
\R \underline{\Gamma} ^\dag _{Z\cap X' } u'  _{+} \iota^{\#!}  (\E)$.
From $(P _{n-1,r}) $, we get that
$\R \underline{\Gamma} ^\dag _{Z\cap X' } u'  _{+} \iota^{\#!}  (\E)$
is $r$-overholonomic.
Hence, using the stability of the $r$-overholonomicity
under the direct image by a proper morphism,
$\R \underline{\Gamma} ^\dag _{X' \cap Z} u _{+} (\E)$
is also $r$-overholonomic.
Using Mayer-Vietoris exact triangles, we get
that if $\X''$ is the union of some intersections of some irreducible
components of $\ZZ'$ then $\R \underline{\Gamma} ^\dag _{X'' \cap Z} u _{+} (\E)$
is $r$-overholonomic. In particular,
$\R \underline{\Gamma} ^\dag _{Z' \cap Z} u _{+} (\E)$
is $r$-overholonomic.
\medskip

\noindent $\mathrm{(II)}.$
{\it We prove $(P _{n,r-1}) + (Q _{n,r}) \Rightarrow (R _{n,r})$ for any $n\geq 0$, $r \geq 0$.}

We suppose $r =0$ (resp. $r \geq 1$)
  By \ref{commf_0*-iso2}, it is sufficient to prove that for any divisor
  $D$ of $X$, $\E (\hdag Z \cup D)$ is $\D ^\dag _{\X,\Q}$-coherent
  (resp. $\DD _{\X } (\E (\hdag Z \cup D)) $ is $r-1$-overholonomic).
Using de Jong's desingularization theorem (\cite{dejong}),
there exist
a proper smooth morphism $f$ : $\PP' \rightarrow \X$ of smooth formal $\V$-schemes,
a smooth scheme $X'$ over $k$, a closed immersion
$\iota _0 '$ : $X ' \hookrightarrow \PP'$,
a projective, surjective, generically finite and \'etale morphism $a _0 $ : $X' \rightarrow X$
such that
$a _0=f _0\circ \iota' _0 $ and
  $Z'' := a _0 ^{-1} (Z\cup D)$ is a strict normal crossing divisor of $X'$.
Since $\E (\hdag Z \cup D)$ is associated to an isocrystal on $X \setminus (Z \cup D)$ overconvergent along
$Z \cup D$ (i.e., is a
coherent
$\D ^\dag _{\X} (\hdag Z \cup D) _{\Q}$-modules, $\O _{\X} (\hdag Z \cup D) _{\Q}$-coherent),
by \cite[6.1.4]{caro_devissge_surcoh}
and \cite[6.3.1]{caro_devissge_surcoh}
$\E (\hdag Z \cup D)$ is a direct factor
of
$f _+ \R \underline{\Gamma} ^\dag _{X'} f ^! (\E (\hdag Z \cup D)) $.
Since $r-1$-overholonomocity
is stable under direct image by a proper morphism
(resp. and furthermore since $f _+$ commutes with $\DD_{\X }$),
it remains to prove that $\R \underline{\Gamma} ^\dag _{X'} f ^! (\E (\hdag Z \cup D)) $
is $\D ^\dag _{\PP',\Q}$-coherent
(resp. $\DD _{\X } \circ \R \underline{\Gamma} ^\dag _{X'} \circ f ^! (\E (\hdag Z \cup D)) $ is $r-1$-overholonomic).
This is local in $\PP'$.
Then, we can suppose that there exists a lifting
$\iota '$ : $\X ' \hookrightarrow \PP'$ of $\iota ' _0$ and that
$Z ''$ lifts to a
relatively strict normal crossing divisor $\ZZ''$ of $\X'$ over $\V$.
We pose $a = f \circ \iota '$ and denote by
$u'\,:\, (\X', \ZZ '') \rightarrow \X'$ and
$a ^\#$ : $ (\X ' , \ZZ'') \rightarrow (\X,\ZZ)$
the canonical morphisms.

By \cite[4.4.5]{Beintro2},
$\R \underline{\Gamma} ^\dag _{X'} f ^! (\E (\hdag Z \cup D))
\riso
\iota ' _+ \iota ^{\prime !} f ^! (\E (\hdag Z \cup D))
\riso
\iota ' _+ a ^! (\E (\hdag Z \cup D))$.
Then, we come down to prove that
$a ^! (\E (\hdag Z \cup D)) =a ^* (\E (\hdag Z \cup D)) $ (by flatness)
is $\D ^\dag _{\X',\Q}$-coherent (resp. $\DD _{\X } a ^* (\E (\hdag Z \cup D))$ is $r-1$-overholonomic).
We have $a ^* (\E (\hdag Z \cup D))\riso (\hdag Z'') \circ a ^* (\E (\hdag Z)) \riso a _{ Z'',Z} ^{*} (\E  (\hdag Z))$.
We get from \ref{commf_0*-iso2} the following isomorphism:
$a _{Z'',Z} ^{*} (\E  (\hdag Z))
\riso  (\hdag Z'')(a ^{\#*} (\E))   $.
Thus, it remains to prove that
$(\hdag Z'')(a ^{\#*} (\E))   $
is $\D ^\dag _{\X',\Q}$-coherent (resp. $\DD _{\X } \circ (\hdag Z'')(a  ^{\#*} (\E))   $ is $r-1$-overholonomic).
We check this separately:

{\it Non respective case.} By $(Q _{n,0})$, since $a  ^{\#*} (\E)$
satisfies the condition (c) (see \ref{remexp}.\ref{remexp-5}),
the morphism
  $\R \underline{\Gamma} ^\dag _{Z''} u'_+ (a  ^{\#*} (\E)) $ is overcoherent.
 By \ref{exacttriu+F},
 using the exact triangle of localization of $u'_+ (a  ^{\#*} (\E))$ with respect to $Z''$,
 this implies that
$(\hdag Z'') (a  ^{\#*} (\E))  $ is $\D ^\dag _{\X',\Q}$-coherent.

{\it Respective case.}
By applying the functor $\DD_{\X }$ to the exact triangle of localization of
$u'_+ (a  ^{\#*} (\E))$ with respect to $Z''$ (see \ref{exacttriu+F}), we get
$\DD _{\X } \circ (\hdag Z'')(a  ^{\#*} (\E))   =
\mathrm{Cone} \left(
\DD _{\X } \circ  u'_+ (a  ^{\#*} (\E))
\rightarrow
\DD _{\X } \circ \R \underline{\Gamma} ^\dag _{Z''} \circ  u'_+ (a  ^{\#*} (\E))
\right)[-1]$.
Since $a  ^{\#*} (\E)$ satisfies the condition (c) (see \ref{remexp}.\ref{remexp-5}),
using $(Q _{n,r})$ hypothesis,
we get that
$\DD _{\X } \circ \R \underline{\Gamma} ^\dag _{Z''} \circ u'_+ (a  ^{\#*} (\E))$
is $r-1$-overholonomic.
Also, the log-relative duality isomorphism of \ref{relativeduality-iso} gives:
$\DD _{\X }\circ u'_+ (a  ^{\#*} (\E)) \riso
u'_+ ((a  ^{\#*} (\E)) ^{\vee} (-\ZZ''))$.
Since $(a  ^{\#*} (\E)) ^{\vee} (-\ZZ'')$ satisfies also the condition (c) (see \ref{remexp}.\ref{remexp-5})
of our theorem,
using $(P _{n,r-1})$ we obtain that
$u'_+ ((a  ^{\#*} (\E)) ^{\vee} (-\ZZ''))$
is $r-1$-overholonomic.
Hence,
$\DD _{\X } \circ (\hdag Z'')(a  ^{\#*} (\E)) $ is
$r-1$-overholonomic.
\medskip

\noindent $\mathrm{(III)}.$
{\it Conclusion.}

For any $n\geq 0$, we know that $(P _{n, -1})$ is true.
Also, for any $r \geq -1$,
$(P _{0,r})$ is already known
(see \cite[7.3]{caro_surholonome}).

We get from the two previous steps that,
for any $r\geq 0$ and $n\geq 1$, $(P _{n, r-1})+ (P _{n-1, r})\Rightarrow (Q _{n, r})+(R _{n, r})$.
Using the exact triangle of localization of
$u _+ (\E)$ with respect to $Z$ we get
$(Q _{n, r})+(R _{n, r})\Rightarrow (P _{n, r})$. Thus,
$(P _{n, r-1})+ (P _{n-1, r})\Rightarrow (P _{n, r})$.
This implies that
$(P _{n, r})$ is true for any $r \geq -1$ and $n \geq 0$.
\end{proof}

\begin{rema}
  We have used in (the step $\mathrm{(II)}$ of) the proof of \ref{theorhoiso2},
  the stability of the condition (c) by
inverse image
and above all by the functor $\E \mapsto \E^\vee(-\ZZ)$.
Since condition (b') of \ref{Nobuo} is not stable by the functor $\E \mapsto \E^\vee(-\ZZ)$,
we do need the strong version of theorem \ref{Nobuorigid} and proposition \ref{overconv}.
\end{rema}

\begin{theo}
  \label{holo-qunip0}
Let $\PP$ be a separated smooth formal scheme over $\V$, $T $ a divisor of $P$,
$X$ a closed smooth subscheme such that $Z:=T \cap X$ is a divisor of $X$, $Y:= X \setminus Z$.
  Let $E$ be an $F$-isocrystal on $Y$ overconvergent along $Z$.
Then $\sp _{X \hookrightarrow \PP,T, +} (E)$ is overholonomic.
\end{theo}

\begin{proof}
Since $E$ admits a semi-stable reduction (see \ref{semi-stable}),
there exists a commutative diagram of the form:
\begin{equation}
  \label{holo-qunip-diag}
  \xymatrix @R=0,3cm {
  { Y '} \ar[r] \ar[d] ^{b_0} & {X'} \ar[r] ^{\iota _0'} \ar[d] ^{a_0} & {\PP'} \ar[d] ^f \\
  {Y} \ar[r] & {X } \ar[r] ^{\iota _0} & {\PP,}}
\end{equation}
such that $f$ is a proper smooth morphism of smooth formal $\V$-schemes,
the left square is cartesian, $X'$ is a smooth scheme over $k$, $\iota _0'$ is a closed immersion,
  $a_0$ is a projective, surjective, generically finite and \'etale morphism,
  $a_0 ^{-1} (Z)$ is a strict normal crossing divisor of $X'$
  and the $F$-isocrystal $a_0^* (E)$ on $Y'$ overconvergent along $a_0 ^{-1} (Z)$ is log-extendable on $X'$.
  We pose $\E:= \sp _{X \hookrightarrow \PP, T, +} (E)$.
We have $\R \underline{\Gamma} ^\dag _{X'} f _T ^! (\E ) \riso
\sp _{X' \hookrightarrow \PP', f ^{-1}(T),+} (a_0^* (E))$.
By \cite[6.1.4]{caro_devissge_surcoh}), $\E \in F\text{-}\mathrm{Isoc}^{\dag \dag} (\PP, T, X/K)$.
Then by \cite[6.3.1]{caro_devissge_surcoh}, we check that $\E$ is a direct factor of
  $f _{T,+} \sp _{X' \hookrightarrow \PP', f ^{-1} (T),+} (a_0^* (E))$.
Since the overholonomicity is stable under direct image by a proper morphism,
it is sufficient to prove that
$\sp _{X' \hookrightarrow \PP', f ^{-1}(T),+} (a_0^* (E))$ is
overholonomic. This last statement is local in $\PP'$. Then,
we can suppose that there exists a lifting $\iota '$ : $\X' \hookrightarrow \PP'$ of $\iota _0 '$ and that
$a_0 ^{-1} (Z)$ lifts to a strict normal crossing divisor $\ZZ'$ of $\X'$ over $\S$.
Then,
$\sp _{X' \hookrightarrow \PP', f ^{-1}(T),+} (a_0^* (E))
\riso
\iota ' _+ \sp _* (a_0^* (E))$, where
$\sp$ : $ \X' _K \rightarrow \X'$ is the specialization morphism of $\X'$.
It remains to check that $\sp _* (a_0^* (E))$ is
overholonomic.
But since $a_0^* (E)$ is an $F$-isocrystal on $Y'$ overconvergent along $a_0 ^{-1} (Z)$ which is log-extendable on $X'$,
it follows from \ref{theorhoiso2} that $\sp _* (a_0^* (E))$ is overholonomic.
\vspace*{3mm}

\end{proof}

The following theorem answers partially positively to the conjecture \cite[3.2.25.1]{caro-2006-surcoh-surcv}:
\begin{theo}
  \label{holo-qunip}
  Let $Y$ be a smooth separated scheme of finite type over $k$.
  Let $E$ be an overconvergent $F$-isocrystal on $Y$.
Then $\sp _{Y+} (E)$ is an overholonomic arithmetic $\D _Y$-module (see \cite[3.2.10]{caro_surcoherent}), where
$\sp _{Y,+}:\, F\text{-}\mathrm{Isoc} ^{\dag}( Y/K )\cong F\text{-}\mathrm{Isoc} ^{\dag \dag}( Y/K )$
is the canonical equivalence from
the category of overconvergent $F$-isocrystals on $Y$ into the category of overcoherent $F$-isocrystals on $Y$
(see \cite[2.3.1]{caro-2006-surcoh-surcv}).
\end{theo}

\begin{proof}
  The theorem is local in $Y$. We can suppose $Y$ affine and then that
there exists an immersion of $Y$ into in proper smooth formal $\V$-scheme $\PP$,
a divisor $T$ of $P$ such that $Y =X \setminus T$ where
$X$ is the closure of $Y$ in $P$. Let $Z := X \cap T$ and $\E:= \sp _{Y+} (E) \in F\text{-}\mathrm{Isoc} ^{\dag \dag} (Y/K) =
F\text{-}\mathrm{Isoc}^{\dag \dag} (\PP, T, X/K)$ (notation of \cite[6.2.1]{caro_devissge_surcoh} and \cite[2.2.4]{caro-2006-surcoh-surcv}).

Using de Jong's desingularization, we come down to the case where $X$ is smooth
(similarly to the proof of \ref{holo-qunip0}), which was already checked in \ref{holo-qunip0}.
\end{proof}

\begin{theo}
\label{theo-overcoh=dev}
Let $\PP$ be a proper smooth formal scheme over $\V$, $T $ a divisor of $P$,
$\E \in F \text{-}D ^\mathrm{b} _\mathrm{coh} (\D ^\dag _{\PP} (\hdag T) _{\Q})$.
Then the following assertion are equivalent:
\begin{enumerate}
  \item The $F$-complex $\E$ is $\D ^\dag _{\PP} (\hdag T) _{\Q}$-overcoherent;
  \item The $F$-complex $\E$ is $\D ^\dag _{\PP,\Q}$-overcoherent;
   \item The $F$-complex $\E$ is overholonomic;
  \item The $F$-complex $\E$ is devissable in overconvergent $F$-isocrystals.
\end{enumerate}
\end{theo}

\begin{proof}
By \cite[3.1.2]{caro-2006-surcoh-surcv},
if $\E$ is $F$-$\D ^\dag _{\PP} (\hdag T) _{\Q}$-overcoherent
then there exists a devissage of
$\E$ in overconvergent $F$-isocrystals.
By \ref{holo-qunip}, if there exists a devissage of
$\E$ in overconvergent $F$-isocrystals then $\E$ is overholonomic.
Finally, it is obvious that if $\E$ is overholonomic then
$\E$ is $\D ^\dag _{\PP,\Q}$-overcoherent
and that if $\E$ is $\D ^\dag _{\PP,\Q}$-overcoherent then $\E$ is $\D ^\dag _{\PP} (\hdag T) _{\Q}$-overcoherent.
\end{proof}

We end this section with the following consequences of \ref{theo-overcoh=dev} explained respectively in \cite[3.2.26.1]{caro-2006-surcoh-surcv}
and \cite[5.8]{caro-stab-prod-tens}:

\begin{coro}
Let $\PP$ be a proper smooth formal scheme over $\V$, $T $ a divisor of $P$,
$Y$ a subscheme of $P$.

\begin{enumerate}
\item We have an equivalence between the category of quasi-coherent $F$-complexes devissable in overconvergent $F$-isocrystals and
the category of coherent $F$-complexes devissable in overconvergent $F$-isocrystals, i.e.,
$$F \text{-} \smash[b]{\underset{^{\longrightarrow }}{LD }}  ^\mathrm{b} _{\Q, \mathrm{dev}}
(\overset{^\mathrm{g}}{} \smash{\widehat{\D}} _{\PP} ^{(\bullet)} (T ))
\cong
F\text{-}D ^\mathrm{b} _{\mathrm{dev}}
(\smash{\D} ^\dag _{\PP} (\hdag T) _\Q).$$

  \item Denoting by $F\text{-}D ^\mathrm{b} _\mathrm{ovhol} (\smash{\D} _{Y }) $,
  the category of overholonomic $F$-complexes of arithmetic $\D _Y$-modules,
  we get a canonical tensor product:
\begin{equation}\label{impl-de-conj-otimes}
  -\smash{\overset{\L}{\otimes}} ^\dag _{\O _{Y} } -
  \ : \
F\text{-}D ^\mathrm{b} _\mathrm{ovhol} (\smash{\D} _{Y }) \times
F\text{-}D ^\mathrm{b} _\mathrm{ovhol} (\smash{\D} _{Y })
\rightarrow
F\text{-}D ^\mathrm{b} _\mathrm{ovhol} (\smash{\D} _{Y }).
\end{equation}

\end{enumerate}

\end{coro}

\subsection{Some precisions for the case of curves}
In this section, $i$ : $\ZZ \hookrightarrow \X$ is a closed immersion of
separated smooth formal $\V$-schemes such that $\dim X=1$ and $Z$ is a divisor of $X$.
Let $\Y:= \X \setminus \ZZ$, $\X ^\# := (\X,\ZZ)$,
$u $ : $ \X ^\# \rightarrow \X$, $f$ : $\X \rightarrow \S$ be the canonical morphisms
and $ f ^\#:=f \circ u\; : \; \X ^\# \rightarrow  \S$.

The next theorem is slightly better for curves than \ref{theorhoiso}
because we have another divisor $D$.
\begin{prop}
\label{conjrhocurve}
  Let $D$ be a divisor of $X$,
$\E$ be a coherent
$\D_{\X^\#}^\dag(\hdag D) _{ \Q}$-module which is a locally projective
$\O_{\X}(\hdag D)_{ \Q}$-module of finite type.
Suppose that $\E$ satisfies the conditions (a) and (b') (see \ref{Nobuo}),
then the canonical morphism
$\rho$ : $u _{D+} (\E) \rightarrow \E (\hdag Z)$ (see \ref{defi-rho}) is an isomorphism.
\end{prop}

\begin{proof}
By \ref{exacttriu+F}, this is equivalent to check
that $\R \underline{\Gamma} ^\dag _Z \circ u _+ (\E)  =0$.
By applying the functor $f _+$ to the localization triangle of $u_{D+} (\E)$ with respect to $Z$ we get :
  \begin{equation}
    \label{f+exacttriu+F}
    f _+ \circ \R \underline{\Gamma} ^\dag _Z \circ u _+ (\E)  \longrightarrow
    f _+ \circ u_+ (\E)
    \overset{f _+ (\rho)}{\longrightarrow} f _+ (\E(\hdag Z))
    \longrightarrow f _+ \circ  \R \underline{\Gamma} ^\dag _Z \circ u _+ (\E)  [1].
  \end{equation}
Following \ref{Nobuobis}, the morphism
$f _+ \circ u_+ (\E)\rightarrow f _+ (\E(\hdag Z))$ is an isomorphism. Then,
by \ref{f+exacttriu+F},
$f _+ \circ \R \underline{\Gamma} ^\dag _Z \circ u _+ (\E)=0$.
Furthermore, since $\R \underline{\Gamma} ^\dag _Z  \riso i _+ \circ i^!$ (by \cite[4.4.5]{Beintro2}), we get
$(f \circ i) _+ \circ i^!\circ u _+ (\E) \riso f _+ \circ  \R \underline{\Gamma} ^\dag _Z  \circ u _+ (\E) =0 $.
Because $f \circ i$ is finite and \'etale, by \ref{h+=0is0} this implies $i^!\circ u _+ (\E) =0$ and
then $\R \underline{\Gamma} ^\dag _Z \circ u _+ (\E) =0 $.
\end{proof}

\begin{rema}
  Even if the assertions look different,
  the proof of
  \ref{conjrhocurve} is the same as that of \cite[2.3.2]{caro_courbe-nouveau}: here the coherent $\D ^\dag _{\X,\Q}$-module
  is $u _+ (\E)$ and we have replaced the finiteness theorem
  of rigid cohomology (this requires the properness of $\X$ and a Frobenius structure) by \ref{Nobuobis}.
\end{rema}

The following theorem extends \cite[2.3]{caro_courbe-nouveau}
(e.g., notice that here $\X$ does not need to be proper).

\begin{theo}
  \label{233}
Let $\E \in F\text{-}D ^\mathrm{b} _\mathrm{coh} (\D ^\dag _{\X ,\Q})$. The following assertions are equivalent:
\begin{enumerate}
  \item \label{233-i} For any closed point $x$ of $X$, for any lifting $i _x$ of the canonical closed immersion induced by $x$,
  the cohomological spaces of $i^! _x (\E)$ have finite dimension as $K$-vector spaces.
  \item For any divisor $T$ of $X$, the complex $\E (\hdag T)$ belongs to $F\text{-}D ^\mathrm{b} _\mathrm{coh} (\D ^\dag _{\X ,\Q})$.
  \item The complex $\E$ is holonomic.
  \item \label{233-iv} The complex $\E$ is smoothly devissable in partially overconvergent $F$-isocrystals.
    \item \label{233-v} The complex $\E$ is $\D ^\dag _{\X ,\Q}$-overcoherent.
  \item \label{233-vi} The complex $\E$ is overholonomic.
\end{enumerate}
\end{theo}
\begin{proof}
To check the equivalence between the three first assertions,
we have only to rewrite the proof of \cite[2.3.3]{caro_courbe-nouveau} where we replace
\cite[2.3.2]{caro_courbe-nouveau} by \ref{holo-qunip0}.

Proof of $\ref{233-i}\Rightarrow \ref{233-iv}$: suppose $\E$ satisfies \ref{233-i}.
By \cite[2.2.17]{caro_courbe-nouveau},
there exists a divisor $Z$ of $X$ such that $\E (\hdag Z)$ is an isocrystal on $X\setminus Z$
overconvergent along $Z$.
Let $i \, :\, \ZZ \hookrightarrow \X$ be a lifting of the $Z \subset X$. Then, by hypothesis,
$i^! (\E)$ is $\O _{\ZZ,\Q}$-coherent.
Hence $\E$ is smoothly devissable in partially overconvergent $F$-isocrystals.
The implication $\ref{233-iv}\Rightarrow \ref{233-vi}$ is a consequence of \ref{holo-qunip0}.
Finally, $\ref{233-vi}\Rightarrow \ref{233-v}\Rightarrow \ref{233-i}$ are obvious.
\end{proof}

For curves the following statement answers positively to
Berthelot's conjecture of \cite[5.3.6.D]{Beintro2} in the case of curves:

\begin{theo}
\label{theoconjD}
  Let $\E \in F\text{-}D ^\mathrm{b} _\mathrm{coh} (\D ^\dag _{\X } (\hdag Z) _{\Q})$ whose restriction on $\Y$
  is a holonomic $F$-$\D ^\dag _{\Y ,\Q}$-module. Then $\E$ is
a holonomic $F$-$\D ^\dag _{\X ,\Q}$-module.
\end{theo}

\begin{proof}
Replacing \cite[4.3.4]{caro_courbe-nouveau} by \ref{holo-qunip0} and
\cite[2.3.3]{caro_courbe-nouveau} by \ref{233},
it is sufficient to rewrite the proof of \cite[4.3.5]{caro_courbe-nouveau}.
\end{proof}

\begin{rema}
\label{finalremark}
  This Berthelot's conjecture above (of \cite[5.3.6.D]{Beintro2})
  leads to Berthelot's conjecture on the stability of the holonomicity under inverse image.
  This latter conjecture, following \cite{caro_surholonome}, implies that holonomicity equals overholonomicity.
\end{rema}

\providecommand{\bysame}{\leavevmode ---\ }
\providecommand{\og}{``}
\providecommand{\fg}{''}
\providecommand{\smfandname}{et}
\providecommand{\smfedsname}{\'eds.}
\providecommand{\smfedname}{\'ed.}
\providecommand{\smfmastersthesisname}{M\'emoire}
\providecommand{\smfphdthesisname}{Th\`ese}

\bigskip
\noindent Daniel Caro\\
Arithm\'etique et G\'eom\'etrie alg\'ebrique\\
B\^atiment 425\\
Universit\'e Paris-Sud\\
91405 Orsay Cedex\\
France.\\
E-mail : daniel.caro@math.u-psud.fr

\bigskip
\noindent Nobuo Tsuzuki\\
          Department of Mathematics\\
          Graduate School of Science\\
          Hiroshima University\\
          Higashi-Hiroshima 739-8526\\
          Japan.\\
E-mail : tsuzuki@math.sci.hiroshima-u.ac.jp

\end{document}